\documentclass[11pt]{amsart}

\usepackage{amsmath, amssymb, amsthm, cite, etoolbox, float , hyperref, tikz, ytableau}
\usetikzlibrary{matrix}
\title[Polysymmetric Expansions in $h^\otimes$ and $e^\otimes$ and plethysm]{The $h$-expansion of the Plethysm $h_n[p_r]$ and Polysymmetric Expansions in $h^\otimes$ and $e^\otimes$}
\author{Aditya Khanna}
\date{}

\theoremstyle{plain}
\newtheorem{theorem}{Theorem}
\newtheorem{lemma}[theorem]{Lemma}
\newtheorem{prop}[theorem]{Proposition}
\newtheorem{corollary}[theorem]{Corollary}

\newtheorem{openproblem}[theorem]{Question}
\theoremstyle{definition}

\newtheorem{example}[theorem]{Example}
\newtheorem{remark}[theorem]{Remark}

\newcommand{\boks}[1]{\ytableausetup{boxsize = #1 cm}}
\newcommand{\y}[1]{\ydiagram{#1}}
\newcommand{\yt}[1]{\ytableaushort{#1}}
\newcommand{\midyt}{\ytableausetup{aligntableaux = center}}

\newcommand{\yg}[2]{\ydiagram{#1}*[*(lightgray)]{#2}}

\newcommand{\ZZ}{\mathbb{Z}}
\newcommand{\QQ}{\mathbb{Q}}

\renewcommand{\SS}{\mathfrak{S}}

\newcommand{\al}{\alpha}
\newcommand{\be}{\beta}
\newcommand{\ga}{\gamma}
\newcommand{\la}{\lambda}
\newcommand{\de}{\delta}
\newcommand{\vn}{\varnothing}
\newcommand{\ox}{\otimes}
\newcommand{\si}{\sigma}

\newcommand{\x}{\textbf{x}}
\newcommand{\inv}{\mathtt{inv}}

\newcommand{\mcL}{\mathcal{L}}

\DeclareMathOperator{\mcM}{\mathcal{M}}

\DeclareMathOperator{\Sym}{Sym}
\DeclareMathOperator{\sym}{Sym}
\DeclareMathOperator{\PSym}{PSym}
\DeclareMathOperator{\psym}{PSym}
\DeclareMathOperator{\Typ}{Typ}
\DeclareMathOperator{\Par}{Par}
\DeclareMathOperator{\XCom}{OTyp} %idk
\DeclareMathOperator{\xcom}{OTyp} %idk
\DeclareMathOperator{\Com}{Com}
\DeclareMathOperator{\WCom}{WCom}

\DeclareMathOperator{\ab}{Abc}
\DeclareMathOperator{\abc}{\ab}
\DeclareMathOperator{\wt}{wt}
\DeclareMathOperator{\bdg}{bdg}
\DeclareMathOperator{\dg}{dg}
\DeclareMathOperator{\sgn}{sgn}
\DeclareMathOperator{\psgn}{psgn}
\DeclareMathOperator{\sort}{sort}
\DeclareMathOperator{\psort}{psort}
\DeclareMathOperator{\area}{area}
\DeclareMathOperator{\height}{ht}

\DeclareMathOperator{\ocon}{ocon} %ordered content
\DeclareMathOperator{\pocon}{pocon} %ordered content

\DeclareMathOperator{\rot}{rot}  %rotation
\DeclareMathOperator{\comrev}{comrev}

\DeclareMathOperator{\wrap}{wrap}
\DeclareMathOperator{\pwrap}{pwrap}

\let\bbmatrix\bordermatrix
\patchcmd{\bbmatrix}{8.75}{8.75}{}{}
\patchcmd{\bbmatrix}{\left(}{\left[}{}{}
\patchcmd{\bbmatrix}{\right)}{\right]}{}{}

\begin{document}
\begin{abstract}
The algebra of polysymmetric functions ($\psym$) is defined as the tensor product of copies of the algebra of symmetric functions ($\sym$) where the $i$th copy has variables scaled by $i$. One way to construct a basis of $\psym$ is to start with a basis $\{f_\la\}$ of $\sym$ and consider all pure tensors arising from this basis. Asvin G and Andrew O'Desky described four families of non-pure tensor bases, namely $H, E, E^+$ and $P$, which we call plethystic bases. In this paper, we study the expansions of the plethystic bases into the pure-tensor bases $\{h^\otimes_\tau\}$ and $\{e^\ox_\tau\}$, and interpret the results via combinatorial objects called polywrapping block tabloids. The expansions of $H, E$ and $E^+$ in $h^\otimes$ and $e^\otimes$ are found via the $h$-expansion of the plethysm $h_n[p_r]$. We compute this $h$-expansion using a statistic on words and prove it through abacus methods.
\end{abstract}
\maketitle
\section{Introduction}

The combinatorics of the $\QQ$-algebra of symmetric functions, Sym, has been studied extensively. The underlying $\QQ$-vector space of Sym comes equipped with some notable bases namely the $h$-basis, the $e$-basis, the $p$-basis, the $m$-basis and the $s$-basis. We define these bases in Section \ref{ssec:intro1}. The coefficients obtained when expressing one basis into another can be computed via combinatorial methods involving objects called tableaux and tabloids. We visualize these combinatorial objects as (possibly labeled) tilings of certain arrays of boxes. The tableau and tabloid descriptions of the change-of-basis coefficients have been explored in the works of Remmel and E\u{g}ecio\u{g}lu \cite{eg-rem, inv-kostka} among others. The interpretations involving tabloids and tableaux have also been useful in computing the change-of-basis coefficients for generalizations of symmetric functions such as quasisymmetric functions \cite{qsym-book} and noncommutative symmetric functions \cite{noncomm}. For a historical overview of tabloid approaches in Sym, see \cite[Sec. 2.4.7]{khanna-thesis}.

In this paper, we present combinatorial descriptions for a new generalization of Sym called the algebra of polysymmetric functions (PSym) by Asvin G and Andrew O'Desky \cite{psym}. An element of PSym is a countable tensor product of symmetric functions with finitely many factors having non-zero degree such that the symmetric function in the $i$th tensor factor has the degrees of its variables scaled by $i$. In \cite{psym}, the authors describe pure-tensor bases arising from the aforementioned notable bases of Sym. We call these the $h^\ox$-basis, the $e^\ox$-basis, the $p^\ox$-basis, the $m^\ox$-basis, and the $s^\ox$-basis. They also describe four bases which do not arise via straightforward tensor products, namely the $H$-basis, the $E$-basis, the $E^+$-basis, and the $P$-basis. We call these the \textit{plethystic bases} and define them in Section \ref{ssec:psym}.

In the paper that introduces polysymmetric functions \cite{psym}, G and O'Desky provide an interpretation for the $m^\ox$-expansion of the $H$-basis using certain tilings they call arrangements. In later work, Khanna and Loehr \cite{KLpsym} provide combinatorial interpretations of the coefficients that appear when expanding the four plethystic bases into the $m^\ox$-basis, the $s^\ox$-basis, and the $p^\ox$-basis. The combinatorial objects that appear in their work are formal tensor products of tabloids. They use inductive methods (polysymmetric analogs of the Pieri rules) to construct these objects and compute the change-of-basis coefficients. Further work by Khanna~\cite{khanna-bij} computes the change-of-basis coefficients between all plethystic bases via bijective (sign-reversing involutions) and inductive arguments. An overview of the results from  \cite{khanna-bij} and \cite{KLpsym} can be found in \cite{khanna-thesis}. A preprint by Martinez \cite{dm-psym} finds the change-of-basis coefficients between plethystic bases via generating function techniques.

In this paper, we combinatorially compute and prove the change-of-basis coefficients obtained while expressing the plethystic bases into the $h^\ox$-basis and the $e^\ox$-basis.
\subsection{Outline} We now briefly mention some of the results in the paper and reference the corresponding sections for ease of navigation. 
 \begin{itemize}
 \item In Section \ref{sec:bg}, we recall the concepts of compositions, partitions, symmetric functions, and plethysm. We also discuss the concepts of ordered types, types, and polysymmetric functions.
\item In Section \ref{sec:P-he-wrapping}, the $h^\ox$-basis and the $e^\ox$-basis expansions of the $P$-basis are described using constant wrapping block tabloids.
\end{itemize}  
For other plethystic bases, we first find the $h$-expansion of the plethysm $h_n[p_r]$. For a list of $r$ non-negative integers, compute $\hat{s}_r(w)$ as follows: point-wise add $(0,1,\ldots, r-1)$ to $w$, then find the entries modulo $r$. If the list thus obtained has any repeated entries, set $\hat{s}_r(w) = 0$. Otherwise, count the number of pairs of elements in this list in which the larger element is to the left of the smaller element.  Call this count $i$ and define $\hat{s}_r(w) = (-1)^i$. For example, if $w = (2,0,1)$, then $w + (0,1,2) = (2,1,3)$. Reducing the entries modulo 3, we find the list $(2,1,0)$. As 2 appears to the left of 1 and 0, and 1 appears to the left of 0, we find $i = 2 + 1 = 3$. So, $\hat{s}_3((2,0,1))  =-1$. Let $n,r$ be positive integers. For a partition $\mu$ of $nr$, define $\hat{s}_r(\mu)$ to be $\sum_{w} \hat{s}_r(w)$ where the sum is over lists $w$ of $r$ non-negative integers for which the multiset of positive parts is the multiset of parts of $\mu$.

 We prove (Theorem \ref{thm:h-exp-of-hnpr}) that
\[
h_n[p_r] = \sum_{\mu} \hat{s}_r(\mu) h_\mu,
\]
where the sum is over all partitions $\mu$ of $nr$.
\begin{itemize}
\item In Section \ref{ssec:Hh-Ee}, the $h^\ox$-basis expansion of the $H$-basis and the $e^\ox$-basis expansion of the $E$-basis are described using polywrapping block tabloids.
\item In Section \ref{ssec:He-Eh}, the $h^\ox$-basis expansion of the $E$-basis and the $e^\ox$-basis expansion of the $H$-basis are described using labeled polywrapping block tabloids.
\item In Section \ref{ssec:Ue-Uh}, the $h^\ox$-basis and the $e^\ox$-basis expansions of the $E^+$-basis are described using labeled and unlabled polywrapping block tabloids respctively.
\item In Section \ref{sec:misc-open}, we expand upon some concepts discussed in the paper, and we analyze some properties of the $\hat{s}_r$ function as a statistic over the set of lists of non-negative integers. We also pose some open problems and future directions.
\end{itemize}

\section{AI and Technology Statement}
No AI or Large Language Models of any kind were used in the conception, ideation, writing, code writing, or editing of this project. All the thought, proof-writing, and editing was done solely by the author of this paper. The only use of an LLM (ChatGPT 5.6 Luna) was in finding the term ``rectangular content'' (used in Lemma \ref{lem:rectangular-content}) while the citations in the Lemma were found via Google Scholar.

The examples in this document were computed by hand or via SageMath using custom code written by the author. The diagrams in this document were made in TikZ using custom Python code written by the author. All or any code can be made available upon request.
\section{Background}\label{sec:bg}
\subsection{Compositions, Partitions, and Symmetric Functions}\label{ssec:intro1}
\midyt
 We first recall the notions of compositions, partitions, symmetric functions, and plethysm which are prerequisite in understanding polysymmetric functions. Detailed treatments of the topics in this subsection can be found in \cite{macd, pleth-expose, loehr-book}.

Let $n$ be a non-negative integer. A \textit{composition} $\be$ of $n$ is a list of positive integers $\be = (\be_1,\be_2, \ldots, \be_l)$ such that $\be_1 + \be_2 + \ldots + \be_l = n$. We call $n$ the \textit{size} of $\be$ and denote it by $|\be|$.  We also use the notation $\area(\be)$ for the size of $\beta$. We call $l$ the \textit{length} of $\be$ and denote it by $\ell(\be)$. The \textit{diagram} $\dg(\be)$ of $\be$ is constructed by placing boxes in $l$ left-justified rows such that the $i$th row from top contains $\be_i$ boxes.
\boks{0.2}
\begin{example}
The list $\be = (3,4,3,1)$ is a composition of size $|\be| = 3 + 4 + 3 + 1 = 11$ and length $\ell(\be) = 4$. The diagram of $\be$ is $\dg(\be) = \y{3,4,3,1}$.
\end{example}
For a non-negative integer $n$, denote the set of compositions of $n$ by $\Com(n)$ where $\Com(0)$ is the singleton containing the empty composition $(0) = \varnothing$. Denote the set of all compositions by $\Com$.

A composition $\la = (\la_1, \la_2, \ldots, \la_l)$ with weakly decreasing entries $\la_1 \geq \la_2 \geq \ldots \geq \la_l$ is called a \textit{partition}. We call $\la_i$ for $i = 1,2,\ldots, l$ a \textit{part} of $\la$. For instance, $(4,3,3,1)$ is a partition of $11$ with parts 4, 3, 3, and 1. We denote the set of partitions of $n$ by $\Par(n)$ and the set of all partitions by $\Par$. The map $\sort:\Com\to \Par$ sends each composition $\beta$ to the unique partition $\sort(\be)$ formed by rearranging the entries of $\be$ in weakly decreasing order.

 A formal power series $f\in\QQ[[x_1, x_2, \ldots]]$ of bounded degree is called a \textit{symmetric function} if $f$ is preserved under any permutation of the indeterminates $x_1, x_2, \ldots$. We denote the set of all homogeneous symmetric functions of degree $n$ by $\Sym(n)$. The set of all symmetric functions is denoted by $\Sym$. $\sym$ is a graded $\QQ$-algebra where the degree $n$ component is the $\QQ$-vector space $\sym(n)$~\cite[I.2]{macd}.
\begin{example}
The formal power series $x_1^4 + x_2^4 + x_3^4 + \ldots + x_1^2 x_2 + x_2^2 x_1 + x_2^2 x_3 + x_3^2 x_2 + x_1^2 x_3 + x_3^2 x_1 + \ldots $ is a symmetric function in $\sym(4) \oplus \sym(3)$.
\end{example}
The partitions of $n$ index the basis elements of $\sym(n)$ \cite[Thm. 9.23]{loehr-book}. So the combinatorics of partitions is intimately tied to the algebra of symmetric functions. For any partition $\la = (\la_1,\la_2, \ldots, \la_l)$, define $\x^\la = x_1^{\la_1}x_2^{\la_2}\ldots x_l^{\la_l}$. Define $m_\la$ to be the sum of all distinct monomials obtained by permuting the indices of $\x^\la$ in all possible ways.
\begin{example}
We compute $m_{(2,2,1)} = x_1^2 x_2^2 x_3 + x_1 x_2^2 x_3^2 + x_1^2 x_2 x_3^2 + \ldots$.
\end{example}
Let $n$ be a positive integer. Define the \textit{$n$th complete homogeneous symmetric function} $h_n$ to be the sum of all monomials in $x_1, x_2,\ldots$ of degree $n$. We call a monomial \textit{square-free} if each variable appears exactly once as a factor of the monomial. Define the \textit{$n$th elementary symmetric function} $e_n$ to be the sum of all square-free monomials in $x_1, x_2,\ldots$ of degree $n$. Set $h_0= e_0 = 1$. Define the \textit{$n$th power-sum symmetric function} $p_n$ as the sum $x_1^n + x_2^n + \ldots$. The sets $\{h_n\}_{n\geq 0}$, $\{e_n\}_{n\geq 0}$, and $\{p_n\}_{n\geq 1}$ each generate $\sym$ as a $\QQ$-algebra~\cite[I.2 (2.8)]{macd}.

\begin{example}
We compute the following symmetric functions:
\begin{align*}
h_3 = & x_1^3 + x_2^3 + x_3^3 + \cdots +  x_1 x_2 x_3 + \cdots + \\& x_1^2 x_2 + x_1x_2^2 + x_1^2 x_3 + x_1x_3^2 + x_2^2x_3 + x_2x_3^2 + \cdots\\
e_3 = & x_1 x_2 x_3 + x_1 x_3 x_4 + x_1 x_2 x_4 + \cdots\\
p_3 = &x_1^3 + x_2^3 + x_3^3 + \cdots
\end{align*}
\end{example}

For a composition $\beta = (\beta_1, \beta_2, \ldots, \be_l)$, we define  $h_\be = h_{\be_1}h_{\be_2}\ldots h_{\be_l}$, $e_\be = e_{\be_1}e_{\be_2}\ldots e_{\be_l}$ and $p_\be = p_{\be_1}p_{\be_2}\ldots p_{\be_l}$. As $\sym$ is a commutative $\QQ$-algebra, the equalities $h_\beta = h_{\sort(\beta)}$, $e_\beta = e_{\sort(\beta)}$, and $p_\beta = p_{\sort(\beta)}$ hold. The sets $\{h_\la\}_{\la\in \Par(n)}$, $\{e_\la\}_{\la\in \Par(n)}$, and $\{p_\la\}_{\la\in \Par(n)}$ each form a $\QQ$-linear basis of $\sym(n)$~\cite[Thm. 9.66, 9.71, 9.79]{loehr-book}.

For a partition $\la = (\la_1, \la_2, \ldots, \la_l)$, define the \textit{Schur function} $s_\la$ using the Jacobi-Trudi identity \cite[10.60]{loehr-book}:
\[
s_\la = \det\begin{bmatrix}
h_{\la_1} & h_{\la_1 + 1} & \ldots & h_{\la_1 + l - 1}\\
h_{\la_2 - 1} & h_{\la_2} & \ldots & h_{\la_l + l - 2}\\
\vdots & \vdots & \ddots & \vdots\\
h_{\la_l - l + 1} & h_{\la_l - l + 2}& \ldots & h_{\la_l}
\end{bmatrix}
\]

We discuss a useful combinatorial interpretation of the Jacobi-Trudi identity by E\u{g}ecio\u{g}lu and Remmel \cite{inv-kostka} in Section \ref{ssec:h-exp-hnpr}.

An alternate description of Schur functions is obtained through the bialternant formula.
For a list $w$ of $l$ non-negative integers, define the $l\times l$ matrix $A_w$ such that the entry in row $i$ and column $j$ is $x_i^{w_j + l - j}$. Denote by $\vec{0}_l$, the list containing $l$ zeroes. The Jacobi's bilaternant formula \cite[I.3 (3.1)]{macd} states that for a partition $\la$, $s_\la = \frac{\det(A_\la)}{\det(A_{\vec{0}_l})}$, where we append $l - \ell(\la)$ trailing zeros to $\la$ if $\ell(\la) < l$.
\begin{remark}
The Schur function $s_\la$ is also defined as the generating function of the content statistic over semi-standard Young tableau of shape $\la$. We omit the discussion of this interpretation in this paper as the definition via the Jacobi-Trudi identity and the bilalternant formulation are of greater use to us.
\end{remark}

As $\Sym$ is a $\QQ$-algebra; sums, products, and rational multiples of symmetric functions are still symmetric functions. We now recall the operation of plethysm which may be informally understood as the extension of the notion of function composition to symmetric functions. The following paragraph summarizes the definitions and properties from \cite{pleth-expose} that we need. 

For an arbitrary symmetric function $f$ and a formal power series $g$, we call $f[g]$ the \textit{plethysm of $f$ with $g$}. We call $f$ the \textit{outer function} and $g$ the \textit{inner function}. Let $c$ be a rational number. For power-sum symmetric functions $p_m$ and $p_n$, $p_m[p_n] = p_{mn}$. For any arbitrary formal power series $g_1$ and $g_2$, $p_m[g_1 + g_2] = p_m[g_1] + p_m[g_2]$, $p_m[g_1g_2] = p_m[g_1]p_m[g_2]$ and $p_m[cg_1] =c p_m[g_1]$. For symmetric functions $f_1 $ and $f_2$ and an arbitrary formal power series $g$, $(f_1 + f_2)[g] = f_1[g] + f_2[g]$, $(f_1f_2)[g] = f_1[g]f_2[g]$ and $(cf_1)[g] = c\cdot f_1[g]$. We remark that $p_n[f] = f[p_n]$ for all $f\in \Sym$ and positive integers $n$~\cite[Thm. 3]{pleth-expose}. Also, $p_1[f] = f = f[p_1]$ \cite[Thm. 4]{pleth-expose}.

To compute a plethysm algebraically, we require  the following proposition.
\begin{prop}[I.8 (8.4) in \cite{macd}]\label{prop:monomial-sub}
 For all $f\in \Sym$ and $n\geq 1$, \[f[p_n](x_1,x_2, \ldots) = f(x_1^n, x_2^n,\ldots).\]
\end{prop}
\begin{example}
We find
$h_3[p_4] = x_1^{12} + x_2^{12} + x_3^{12} + \cdots +  x_1^4 x_2^4 x_3^4 + \cdots + x_1^8 x_2^4 + x_1^4x_2^8 + x_1^8 x_3^4 + x_1^4x_3^8 + x_2^8x_3^4 + x_2^4x_3^8 + \cdots$
\end{example}
\subsection{Polysymmetric functions}\label{ssec:psym}
We now discuss the concepts of ordered types, splitting types, and polysymmetric functions. Polysymmetric functions were introduced by Asvin G and Andrew O'Desky in 2022 \cite{psym}. Detailed treatments of the topics discussed in this subsection can be found in \cite{psym, KLpsym,khanna-bij,khanna-thesis}

Define an \textit{ordered type} $\delta$ of $n$ as a list of compositions $(\be^{(1)},\be^{(2)},\ldots)$ such that $\sum_{i\geq 1} i \area(\be^{(i)}) = n$. We call $n$ the \textit{size} of $\delta$ and denote it by $|\delta|$. The \textit{length} $\ell(\delta)$ of $\delta$ is defined as $\sum_{i\geq 1} \ell(\be^{(i)})$.  The set of ordered types of $n$ is denoted by $\xcom(n)$ and the set of all ordered types is denoted by $\xcom$.

We write $\de|_k = \be^{(k)}$ and formally write $\delta = 1^{\delta|_1} 2^{\de|_2}\ldots$. We may choose to drop the parentheses or commas in the exponent $\delta|_i$ for clarity.
The positive integers $1,2, \ldots$  in the base of $\de$ are called \textit{degrees} while the entries of $\delta|_i$ in the exponents of $\de$ are called the \textit{associated multiplicities of $i$} for $i =1,2,\ldots$. For example, the object $\de = 1^{2,4,2} 2^{2,1} 4^{2,1,3,1}$ is an ordered type of $42$ with degrees 1, 2, and 4. The associated multiplicities for 1 are 2, 4, and 2, the associated multiplicities for 2 are 2 and 1, and the associated multiplicities for 4 are 2, 1, 3, and 1. 

If $\delta$ is such that $\delta|_d = (m)$ and $\delta|_i = \varnothing$ for all $i\neq d$, then we write $\delta = d^m$. Such a $\de$ is called a \textit{block} with degree $d$ and associated multiplicity $m$. We may view an ordered type $\delta$ as a list of $\de|_i$ for $i\geq 1$ where each $\delta|_i$ is an ordered collection of blocks with degree $i$. We list out the blocks of $\delta$ and call this list (without parentheses or commas) the \textit{block notation} for $\de$. In the preceding example, we may write $\delta$ in block notation as $1^2 1^4 1^2 2^2 2^1 4^2 4^1 4^3 4^1$.
\begin{remark}
Related objects known as polycompositions are the focal combinatorial object in the paper \cite{khanna-bij} that deals with transition matrices between plethystic bases. A polycomposition is a list of ordered collections of blocks with the same multiplicity rather than the same degree as in ordered types. Thus, in a sense, polycompositions are dual objects to ordered types. 
\end{remark}

The \textit{block diagram} $\bdg(\de)$ of an ordered type $\de$ is an ordered list of rectangular arrays of unit boxes, where for each appearance of $d^r$ in the block notation of $\de$, we place a rectangular array consisting of $d$ rows and $r$ columns.
\begin{example}
For $\delta = 1^{2,1,2} 2^{3} 3^2 4^{1,2}$,
\[
\bdg(\de) = \y{2} \: \y{1}\: \y{2} \: \y{3,3} \:
\y{2,2,2} \: \y{1,1,1,1} \: \y{2,2,2,2}.
\]
\end{example}

% We visualize the diagram $\dg(\delta)$ of $\delta$ as a formal tensor product of diagrams of $\delta|_i$.
%\begin{example}Let $\de = 1^{2,3,2} 2^{2,1} 4^{2,1,3,1}$ is an element of $\xcom(42)$ with $\ell(\de) = 9$. We visualize it as\[\dg(\de) = \y{2,3,2}\ox \y{2,1}\ox \vn \ox \y{2,1,3,1}.\]\end{example}

Let $\delta\in \xcom(n)$. If all $\delta|_i$ for $i\geq 1$ are partitions, then we call $\delta$ a \textit{splitting type}, or simply a \textit{type} of $n$. The set of types of $n$ is denoted by $\Typ(n)$ and the set of all types is denoted by $\Typ$. For example, $1^{4,2,2}2^{2,1} 4^{3,2,1,1}$ is a type of 42. Define $\psort:\xcom(n) \to \Typ(n)$ to be the map that sends an ordered type to the type obtained by replacing each composition $\delta|_i$ with $\sort(\delta|_i)$. For instance, $\psort(1^{2,4,2} 2^{2,1} 4^{2,1,3,1}) = 1^{4,2,2}2^{2,1}4^{3,2,1,1}$.

With the perspective of block notation, a type is to be viewed as a multiset of blocks. We often implicitly translate between block notation for types and ordered types, and the formal notation with compositions or partitions in exponents.

%--- 6 Aug 2026

In order to define polysymmetric functions, we first introduce the set of doubly-indexed indeterminates $\{x_{i,j}\}_{i,j\geq 1}$ where each indeterminate $x_{i,j}$ has degree $i$. We often drop the comma between the indices and write $x_{i,j}$ as $x_{ij}$. The $\QQ$-algebra of formal power series in these indeterminates is denoted by $\QQ[[\x_{**}]]$. For a monomial $f\in \QQ[[\x_{**}]]$, the degree of $f$ is the total degree of indeterminates that appear as factors (with multiplicity). For example, $x_{11}^3 x_{12}^2 x_{23}^4 x_{41}^3$ has degree $1\cdot 3 + 1\cdot 2 + 2\cdot 4 + 4\cdot 3= 25$.

A formal power series $f\in \QQ[[\x_{**}]]$ is called a \textit{polysymmetric} function if it is preserved under any swapping of pairs of indeterminates $x_{ij}$ and $x_{ik}$ for all $i,j,k\geq 1$. We denote the $\QQ$-vector space of homogeneous polysymmetric functions of degree $n$ by $\psym(n)$ and the graded $\QQ$-algebra of all polysymmetric functions is denoted by $\psym$. The $n$th degree component of $\psym$ is $\psym(n)$.

\begin{example}
The formal power series
\[
F = \sum_{\substack{j>i \geq 1\\k\geq 1}} x_{1i}x_{1j} x_{2k} + \sum_{j > i \geq 1} x_{2i}x_{2j}^3 + \sum_{i\geq 1}x_{3i}^2
\]
is a polysymmetric function in $\PSym(4)\oplus \PSym(8) \oplus \PSym(6)$.
\end{example}

Equivalently, one may view the algebra of polysymmetric functions as a tensor product of copies of the algebra of symmetric functions where the variables in the $i$th copy have their degrees scaled by $i$. If $\sym^{(i)}$ is the algebra of symmetric functions where $x_j$ is replaced by $x_{ij}$, then $\psym\cong \bigotimes_{i\geq 1} \Sym^{(i)}$.

 Let $g\in \sym$. We use the shorthand $g(\x_*)$ for $g(x_1, x_2, \ldots)$. We embed $g$ in $\psym(d)$ by writing $g(\x_{d*})$ for $g(x_{d1}, x_{d2},\ldots)$. For $G\in \psym$, we write $G(\x_{**})$ to emphasize that the variable set is $\{x_{i,j}\}_{i,j\geq 1}$. Suppose for a particular polysymmetric function $f(\x_{**})$, there exist symmetric functions $f_i$ for $i\geq 1$ such that $f(\x_{**}) = f_1(\x_{1*})f_2(\x_{2*})\cdots$. Then $f = f_1\otimes f_2 \otimes \cdots$ is an element of $\psym\cong \bigoplus_{i\geq 1} \Sym^{(i)}$. In this paper, we use these two notations interchangeably.

The types of $n$ index the basis elements of $\psym(n)$~\cite[Thm. 3.1]{psym}. First, we discuss how to construct a {pure-tensor basis} from a symmetric function basis. Let $\{h_\la\}_{\la\in \Par}$ be the complete homogeneous symmetric function basis of $\sym$. For an ordered type $\de$, define
  \[
  h_\de^\otimes = h_{\de|_1}\otimes h_{\de|_2}\otimes \ldots.
  \] 
  \begin{example}
  For $\de = 1^{2,4,2}2^{2,1}4^{2,1,3,1}$, 
  \[
  h_\de = h_{2,4,2}\ox h_{2,1}\ox 1 \ox h_{2,1,3,1}.
  \]
  \end{example}
Similarly, we define $p^\otimes_\de$ and $e^\ox_\de$. For a type $\tau$, we define $h^\otimes_\tau$ as
\[
  h_\tau^\otimes = h_{\tau|_1}\otimes h_{\tau|_2}\otimes \ldots
  \] 
  Similarly define $e^\otimes_\tau$, $p^\ox_\tau$, $m^\ox_\tau$ and $s^\ox_\tau$. Each of the sets $\{h^\ox_\tau\}_{\tau\in \Typ(n)}$, $\{e^\ox_\tau\}_{\tau\in \Typ(n)}$, $\{p^\ox_\tau\}_{\tau\in \Typ(n)}$, $\{m^\ox_\tau\}_{\tau\in \Typ(n)}$, and $\{s^\ox_\tau\}_{\tau\in \Typ(n)}$ generate $\PSym(n)$ as a $\QQ$-vector space.
  
  \begin{remark}
Refer to Rem. 1.5 in \cite{KLpsym} for discussion of the representation theoretic connection between polysymmetric functions and the algebra of uniform block permutations which was defined by Orellana et al. \cite{ubp}.
  \end{remark}

Asvin G and Andrew O'Desky \cite[Sec. 3]{psym} define four families of bases which are not pure-tensor bases. We use the term monomial to mean a monomial in $\QQ[[\x_{**}]]$ and a monomial is called \textit{square-free} if each indeterminate factor appears exactly once.

Let $d$ be a positive integer. Define $H_d$ to be the sum of all degree $d$ monomials. Define $E^+_d$ to be the sum of all degree $d$ square-free monomials. Let $\text{len}(f)$ be the number of variables in a square-free monomial $f$. Define $E_d$ to be the sum of all degree $d$ square-free monomials such that the coefficient of the monomial $f$ is $(-1)^{\text{len}(f)}$. Define $P_d = \sum\limits_{k\mid d} k  \sum_{j\geq 1}x_{kj}^{d/k}$ where the sum is over divisors $k$ of $d$. We also define $H_0 = E_0 = E^+_0 = P_0 = 1$. 

\begin{example}
Let $d = 3$. The monomial expansions of the above polysymmetric functions are:
\begin{align*}
H_3 &= x_{31} + x_{21}x_{11} + x_{11}x_{12}x_{13} + x_{11}^2 x_{12} + x_{11}^3 + \cdots\\
E_3 &= -x_{31} + x_{21}x_{11} - x_{11}x_{12}x_{13} + \ldots\\
E^+_3&= x_{31} + x_{21}x_{11} + x_{11}x_{12}x_{13} + \cdots
\end{align*}
Also for $d = 6$, $P_6 = x_{11}^6 + 2x_{21}^3 + 3x_{31}^2 + 6x_{61} + \cdots$.
\end{example}

For $F\in \QQ[[\x_{**}]]$, the notation $F(\x_{**}^r)$ denotes the power series obtained by replacing each indeterminate $x_{ij}$ with $x_{ij}^r$.

For a block $d^r$, define $H_{d^r} = H_d(\x_{**}^r)$. So,
\[
H_{3^4} = x_{31}^4 + x_{21}^4x_{11}^4 + x_{11}^4 x_{12}^4x_{13}^4 + x_{11}^8 x_{12}^4 + x_{11}^{12} + \cdots.
\] For an ordered type $\delta = d^{r_1, r_2,\ldots, r_k}$, define $H_{\delta}(\x_{**}) = H_{d}(\x_{**}^{r_1})H_d(\x_{**}^{r_2})\cdots H_{d}(\x_{**}^{r_k})$. For a general ordered type $\delta$, define $H_\delta = H_{\delta|_1} H_{\delta|_2}\ldots$. For example,
\begin{align*}
H_{1^{2,4,2} 2^{2,1} 4^{2,1,3,1}} &= H_{1^{2,4,2}} H_{2^{2,1}}H_{4^{2,1,3,1}}\\
&=H_1(\x_{**}^2)H_1(\x_{**}^4)H_1(\x_{**}^2)\cdot H_2(\x_{**}^2)H_2(\x_{**})\cdot\\ &\:H_4(\x_{**}^2)H_4(\x_{**})H_4(\x_{**}^3)H_4(\x_{**}).
\end{align*}

Similarly, define $E_\de$, $E^+_\de$, and $P_\de$. 

The sets $\{H_{d^r}\}_{d,r\geq 1}$,  $\{E_{d^r}\}_{d,r\geq 1}$,  $\{E^+_{d^r}\}_{d,r\geq 1}$, and  $\{P_{d^r}\}_{d,r\geq 1}$ each generate $\psym$ as a $\QQ$-algebra \cite[Thm. 3.1, Cor. 3.3]{psym}.

We recall the expansions of $H_d$, $E_d$, $E^+_d$ and $P_d$ in pure-tensor bases as proved in \cite{KLpsym} . For a partition $\la$, define $m_i(\la)$ to be the number of times the part $i$ appears in $\la$.

\begin{prop}[Prop. 2.9, Prop 2.23 in \cite{KLpsym}]\label{prop:d-expansions}
Let $d$ be a non-negative integer. Then
\begin{enumerate}
\item $H_d = \sum\limits_{\la\in \Par(d)} h_{m_1(\la)} \otimes h_{m_2(\la)} \otimes \cdots$.
\item $E_d = \sum\limits_{\la\in \Par(d)} (-1)^{\ell(\la)} e_{m_1(\la)} \otimes e_{m_2(\la)} \otimes \cdots$.
\item $E^+_d = \sum\limits_{\la\in \Par(d)} e_{m_1(\la)} \otimes e_{m_2(\la)} \otimes \cdots$.
\item $P_d = \sum\limits_{k|d} k p^\otimes_{k^{d/k}}$.
\end{enumerate}
\end{prop}

The polysymmetric function $H_{d^r}$ is obtained by raising each variable in $H_d$ to the power $r$ which is reminiscent of performing the operation of plethysm with $p_r$ as in Proposition \ref{prop:monomial-sub}. This observation leads to the following expansions.
\begin{prop}[Prop. 2.9, Prop 2.23 in \cite{KLpsym}]\label{prop:dr-expansions}
Let $d$ be a non-negative integer and $r$ be a positive integer. Then
\begin{enumerate}
\item $H_{d^r} = \sum\limits_{\la\in \Par(d)} h_{m_1(\la)}[p_r] \otimes h_{m_2(\la)}[p_r] \otimes \cdots$.
\item $E_{d^r}= \sum\limits_{\la\in \Par(d)} (-1)^{\ell(\la)} e_{m_1(\la)}[p_r] \otimes e_{m_2(\la)}[p_r] \otimes \cdots$.
\item $E^+_{d^r} = \sum\limits_{\la\in \Par(d)} e_{m_1(\la)}[p_r] \otimes e_{m_2(\la)}[p_r] \otimes \cdots$.
\item $P_{d^r} = \sum\limits_{k|d} k p^\otimes_{k^{dr/k}}$.
\end{enumerate}
\end{prop}

\subsection{Algebra involutions on Sym and PSym}
Recall that there exists a unique algebra involution $\omega$ on $\sym$ which for any composition $\beta$ maps $h_\be$ to $e_\be$ and $e_\be$ to $h_\be$ \cite[I.2 (2.7)]{macd}. Also, $\omega(p_\la) = (-1)^{|\la|-\ell(\la)} p_\la$~\cite[I.2 (2.13)]{macd}. Let $I$ be some multiset consisting of partitions of $n$. If for some $f\in \sym(n)$, $f = \sum_{\al\in I} c_\al h_{\al}$, then $\omega(f) = \sum_{\al\in I} c_\al e_{\al}$. Informally, if we can express $f$ as a linear combination of elements of the $h$-basis, then we can immediately find $\omega(f)$ as a linear combination of the $e$-basis. In this paper, we make signifcant use of this property. The involution $\omega$ also interacts nicely with the plethysm operation.
\begin{prop}[I.8 Ex. 1(a) in \cite{macd}]
Let $f$ and $g$ be homogeneous symmetric functions. Then
\[
\omega(f[g]) = \begin{cases}
f[\omega(g)] & \text{if } \deg(g) \text{ is even}\\
\omega(f)[\omega(g)] & \text{if } \deg(g) \text{ is odd}.
\end{cases}
\]
\end{prop}

The following proposition shows how $\omega$ can be expressed in terms of plethysm.

\begin{prop}[I.8 Ex. 1(a) in \cite{macd} and Thm. 6 in \cite{pleth-expose}]\label{prop:negation-rule}
For all $f,g\in \sym$, $f[-g]= (-1)^{\deg(f)} (\omega(f))[g]$.
\end{prop}

\begin{remark}
When $g = p_1$, $f[-p_1] = (-1)^{\deg(f)}\omega(f)$ is the antipode map for the Hopf algebra of symmetric functions \cite[Prop. 2.4.1]{hopf}. %Define \[S(f[g]) = f[-g].\]
\end{remark}

In \cite{psym}, the authors define $\Omega$ as an analog of $\omega$ for the algebra of polysymmetric functions. For all ordered types $\delta$, $\Omega$ is an algebra isomorphism such that $\Omega(H_\de) = E_\de$ and $\Omega(E_\de) = H_{\de}$. From \cite[Cor. 16]{khanna-bij}, we know that $\Omega(P_\de) = (-1)^{\ell(\de)} P_\de$. We now study the action of $\Omega$ on pure-tensors.

\begin{prop}\label{prop:Omega-on-plethysm}
If $f_1, f_2, \ldots, g_1, g_2,\ldots$ are symmetric functions, then
\[
\Omega(f_1[g_1]\otimes f_2[g_2]\otimes \cdots)  = f_1[-g_1]\otimes f_2[-g_2]\otimes \cdots.
\]
\end{prop}
\begin{proof}
Define $\Omega':\psym\to \psym$ to be the algebra homomorphism that sends $f_1[g_1]\otimes f_2[g_2]\otimes \cdots$ to $f_1[-g_1]\otimes f_2[-g_2]\otimes \cdots$. We show that $\Omega'$ agrees with $\Omega$ on the algebraic generators of $\psym$ and thus agrees with $\Omega$ on all of $\psym$. From Thm. 3.1 in \cite{psym}, we know that $\{H_{d^r}\}_{d,r\geq 1}$ generates $\psym$ as a $\QQ$-algebra. Choose positive integers $d$ and $r$. Using Proposition \ref{prop:dr-expansions} (1), we can apply $\Omega'$ on $H_{d^r}$ to find
\[
\Omega'(H_{d^r}) = \Omega'(\sum\limits_{\la\in \Par(d)} h_{m_1(\la)}[p_r] \otimes h_{m_2(\la)}[p_r] \otimes \cdots).
\]
Under the action of $\Omega'$, the right hand side can be rewritten as
\[
\Omega'(H_{d^r}) = \sum\limits_{\la\in \Par(d)} h_{m_1(\la)}[-p_r] \otimes h_{m_2(\la)}[-p_r] \otimes \cdots.
\] 
Using the negation rule for plethysm from Proposition \ref{prop:negation-rule} and $\omega(h_n) = e_n$ for all $n\geq 1$, we find that $h_n[-p_r] = (-1)^n e_n[p_r]$ for all $n,r \geq 1$. Now
\begin{align*}
\Omega'(H_{d^r}) &= \sum\limits_{\la\in \Par(d)} (-1)^{m_1(\la)}e_{m_1(\la)}[p_r] \otimes (-1)^{m_2(\la)}e_{m_2(\la)}[p_r] \otimes \cdots\\ &= \sum\limits_{\la\in \Par(d)} (-1)^{m_1(\la)+m_2(\la) + \cdots}e_{m_1(\la)}[p_r] \otimes e_{m_2(\la)}[p_r] \otimes \cdots .
\end{align*}
The sum $m_1(\la) + m_2(\la) + \ldots$ is the sum over multiplicities of all parts of $\la$ and is thus equal to the number of parts $\ell(\la)$ of $\la$. This shows that
\[
\Omega'(H_{d^r}) = \sum\limits_{\la\in \Par(d)} (-1)^{\ell(\la)}e_{m_1(\la)}[p_r] \otimes e_{m_2(\la)}[p_r] \otimes \cdots = E_{d^r},
\]
where the final equality follows from part (2) of Proposition \ref{prop:dr-expansions}. As $d$ and $r$ were chosen arbitrarily, $\Omega'$ agrees with $\Omega$ on the generating set $\{H_{d^r}\}_{d,r\geq 1}$ and thus $\Omega'= \Omega$ as $\QQ$-algebra isomorphisms on $\psym$.
\end{proof}

Recall that for a composition $\beta$,  $\area(\be)$ is the sum of the entries of $\be$ and is equal to the size of $\be$. The notions of size and area coincide for compositions but differ for ordered types. The size of an ordered type $\de$ is $|\de| = \sum_{i\geq 1} i \area(\de|_i)$. On the other hand, we define the \textit{area} of an ordered type $\de$ as $\area(\de) =  \sum_{i\geq 1} \area(\de|_i)$. This is equal to the total number of boxes in the top rows of the blocks in the block diagram of $\de$. This definition specializes naturally to types $\tau$ as $\area(\tau) =  \sum_{i\geq 1} \area(\tau|_i)$.
\begin{corollary}\label{cor:Omega-of-htensor}
For an ordered type $\de$, $\Omega(h^\otimes_\de) = (-1)^{\area(\de)} e^\otimes_\de$ and $\Omega(e^\otimes_\de) = (-1)^{\area(\de)} h^\otimes_\de$ .
\end{corollary}
\begin{proof}
We write $h^\otimes_\de = h_{\de|_1}\otimes h_{\de|_2}\otimes \cdots$. Specializing $f_i$ to $h_{\de|_i}$ and $g_i$ to $p_1$ for $i\geq 1$ in Proposition \ref{prop:Omega-on-plethysm} yields
\[
\Omega(h^\otimes_\de) = h_{\de|_1}[-p_1]\otimes h_{\de|_2}[-p_2]\otimes \cdots.
\]
Using Proposition \ref{prop:negation-rule}, $h_\beta[-p_1] = (-1)^{\area(\be)} e_\be$ for any composition $\be$, so
\[
\Omega(h^\otimes_\de) = (-1)^{\area(\de|_1)}e_{\de|_1}\otimes (-1)^{\area(\de|_2)}e_{\de|_2}\otimes \cdots= (-1)^{\area(\de)} e^\otimes_\de.
\]
The proof for $\Omega(e^\otimes_\de) = (-1)^{\area(\de)} h^\otimes_\de$ is identical.
\end{proof}
\subsection{Transition Matrices}
Let $n$ be a non-negative integer. Suppose $\{F_\tau\}_{\tau\in \Typ(n)}$ and $\{G_\tau\}_{\tau\in \Typ(n)}$  are bases of $\psym(n)$. Suppose for each $\sigma\in \Typ(n)$, we have the unique expansions
\[
F_\sigma = \sum\limits_{\tau\in \Typ(n)} a_{\tau, \sigma} G_\tau.
\]
We can record the coefficients appearing in the expansions using the \textit{transition matrix from $\{F_\tau\}_{\tau\in \Typ(n)}$ to $\{G_\tau\}_{\tau\in \Typ(n)}$} which we denote by $\mcM_n(F,G)$. The entry in column $\sigma$ and row $\tau$ records the coefficient $a_{\tau,\si}$ of $G_\tau$ in the $G$-expansion of $F_\sigma$. We present example transition matrices for $n = 4$ for bases dealt with in this paper in Section \ref{sec:matrices}.

%Sep 4 2026
\section{The $h^\otimes$-expansion and the $e^\otimes$-expansion of $P_\tau$}\label{sec:P-he-wrapping}
In this section, we find and combinatorially interpret the transition matrices $\mcM_n(P,h^\otimes)$ and $\mcM_n(P,e^\otimes)$ for $n\geq 1$. We first prove the $h^\otimes$-expansion of $P_{d^r}$ and provide a combinatorial for it. We then use this interpretation to find the $h^\ox$-expansion of $P_\si$ for types $\si$. We then apply $\Omega$ to the $h^\otimes$-expansion of $P_\tau$ to obtain the $e^\otimes$-expansion of $P_\tau$.

First, we rewrite expansion (4) from Proposition \ref{prop:dr-expansions} using the definition of $p^\otimes_\tau$:
\[P_{d^r} = \sum\limits_{k|d} k p_{dr/k}(\x_{k*}).\]

For a composition $\be$, $L(\be)$ is the last entry of $\be$. For instance, if $\be = (3,4,1,4)$, then $L(\be) = 4$. From \cite{eg-rem}, we can deduce the composition-indexed $h$-expansion of $p_n$ for $n\geq 1$ to be
\[
p_n = \sum\limits_{\alpha\in \Com(n)} (-1)^{\ell(\al)-1} L(\al) h_\al.
\]
Plugging the composition-indexed $h$-expansion of $p_n$ into the $p^\otimes$-expansion of $P_{d^r}$ yields the following.
\begin{lemma}\label{lem:P-dr-h}
For integers $d,r\geq 1$,
\[
P_{d^r} = \sum\limits_{k\mid d} k \sum_{\be\in \Com(dr/k)} (-1)^{\ell(\be)-1} L(\be) h^\otimes_{k^\beta}.
\]
\end{lemma}

We now interpret the $h^\otimes$-expansion of $P_{d^r}$ combinatorially via $k$-wrappings of the block $d^r$. A \textit{$k$-sheet} in a block $\bdg(d^r)$ is a collection of $k$ contiguous rows of $\bdg(d^r)$. We partition $d$ rows of the block diagram $\bdg(d^r)$ into $d/k$ \text{$k$-sheets}, each consisting of $k$ consecutive rows starting from top. Thus the first $k$-sheet consists of rows $1,2,\ldots, k$, the second $k$-sheet consists of rows $k+1, \ldots, 2k$, and so on. For a composition $\beta = (\be_1, \be_2, \ldots, \be_l)$ of $dr/k$, a \textit{$k$-wrapping of the block $d^r$ of content $\be$} is a partition\footnote{The term partition here is to be understood in the sense of set partitions and not integer partitions.} $(B_1, B_2, \ldots, B_l)$ of the boxes of $\bdg(d^r)$ as follows:
\begin{enumerate}
\item Suppose $\be_1 = w_1 r + e_1$ where $0 \leq e_1 < r$. The partition $B_1$ consists of all top $w_1$ $k$-sheets and $e_1$ columns of the $(w_1 + 1)$st $k$-sheet.
\item Suppose $\be_1 + \be_2 = w_2r + e_2$ where $0 \leq e_2 < r$. The partition $B_1\cup B_2$ consists of all top $w_2$ $k$-sheets and $e_2$ columns of the $(w_2 + 1)$st $k$-sheet. We find the partition $B_2$ by removing the boxes in $B_1$ from $B_1 \cup B_2$.
\item We continue this process for all $\be_1 + \be_2 + \ldots + \be_k$ for $k = 1,2,\ldots, l$. As $\be$ is a composition of a multiple of $r$, $e_l$ defined as \[|\be| = \be_1 + \be_2 + \ldots + \be_l = w_l r + e_l\] for $0\leq e_l < r$ must equal 0.
\end{enumerate}
We consider column $r$ in each $k$-sheet $f$ and column 1 in each $k$-sheet $f + 1$ to be adjacent for all $f = 1,2,\ldots, d/k-1$. We define a \textit{rectangle} as follows: place two vertical lines of height $k$ in two (possibly same) $k$-sheets. From the end points of the first vertical line we draw two parallel horizontal lines going to the right. When these horizontal lines touch the edge of the block diagram, we start drawing horizontal lines separated by the same distance from the left edge of the $k$-sheet immediately below. These lines continue onwards until they reach the endpoints of the second vertical line. The area bounded by this figure spanning possibly mutliple $k$-sheets is called a rectangle.
We visualize the $k$-wrapping $B = (B_1, B_2, \ldots, B_l)$ by covering the boxes contained in each $B_i$ for $i= 1,2,\ldots, l$ with a rectangle. For the $k$-wrapping $B = (B_1, B_2, \ldots, B_l)$, define $\ell(B) = l$, $\sgn(B) = (-1)^{\ell(B)-1}$ and $\mcL(B) = |B_l| = k\be_l$. Define the \textit{ordered content} $\ocon(B)$ to be the ordered type $k^\be$ where $\be = (\be_1, \be_2, \ldots, \be_l)$ is the composition defined as follows: each entry $\be_i$ counts the total number of columns spanned by $B_i$, where we count a column $f$ times if it appears in $f$ $k$-sheets.

\begin{example}
The following object $B$ is a 3-wrapping of the block $12^5$ with ordered content $(2,4,3,8,3)$. We find $\sgn(B) = (-1)^{5-1} = 1$ and $\mcL(B) = 3$. Note that $B_4$ spans columns 1, 2 and 5 twice, and columns 3 and 4 once. Thus, we find $\be_4 = 2 + 2 + 1 + 1 + 2 =8$.
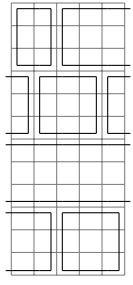
\begin{center}
\raisebox{12.0pt}{\begin{tikzpicture}[scale =0.300,baseline=(current bounding box.north)]
\draw[step = 1.0, gray, thin] (0,0) grid (5,-12);
\draw (0.250,-0.250) -- (0.250,-2.75);
\draw (0.250,-0.250) -- (1.25,-0.250);
\draw (0.250,-2.75) -- (1.25,-2.75);
\draw (1.25,-0.250) -- (1.75,-0.250);
\draw (1.25,-2.75) -- (1.75,-2.75);
\draw (1.75,-0.250) -- (1.75,-2.75);
\draw (2.25,-0.250) -- (2.25,-2.75);
\draw (2.25,-0.250) -- (3.25,-0.250);
\draw (2.25,-2.75) -- (3.25,-2.75);
\draw (3.25,-0.250) -- (4.25,-0.250);
\draw (3.25,-2.75) -- (4.25,-2.75);
\draw (4.25,-0.250) -- (5.25,-0.250);
\draw (4.25,-2.75) -- (5.25,-2.75);
\draw (-0.25,-3.25) -- (0.25,-3.25);
\draw (-0.25,-5.75) -- (0.25,-5.75);
\draw (0.250,-3.25) -- (0.750,-3.25);
\draw (0.250,-5.75) -- (0.750,-5.75);
\draw (0.750,-3.25) -- (0.750,-5.75);
\draw (1.25,-3.25) -- (1.25,-5.75);
\draw (1.25,-3.25) -- (2.25,-3.25);
\draw (1.25,-5.75) -- (2.25,-5.75);
\draw (2.25,-3.25) -- (3.25,-3.25);
\draw (2.25,-5.75) -- (3.25,-5.75);
\draw (3.25,-3.25) -- (3.75,-3.25);
\draw (3.25,-5.75) -- (3.75,-5.75);
\draw (3.75,-3.25) -- (3.75,-5.75);
\draw (4.25,-3.25) -- (4.25,-5.75);
\draw (4.25,-3.25) -- (5.25,-3.25);
\draw (4.25,-5.75) -- (5.25,-5.75);
\draw (-0.25,-6.25) -- (0.25,-6.25);
\draw (-0.25,-8.75) -- (0.25,-8.75);
\draw (0.250,-6.25) -- (1.25,-6.25);
\draw (0.250,-8.75) -- (1.25,-8.75);
\draw (1.25,-6.25) -- (2.25,-6.25);
\draw (1.25,-8.75) -- (2.25,-8.75);
\draw (2.25,-6.25) -- (3.25,-6.25);
\draw (2.25,-8.75) -- (3.25,-8.75);
\draw (3.25,-6.25) -- (4.25,-6.25);
\draw (3.25,-8.75) -- (4.25,-8.75);
\draw (4.25,-6.25) -- (5.25,-6.25);
\draw (4.25,-8.75) -- (5.25,-8.75);
\draw (-0.25,-9.25) -- (0.25,-9.25);
\draw (-0.25,-11.8) -- (0.25,-11.8);
\draw (0.250,-9.25) -- (1.25,-9.25);
\draw (0.250,-11.8) -- (1.25,-11.8);
\draw (1.25,-9.25) -- (1.75,-9.25);
\draw (1.25,-11.8) -- (1.75,-11.8);
\draw (1.75,-9.25) -- (1.75,-11.8);
\draw (2.25,-9.25) -- (2.25,-11.8);
\draw (2.25,-9.25) -- (3.25,-9.25);
\draw (2.25,-11.8) -- (3.25,-11.8);
\draw (3.25,-9.25) -- (4.25,-9.25);
\draw (3.25,-11.8) -- (4.25,-11.8);
\draw (4.25,-9.25) -- (4.75,-9.25);
\draw (4.25,-11.8) -- (4.75,-11.8);
\draw (4.75,-9.25) -- (4.75,-11.8);
\end{tikzpicture}}
\end{center}
\end{example}

A \textit{wrapping of $d^r$} is any $k$-wrapping of $d^r$ for some $k$ dividing $d$. Denote the set of wrappings of $d^r$ by $\wrap(d^r)$.

%6 Sep 2026 00:07
\begin{prop}\label{prop:P-dr-comb}
For positive integers $d$ and $r$,
\[
P_{d^r} = \sum_{B\in \wrap(d^r)} \sgn(B) \mcL(B) h^\otimes_{{\ocon(B)}}.
\]
\end{prop}
\begin{proof}
We show that the above expression is equivalent to the $h^\ox$-expansion in Lemma \ref{lem:P-dr-h}.
We rewrite the $h^\ox$-expansion of $P_{d^r}$ from Lemma \ref{lem:P-dr-h} as \[P_{d^r} = \sum\limits_{k\mid d} \sum_{\be\in \Com(dr/k)} (-1)^{\ell(\be)-1} k L(\be) h^\otimes_{k^\beta}.\]
With each term $(-1)^{\ell(\beta)} kL(\beta) h^\otimes_{k^\beta}$ in the $h^\otimes$-expansion of $P_{d^r}$, we associate the $k$-wrapping $B = (B_1, B_2, \ldots, B_l)$ of $d^r$ of ordered content $k^\beta$. As $\ell(\beta)$ is also the number of entries in the $k$-wrapping $B$, $(-1)^{\ell(\beta)-1} = \sgn(B)$. As $B_l$ spans $\beta_l = L(\be)$ columns, each of height $k$, the number of boxes in $B_l$ is $|B_l| = k L(\beta)$. This shows that for a chosen composition $\beta$ and the $k$-wrapping $B$ of $d^r$ of ordered content $k^\be$, $(-1)^{\ell(\beta)} kL(\beta) h^\otimes_{k^\beta} =  \sgn(B) \mcL(B)h^\ox_{\ocon(B)}$. 

On the other hand, start with a $k$-wrapping $B = (B_1, B_2, \ldots, B_l)$. As each rectangle has the same height, the height must divide the total height of the array, thus $k$ must divide $d$. Furthermore, by recording the number of columns with multiplicity spanned by $B_i$, the ordered content $\ocon(B)$ is $k^\al$ for the composition $\al = (\al_1, \al_2, \ldots, \al_l)$ of $dr/k$. We also find $\sgn(B) = (-1)^{\ell(\al) - 1}$ and $\mcL(B) = |B_l| = kL(\alpha)$. Thus, the $k$-wrapping $B$ gives rise to the term $h^{\otimes}_{k^\alpha}$ with the coefficient $\sgn(B)\mcL(B) = (-1)^{\ell(\alpha)-1} k L(\alpha)$.
\end{proof}

\begin{example}
The coefficient of $h^\ox_{2^{2,2,1,1}}$ in $P_{4^3}$ is $-18$. We find $\sgn(B_i) = -1$ for $i = 1,2,\ldots, 6$.  Also $\mcL(B_1) = \mcL(B_2) = \mcL(B_3) = 2\cdot 2$ as the last rectangle spans 2 columns and has height 2, and $\mcL(B_4) = \mcL(B_5) = \mcL(B_6) = 2\cdot 1$ as the last rectangle spans 1 column and has height 2.
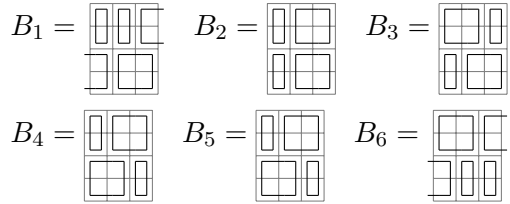
\[
B_1 = \raisebox{12.0pt}{\begin{tikzpicture}[scale =0.300,baseline=(current bounding box.north)]
\draw[step = 1.0, gray, thin] (0,0) grid (3,-4);
\draw (0.250,-0.250) -- (0.250,-1.75);
\draw (0.250,-0.250) -- (0.750,-0.250);
\draw (0.250,-1.75) -- (0.750,-1.75);
\draw (0.750,-0.250) -- (0.750,-1.75);
\draw (1.25,-0.250) -- (1.25,-1.75);
\draw (1.25,-0.250) -- (1.75,-0.250);
\draw (1.25,-1.75) -- (1.75,-1.75);
\draw (1.75,-0.250) -- (1.75,-1.75);
\draw (2.25,-0.250) -- (2.25,-1.75);
\draw (2.25,-0.250) -- (3.25,-0.250);
\draw (2.25,-1.75) -- (3.25,-1.75);
\draw (-0.25,-2.25) -- (0.25,-2.25);
\draw (-0.25,-3.75) -- (0.25,-3.75);
\draw (0.250,-2.25) -- (0.750,-2.25);
\draw (0.250,-3.75) -- (0.750,-3.75);
\draw (0.750,-2.25) -- (0.750,-3.75);
\draw (1.25,-2.25) -- (1.25,-3.75);
\draw (1.25,-2.25) -- (2.25,-2.25);
\draw (1.25,-3.75) -- (2.25,-3.75);
\draw (2.25,-2.25) -- (2.75,-2.25);
\draw (2.25,-3.75) -- (2.75,-3.75);
\draw (2.75,-2.25) -- (2.75,-3.75);
\end{tikzpicture}}\quad 
B_2 = \raisebox{12.0pt}{\begin{tikzpicture}[scale =0.300,baseline=(current bounding box.north)]
\draw[step = 1.0, gray, thin] (0,0) grid (3,-4);
\draw (0.250,-0.250) -- (0.250,-1.75);
\draw (0.250,-0.250) -- (0.750,-0.250);
\draw (0.250,-1.75) -- (0.750,-1.75);
\draw (0.750,-0.250) -- (0.750,-1.75);
\draw (1.25,-0.250) -- (1.25,-1.75);
\draw (1.25,-0.250) -- (2.25,-0.250);
\draw (1.25,-1.75) -- (2.25,-1.75);
\draw (2.25,-0.250) -- (2.75,-0.250);
\draw (2.25,-1.75) -- (2.75,-1.75);
\draw (2.75,-0.250) -- (2.75,-1.75);
\draw (0.250,-2.25) -- (0.250,-3.75);
\draw (0.250,-2.25) -- (0.750,-2.25);
\draw (0.250,-3.75) -- (0.750,-3.75);
\draw (0.750,-2.25) -- (0.750,-3.75);
\draw (1.25,-2.25) -- (1.25,-3.75);
\draw (1.25,-2.25) -- (2.25,-2.25);
\draw (1.25,-3.75) -- (2.25,-3.75);
\draw (2.25,-2.25) -- (2.75,-2.25);
\draw (2.25,-3.75) -- (2.75,-3.75);
\draw (2.75,-2.25) -- (2.75,-3.75);
\end{tikzpicture}}\quad B_3 = \raisebox{12.0pt}{\begin{tikzpicture}[scale =0.300,baseline=(current bounding box.north)]
\draw[step = 1.0, gray, thin] (0,0) grid (3,-4);
\draw (0.250,-0.250) -- (0.250,-1.75);
\draw (0.250,-0.250) -- (1.25,-0.250);
\draw (0.250,-1.75) -- (1.25,-1.75);
\draw (1.25,-0.250) -- (1.75,-0.250);
\draw (1.25,-1.75) -- (1.75,-1.75);
\draw (1.75,-0.250) -- (1.75,-1.75);
\draw (2.25,-0.250) -- (2.25,-1.75);
\draw (2.25,-0.250) -- (2.75,-0.250);
\draw (2.25,-1.75) -- (2.75,-1.75);
\draw (2.75,-0.250) -- (2.75,-1.75);
\draw (0.250,-2.25) -- (0.250,-3.75);
\draw (0.250,-2.25) -- (0.750,-2.25);
\draw (0.250,-3.75) -- (0.750,-3.75);
\draw (0.750,-2.25) -- (0.750,-3.75);
\draw (1.25,-2.25) -- (1.25,-3.75);
\draw (1.25,-2.25) -- (2.25,-2.25);
\draw (1.25,-3.75) -- (2.25,-3.75);
\draw (2.25,-2.25) -- (2.75,-2.25);
\draw (2.25,-3.75) -- (2.75,-3.75);
\draw (2.75,-2.25) -- (2.75,-3.75);
\end{tikzpicture}}
\]
\[
B_4 = \raisebox{12.0pt}{\begin{tikzpicture}[scale =0.300,baseline=(current bounding box.north)]
\draw[step = 1.0, gray, thin] (0,0) grid (3,-4);
\draw (0.250,-0.250) -- (0.250,-1.75);
\draw (0.250,-0.250) -- (0.750,-0.250);
\draw (0.250,-1.75) -- (0.750,-1.75);
\draw (0.750,-0.250) -- (0.750,-1.75);
\draw (1.25,-0.250) -- (1.25,-1.75);
\draw (1.25,-0.250) -- (2.25,-0.250);
\draw (1.25,-1.75) -- (2.25,-1.75);
\draw (2.25,-0.250) -- (2.75,-0.250);
\draw (2.25,-1.75) -- (2.75,-1.75);
\draw (2.75,-0.250) -- (2.75,-1.75);
\draw (0.250,-2.25) -- (0.250,-3.75);
\draw (0.250,-2.25) -- (1.25,-2.25);
\draw (0.250,-3.75) -- (1.25,-3.75);
\draw (1.25,-2.25) -- (1.75,-2.25);
\draw (1.25,-3.75) -- (1.75,-3.75);
\draw (1.75,-2.25) -- (1.75,-3.75);
\draw (2.25,-2.25) -- (2.25,-3.75);
\draw (2.25,-2.25) -- (2.75,-2.25);
\draw (2.25,-3.75) -- (2.75,-3.75);
\draw (2.75,-2.25) -- (2.75,-3.75);
\end{tikzpicture}}\quad B_5 = \raisebox{12.0pt}{\begin{tikzpicture}[scale =0.300,baseline=(current bounding box.north)]
\draw[step = 1.0, gray, thin] (0,0) grid (3,-4);
\draw (0.250,-0.250) -- (0.250,-1.75);
\draw (0.250,-0.250) -- (0.750,-0.250);
\draw (0.250,-1.75) -- (0.750,-1.75);
\draw (0.750,-0.250) -- (0.750,-1.75);
\draw (1.25,-0.250) -- (1.25,-1.75);
\draw (1.25,-0.250) -- (2.25,-0.250);
\draw (1.25,-1.75) -- (2.25,-1.75);
\draw (2.25,-0.250) -- (2.75,-0.250);
\draw (2.25,-1.75) -- (2.75,-1.75);
\draw (2.75,-0.250) -- (2.75,-1.75);
\draw (0.250,-2.25) -- (0.250,-3.75);
\draw (0.250,-2.25) -- (1.25,-2.25);
\draw (0.250,-3.75) -- (1.25,-3.75);
\draw (1.25,-2.25) -- (1.75,-2.25);
\draw (1.25,-3.75) -- (1.75,-3.75);
\draw (1.75,-2.25) -- (1.75,-3.75);
\draw (2.25,-2.25) -- (2.25,-3.75);
\draw (2.25,-2.25) -- (2.75,-2.25);
\draw (2.25,-3.75) -- (2.75,-3.75);
\draw (2.75,-2.25) -- (2.75,-3.75);
\end{tikzpicture}} \quad B_6 = \raisebox{12.0pt}{\begin{tikzpicture}[scale =0.300,baseline=(current bounding box.north)]
\draw[step = 1.0, gray, thin] (0,0) grid (3,-4);
\draw (0.250,-0.250) -- (0.250,-1.75);
\draw (0.250,-0.250) -- (1.25,-0.250);
\draw (0.250,-1.75) -- (1.25,-1.75);
\draw (1.25,-0.250) -- (1.75,-0.250);
\draw (1.25,-1.75) -- (1.75,-1.75);
\draw (1.75,-0.250) -- (1.75,-1.75);
\draw (2.25,-0.250) -- (2.25,-1.75);
\draw (2.25,-0.250) -- (3.25,-0.250);
\draw (2.25,-1.75) -- (3.25,-1.75);
\draw (-0.25,-2.25) -- (0.25,-2.25);
\draw (-0.25,-3.75) -- (0.25,-3.75);
\draw (0.250,-2.25) -- (0.750,-2.25);
\draw (0.250,-3.75) -- (0.750,-3.75);
\draw (0.750,-2.25) -- (0.750,-3.75);
\draw (1.25,-2.25) -- (1.25,-3.75);
\draw (1.25,-2.25) -- (1.75,-2.25);
\draw (1.25,-3.75) -- (1.75,-3.75);
\draw (1.75,-2.25) -- (1.75,-3.75);
\draw (2.25,-2.25) -- (2.25,-3.75);
\draw (2.25,-2.25) -- (2.75,-2.25);
\draw (2.25,-3.75) -- (2.75,-3.75);
\draw (2.75,-2.25) -- (2.75,-3.75);
\end{tikzpicture}}
\]
\end{example}
Let $\si$ and $\tau$ be types of a positive integer $n$. Define a \textit{constant wrapping block tabloid} $T$ of shape $\sigma = d_1^{r_1} d_2^{r_2}\ldots d_k^{r_k}$ and content $\tau = e_1^{s_1}e_2^{s_2}\ldots e_l^{s_l}$ as simultaneous wrappings of all blocks $d_1^{r_1}, d_2^{r_2},\ldots, d_k^{r_k}$ using rectangles of height $e_i$ which span $s_i$ columns (possibly across multiple rows) for $i = 1,2,\ldots, l$. Define $\wt(T)$ to be the product of number of boxes contained in the last rectangles in the blocks $d_i^{r_i}$ for $i = 1,2,\ldots, k$.

Equivalently a constant wrapping block tabloid $T$ of shape $\si$ and content $\tau$ is a list $(B^{(1)}, B^{(2)}, \ldots, B^{(k)})$ where each $B^{(i)}$ is a wrapping of $d_i^{r_i}$ and the multiset union $\bigcup_{i=1}^k \psort(\ocon(B^{(i)}))$ equals $\tau$. Define $\sgn(T) = (-1)^{\ell(\tau)-\ell(\si)}$ and $\wt(T) = \prod_{i=1}^k \mcL(B^{(i)})$.

\begin{example}
The following constant wrapping block tabloid has shape $\sigma = 6^4 4^3 4^3$ and content $\tau = 2^{3,3,2,2,2,2,2,1,1}1^{5,3,1,3}$:
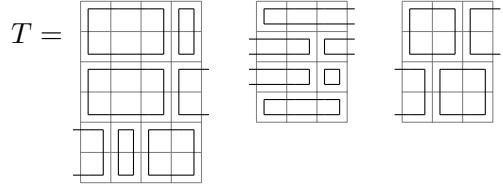
\begin{center}
$T =$
\raisebox{16.0pt}{\begin{tikzpicture}[scale =0.400,baseline=(current bounding box.north)]
\draw[step = 1.0, gray, thin] (0,0) grid (4,-6);
\draw (0.250,-0.250) -- (0.250,-1.75);
\draw (0.250,-0.250) -- (1.25,-0.250);
\draw (0.250,-1.75) -- (1.25,-1.75);
\draw (1.25,-0.250) -- (2.25,-0.250);
\draw (1.25,-1.75) -- (2.25,-1.75);
\draw (2.25,-0.250) -- (2.75,-0.250);
\draw (2.25,-1.75) -- (2.75,-1.75);
\draw (2.75,-0.250) -- (2.75,-1.75);
\draw (3.25,-0.250) -- (3.25,-1.75);
\draw (3.25,-0.250) -- (3.75,-0.250);
\draw (3.25,-1.75) -- (3.75,-1.75);
\draw (3.75,-0.250) -- (3.75,-1.75);
\draw (0.250,-2.25) -- (0.250,-3.75);
\draw (0.250,-2.25) -- (1.25,-2.25);
\draw (0.250,-3.75) -- (1.25,-3.75);
\draw (1.25,-2.25) -- (2.25,-2.25);
\draw (1.25,-3.75) -- (2.25,-3.75);
\draw (2.25,-2.25) -- (2.75,-2.25);
\draw (2.25,-3.75) -- (2.75,-3.75);
\draw (2.75,-2.25) -- (2.75,-3.75);
\draw (3.25,-2.25) -- (3.25,-3.75);
\draw (3.25,-2.25) -- (4.25,-2.25);
\draw (3.25,-3.75) -- (4.25,-3.75);
\draw (-0.25,-4.25) -- (0.25,-4.25);
\draw (-0.25,-5.75) -- (0.25,-5.75);
\draw (0.250,-4.25) -- (0.750,-4.25);
\draw (0.250,-5.75) -- (0.750,-5.75);
\draw (0.750,-4.25) -- (0.750,-5.75);
\draw (1.25,-4.25) -- (1.25,-5.75);
\draw (1.25,-4.25) -- (1.75,-4.25);
\draw (1.25,-5.75) -- (1.75,-5.75);
\draw (1.75,-4.25) -- (1.75,-5.75);
\draw (2.25,-4.25) -- (2.25,-5.75);
\draw (2.25,-4.25) -- (3.25,-4.25);
\draw (2.25,-5.75) -- (3.25,-5.75);
\draw (3.25,-4.25) -- (3.75,-4.25);
\draw (3.25,-5.75) -- (3.75,-5.75);
\draw (3.75,-4.25) -- (3.75,-5.75);
\end{tikzpicture}}
\quad
\raisebox{16.0pt}{\begin{tikzpicture}[scale =0.400,baseline=(current bounding box.north)]
\draw[step = 1.0, gray, thin] (0,0) grid (3,-4);
\draw (0.250,-0.250) -- (0.250,-0.750);
\draw (0.250,-0.250) -- (1.25,-0.250);
\draw (0.250,-0.750) -- (1.25,-0.750);
\draw (1.25,-0.250) -- (2.25,-0.250);
\draw (1.25,-0.750) -- (2.25,-0.750);
\draw (2.25,-0.250) -- (3.25,-0.250);
\draw (2.25,-0.750) -- (3.25,-0.750);
\draw (-0.25,-1.25) -- (0.25,-1.25);
\draw (-0.25,-1.75) -- (0.25,-1.75);
\draw (0.250,-1.25) -- (1.25,-1.25);
\draw (0.250,-1.75) -- (1.25,-1.75);
\draw (1.25,-1.25) -- (1.75,-1.25);
\draw (1.25,-1.75) -- (1.75,-1.75);
\draw (1.75,-1.25) -- (1.75,-1.75);
\draw (2.25,-1.25) -- (2.25,-1.75);
\draw (2.25,-1.25) -- (3.25,-1.25);
\draw (2.25,-1.75) -- (3.25,-1.75);
\draw (-0.25,-2.25) -- (0.25,-2.25);
\draw (-0.25,-2.75) -- (0.25,-2.75);
\draw (0.250,-2.25) -- (1.25,-2.25);
\draw (0.250,-2.75) -- (1.25,-2.75);
\draw (1.25,-2.25) -- (1.75,-2.25);
\draw (1.25,-2.75) -- (1.75,-2.75);
\draw (1.75,-2.25) -- (1.75,-2.75);
\draw (2.25,-2.25) -- (2.25,-2.75);
\draw (2.25,-2.25) -- (2.75,-2.25);
\draw (2.25,-2.75) -- (2.75,-2.75);
\draw (2.75,-2.25) -- (2.75,-2.75);
\draw (0.250,-3.25) -- (0.250,-3.75);
\draw (0.250,-3.25) -- (1.25,-3.25);
\draw (0.250,-3.75) -- (1.25,-3.75);
\draw (1.25,-3.25) -- (2.25,-3.25);
\draw (1.25,-3.75) -- (2.25,-3.75);
\draw (2.25,-3.25) -- (2.75,-3.25);
\draw (2.25,-3.75) -- (2.75,-3.75);
\draw (2.75,-3.25) -- (2.75,-3.75);
\end{tikzpicture}} \quad \raisebox{16.0pt}{\begin{tikzpicture}[scale =0.400,baseline=(current bounding box.north)]
\draw[step = 1.0, gray, thin] (0,0) grid (3,-4);
\draw (0.250,-0.250) -- (0.250,-1.75);
\draw (0.250,-0.250) -- (1.25,-0.250);
\draw (0.250,-1.75) -- (1.25,-1.75);
\draw (1.25,-0.250) -- (1.75,-0.250);
\draw (1.25,-1.75) -- (1.75,-1.75);
\draw (1.75,-0.250) -- (1.75,-1.75);
\draw (2.25,-0.250) -- (2.25,-1.75);
\draw (2.25,-0.250) -- (3.25,-0.250);
\draw (2.25,-1.75) -- (3.25,-1.75);
\draw (-0.25,-2.25) -- (0.25,-2.25);
\draw (-0.25,-3.75) -- (0.25,-3.75);
\draw (0.250,-2.25) -- (0.750,-2.25);
\draw (0.250,-3.75) -- (0.750,-3.75);
\draw (0.750,-2.25) -- (0.750,-3.75);
\draw (1.25,-2.25) -- (1.25,-3.75);
\draw (1.25,-2.25) -- (2.25,-2.25);
\draw (1.25,-3.75) -- (2.25,-3.75);
\draw (2.25,-2.25) -- (2.75,-2.25);
\draw (2.25,-3.75) -- (2.75,-3.75);
\draw (2.75,-2.25) -- (2.75,-3.75);
\end{tikzpicture}} \end{center} 
We compute $\sgn(T) = (-1)^{11-3} = 1$ and $\wt(T) = 4\cdot 3 \cdot 4 = 48$.
\end{example}

\begin{theorem}\label{thm:P-in-htensor}
Let $\sigma$ and $\tau$ be types of $n\geq 1$. The coefficient of $h_\tau^\otimes$ in the $h^\otimes$-expansion of $P_\sigma$ is $\sum_T \sgn(T) \wt(T) =  (-1)^{\ell(\sigma) - \ell(\tau)} \sum_T \wt(T)$ where the sum is over all constant wrapping block tabloids $T$ of shape $\sigma$ and content $\tau$.
\end{theorem}
\begin{proof}
From Proposition \ref{prop:P-dr-comb}, 
\begin{align*}
P_{d_1^{r_1}} P_{d_2^{r_2}}\cdots P_{d_k^{r_k}} &= \left( \sum_{B^{(1)}\in \wrap(d^r)} \sgn(B^{(1)}) \mcL(B^{(1)}) h^\otimes_{{\ocon(B^{(1)})}}\right) \\ & \left( \sum_{B^{(2)}\in \wrap(d^r)} \sgn(B^{(2)}) \mcL(B^{(2)}) h^\otimes_{{\ocon(B^{(2)})}}\right) \cdots\\ & \left( \sum_{B^{(k)}\in \wrap(d^r)} \sgn(B^{(k)}) \mcL(B^{(k)}) h^\otimes_{{\ocon(B^{(k)})}}\right) 
\end{align*}
 For the $h$-basis and two compositions $\be$ and $\ga$, $h_\be h_\ga = h_{\be\cup \ga}$ where $\be\cup \ga$ is the partition whose multiset of parts is the multiset union of the multisets of entries of $\be$ and $\ga$. The $h^\ox$-basis satisfies a similar property for types where for two ordered types $\de$ and $\rho$, $h^\ox_\de h^\ox_\rho = h^\ox_{\de\cup \rho}$, where we define $\de\cup \rho$ to be the type viewed as a multiset union of the multisets of blocks of $\de$ and $\rho$. Choosing the wrapping $B^{(i)}$ for $d_i^{r_i}$ for each $i = 1,2,\ldots, k$ yields a constant wrapping block tabloid $T$ with content $\bigcup\limits_{i=1}^k \ocon(B^{(i)})$ which we call $\tau$. Here the union is a multiset union. Thus for $T$, we obtain the associated term in the $h^\ox$-expansion as \[
 \left(\prod_{i=1}^k \sgn(B^{(i)}) \prod_{i=1}^k \mcL(B^{(i)}) \right) h^\ox_{\tau}. \quad (*) \]
From the multiset equality, $\sum_{i=1}^k \ell(B^{(i)}) = \ell(\tau)$ and so $\sum_{i=1}^k (\ell(B^{(i)})-1) = \ell(\tau) - \ell(\si)$. This shows $\sgn(T) = (-1)^{\ell(\tau) - \ell(\si)}$. Thus the above term (*) is $\sgn(T) \wt(T) h^\ox_\tau$ for $T$ a constant wrapping block tabloid of content $\tau$ and shape $\si$. The coefficient of $h^\ox_\tau$ is the sum over all constant wrapping block tabloids of content $\tau$ and shape $\si$. The proof of the theorem is completed by checking the claim that all constant wrapping block tabloids of content $\tau$ are obtained by independently selecting wrappings of blocks $d_i^{r_i}$.
\end{proof}

%UNF Acyclic and parking functions anywhere?

Recall that for an ordered type $\de$, $\area(\delta) = \sum_{i\geq 1} \area(\de|_i)$.
\begin{theorem}\label{thm:P-in-etensor}
Let $\sigma$ and $\tau$ be types of $n\geq 1$. The coefficient of $e_\tau^\otimes$ in the $e^\otimes$-expansion of $P_\sigma$ is $(-1)^{\area(\tau) - \ell(\tau)} \sum_T \wt(T)$ where the sum is over all constant wrapping block tabloids $T$ of shape $\sigma$ and content $\tau$.
\end{theorem}
\begin{proof}
From Corollary 16 in \cite{khanna-bij}, $\Omega(P_\si) = (-1)^{\ell(\si)} P_\si$. If $c_{\tau,\si}$ is the coefficient of $h^\otimes_\tau$ in the $h^\ox$-expansion of $P_\si$ as found in Theorem \ref{thm:P-in-htensor}, then $P_\si = \sum_{\tau} c_{\tau,\si} h^\ox_\tau$. Applying $\Omega$ on this expansion and using Corollary \ref{cor:Omega-of-htensor}, we obtain
\[
(-1)^{\ell(\si)} P_\si = \sum_{\tau} (-1)^{\area(\tau)}c_{\si,\tau} e^\ox_\tau.
\]
As $c_{\si,\tau}$ is $(-1)^{\ell(\sigma) - \ell(\tau)} \sum_T \wt(T)$ where the sum is over all constant wrapping block tabloids $T$ of shape $\sigma$ and content $\tau$, the coefficient of $e^\ox_\tau$ is $(-1)^{\ell(\si)} (-1)^{\ell(\sigma) - \ell(\tau)} (-1)^{\area(\tau)}\sum_T \wt(T)$ which simplifies to the claim of the theorem.
\end{proof}
\section{$h$-expansion of $h_n[p_r]$}
Let $n$ and $r$ be positive integers.
The $h^\ox$-expansions and the $e^\ox$-expansions presented in Section \ref{sec:he-tensor-expa} are derived via the combinatorics of the $h$-expansion of the plethysm $h_n[p_r]$. We first present the proof of the Schur expansion of $h_n[p_r]$ following Ex. 7 from Section I.8 in Macdonald's book \cite{macd}. We then give an interpretation of this expansion in terms of objects called abaci. Then by using the recursive combinatorial interpretation of the Jacobi-Trudi identity as described by E\u{g}ecio\u{g}lu and Remmel \cite{inv-kostka}, we find the $h$-expansion of $h_n[p_r]$. First we discuss the prerequisite combinatorial notions.
\subsection{Ribbons and cores}
%For a partition $\la$, define the \textit{border of $\la$} to be the set of all cells $c$ in $\dg(\la)$ for which at least one of the following holds:
%\begin{itemize}
%\item The cell immediately  to the east of $c$ is not in $\dg(\la)$.
%\item The cell immediately to the south of $c$ is not in $\dg(\la)$.
%\item The cell immediately to the southeast of $c$ is not in $\dg(\la)$.
%\end{itemize}
%We use the notation $\border(\la)$ for the border of $\la$.
%\begin{example}
%For the partition $\la = (4,4,3,3,1)$, we draw $\dg(\la) = \yg{3,2,2}{3+1,2+2,2+1,3,1}$ and shade $\border(\la)$ in gray.
%
%\end{example}
%A subset $R$ of $\border(\la)$ is called \textit{connected} if for each pair of cells $c$ and $c'$ in $R$, we can traverse from $c$ to $c'$ using a sequence of north/south and east/west steps.
 
A \textit{$k$-ribbon}, also called a \textit{$k$-rim hook }or a \textit{$k$-border strip}, is a subset of a partition diagram containing $k$ cells which can be traversed using a sequence of only unit north and east moves.  The\textit{ height} $\height(R)$ of a $k$-ribbon $R$ is one less than the number of rows occupied by $R$. Define the \textit{sign of the ribbon $R$} as $\sgn(R) = (-1)^{\height(R)}$. A \textit{special $k$-ribbon} is a $k$-ribbon whose southwesternmost cell is in the first column.
\begin{example}\label{ex:ribbons}
Let $\la = (4,4,3,3,1)$. We shade the 5-ribbon $R$ in gray in the diagram $\yg{4,4,3,3,1}{0,2+2,2+1,1+2}$. We also shade the special 7-ribbon $S$ in gray in the diagram $\yg{4,4,3,3,1}{0,2+2,2+1,3,1}$. As $R$ occupies 3 rows, $\height(R) = 2$ and $\sgn(R) = 1$. As $S$ occupies 4 rows, $\height(S) = 3$ and $\sgn(S) = -1$.  
\end{example}
A $k$-ribbon $R$ contained in $\dg(\la)$ is called \textit{removable} if removing all cells of $R$ from $\dg(\la)$ results in a partition shape.
We say we \textit{remove} a $k$-ribbon $R$ from $\la$ by removing all cells of $R$ contained in $\dg(\la)$. If we denote the partition thus obtained by $\mu$, then we write $R=\la/\mu$. Equivalently, we say that $\la$ is obtained by \textit{adding} the $k$-ribbon $R$ to $\mu$.
\begin{example}
Continuing the notation of Example \ref{ex:ribbons}, we obtain the diagram $\y{4,2,2,1,1}$ after removing $R$ from $\la$ and obtain the diagram $\y{4,2,2}$ after removing $S$ from $\la$. Write $\mu = (4,2,2,1,1)$ and $\nu = (4,2,2)$, then $R = \la/\mu$ and $S = \la/\nu$.
\end{example}

If we successively remove $k$-ribbons from $\la$, then we obtain a partition shape from which no further $k$-ribbons can be removed. The partition obtained in this manner is called the \textit{$k$-core of $\la$}. The $k$-core of $\la$ is unique and independent of the choices of $k$-ribbons removed~\cite[Thm. 2.7.16]{james-kerber}.  In general, we call a partition a \textit{$k$-core} if no $k$-ribbons can be removed from it.

\newcommand{\mcH}{\mathcal{H}}
\subsection{Abacus with runners}
The exposition of this section is based on the book by James and Kerber~\cite{james-kerber} and the monograph by Olsson~\cite{olsson}. 
The operations of addition and removal of ribbons are best understood through the visual aid of an abacus with runners. For a partition $\la$ and a positive integer $r > \ell(\la)$, define the \textit{$r$th} \textit{hook-set} as
\[
\mcH_r(\la) = \{\la_1 + r - 1, \la_2 + r - 2, \ldots, \lambda_{\ell(\la)} + r - \ell(\la), r-1, \ldots, 1, 0\}.
\]
In words, $\mcH_r(\la)$ is obtained by adding $r-1$ to the largest entry, $r-2$ to the second largest entry, and so on adding $r-i$ to the $i$th largest entry, and then to the set thus obtained, appending the entries $r-1, r-2, \ldots, 0$.
Define $\mcH_{\ell(\la)}(\la) =\{\la_1 + \ell(\la) - 1, \la_2 + \ell(\la) - 2, \ldots, \la_{\ell(\la)}\}$. For $r < \ell(\la)$, $\mcH_r(\la)$ is undefined.
Notice that $\mcH_{r+1}(\la)$ can be obtained by increasing each entry in $\mcH_r(\la)$ by 1 and appending a zero. In notation we write for all $r \geq \ell(\la)$, 
\[
\mcH_{r+1}(\la) = \{x + 1: x\in \mcH_r(\la)\} \cup \{0\}.
\]

\begin{prop}[2.7.13 in \cite{james-kerber}]\label{prop:hook-removal}
Let $\la$ be a partition and $r$ be a positive integer satisfying $r\geq \ell(\la)$. Let $R$ be a $k$-ribbon whose top row lies in the $i$th row from top in $\la$ and its bottom row lies in the $j$th row from top in $\la$. If $\mu$ is the partition obtained by removing $R$ from $\la$, then there exists a value $h\in \mcH_r(\la)$ such that $h-k$ is not in $\mcH_r(\la)$, and $\mcH_r(\mu)$ is obtained from $\mcH_r(\la)$ by replacing $h$ by $h-k$. Furthermore, $h$ is the $i$th largest entry of $\mcH_r(\la)$ and $h-k$ is the $j$th largest entry of $\mcH_r(\mu)$.
\end{prop}

\begin{remark}\label{rem:special-hook-removal}
If $R$ is a special $k$-ribbon, then the bottom row of $R$ lies in the bottom row of $\la$, and so in the notation of Proposition \ref{prop:hook-removal}, $j = \ell(\la)$. This means we obtain $\mcH_r(\mu)$ from $\mcH_r(\la)$ by replacing $h$ by the smallest value not contained in $\mcH_r(\la)$.
\end{remark}

As Proposition \ref{prop:hook-removal} tells us that some value in a hook-set changes by $k$ for every removal of a $k$-ribbon, it is useful for us to view the changes in the elements using equivalence classes mod $k$.

An \textit{abacus with $k$ runners} is a semi-infinite array of all non-negative integers with $k$ columns where we place consecutive integers going left-to-right in each row, and the leftmost column contains multiples of $k$ increasing from top to bottom. We label the columns left to right as $0,1, \ldots, k-1$ and the rows going down as $0,1,2,\ldots$. Each entry in the array is called a \textit{position}. For instance, when $k = 5$, the position 13 lies in row 2 and column 3 (keeping in mind the zero-indexing). In general, the label of a position is obtained by adding the column index to the product of the row index and $k$.

For a partition $\la$ and a positive integer $r\geq \ell(\la)$, define $\ab^r_k(\la)$ as the abacus with $k$ runners where we circle the entries of $\mcH_r(\la)$. The circles are called \textit{beads} and all other positions without circles are called \textit{gaps}. So, in $\ab_k^r(\la)$, the superscript $r$ tells us how many beads we have, the subscript $k$ tells us the number of columns, and $\la$ determines the locations of the beads. We \textit{slide} a bead by removing a bead from some position and adding a bead on another position as \textit{sliding} that bead.
\begin{example}
For $\la = (7,7,5,4,3)$, we find $\mcH_{8}(\la) = \{14,13, 10, 8, 6, 2, 1, 0\}$. We draw $\ab_5^8(\la)$ as
\begin{center}\begin{tikzpicture}[shift up/.style = {transform canvas={yshift=.5mm}},shift down/.style = {transform canvas={yshift=-.5mm}},]\matrix (m) [matrix of nodes, column sep = .75cm, row sep = 0.2cm, nodes = {anchor = center}]
{0 & 1 & 2 & 3 & 4\\ 5 & 6 & 7 & 8 & 9\\ 10 & 11 & 12 & 13 & 14\\ };
\draw (m-1-1) circle (0.3cm);
\draw (m-1-2) circle (0.3cm);
\draw (m-1-3) circle (0.3cm);
\draw (m-3-5) circle (0.3cm);
\draw (m-2-2) circle (0.3cm);
\draw (m-2-4) circle (0.3cm);
\draw (m-3-1) circle (0.3cm);
\draw (m-3-4) circle (0.3cm);
\end{tikzpicture}\end{center}
\end{example}
While drawing an abacus, we omit all rows which do not contain beads below the lowermost row containing a bead.

Proposition \ref{prop:hook-removal} tells us that removing a $k$-ribbon corresponds to reducing a certain value in $\mcH_r(\la)$ by $k$, and so a removal of a $k$-ribbon from $\la$ corresponds to sliding a bead in some runner of $\ab_k^r(\la)$ up by one row. Similarly, the addition of a $k$-ribbon corresponds to sliding a bead down by one row. Suppose $b\in \mcH_r(\la)$, then sliding a bead from $b$ to $b'$ where $b - b' = k$ corresponds to removing some $k$-ribbon from $\la$. We compute $\height(R)$ as the number of beads that are in the set of positions $\{b'+1, b'+2 , \ldots, b-1\}$. The \textit{sign of this sliding operation} from $b$ to $b'$ is given by $(-1)^{\height(R)}$. We say that the bead we slid \textit{jumps over} positions $b'+1, b'+2 , \ldots, b-1$.

If $\mu$ is a $k$-core, then we cannot slide any beads upwards. This means that the beads in each runner in  $\ab_k^r(\mu)$ lie in consecutive rows starting from the top. 
\begin{example}
The partition $\la = (6,2,1,1)$ is a 5-core, and $\abc_5^7(\la)$ is
\begin{center}\begin{tikzpicture}[shift up/.style = {transform canvas={yshift=.5mm}},shift down/.style = {transform canvas={yshift=-.5mm}},]\matrix (m) [matrix of nodes, column sep = .75cm, row sep = 0.2cm, nodes = {anchor = center}]
{0 & 1 & 2 & 3 & 4\\ 5 & 6 & 7 & 8 & 9\\ 10 & 11 & 12 & 13 & 14\\ };
\draw (m-1-1) circle (0.3cm);
\draw (m-1-2) circle (0.3cm);
\draw (m-1-3) circle (0.3cm);
\draw (m-1-5) circle (0.3cm);
\draw (m-2-1) circle (0.3cm);
\draw (m-2-3) circle (0.3cm);
\draw (m-3-3) circle (0.3cm);
\end{tikzpicture}\end{center}
\end{example}
\subsection{$s$-expansion of $h_n[p_r]$}
In this section, we focus on partitions which have length at most $r$ and have empty $r$-cores. These partitions are called \textit{$r$-decomposable} in \cite{wildon}. The abacus of any partition contains at least as many beads as the length of the partition. So a partition with length at most $r$ can always be expressed using an abacus with $r$ beads. For the empty partition $\vn$, $\ab_k^r(\vn)$ contains beads in positions $r,r-1,\ldots, 1,0$. When $k= r$, all positions in row 0 are occupied by these $r$ beads, and thus each runner contains exactly one bead in row 0. An $r$-core is obtained via sliding beads upwards in runners, thus a partition with an empty $r$-core which has length at most $r$ contains exactly one bead per runner in some row. When $\la$ is an $r$-decomposable partition, we define $\ab_r(\la) = \ab_r^r(\la)$ to be the abacus with $r$ runners and $r$ beads that contains exactly one bead per runner.
\begin{example}\label{ex:r-decomp-abacus}
The partition $\la = (14,8,7,1)$ of $30$ is 5-decomposable. We can see this by successively removing 5-hooks from the diagram of $\la$ to verify that the 5-core of $\la$ is empty as shown in the following sequence of diagrams.
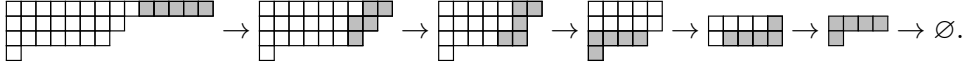
\boks{0.18}
\[
\yg{14,8,7,1}{9+5} \to \yg{9,8,7,1}{7+2,6+2,6+1} \to \yg{7,6,6,1}{5+2,5+1,4+2}\to
\yg{5,5,4,1}{0,0,4,1} \to \yg{5,5}{4+1,1+4} \to \yg{}{4,1}\to \vn.
\]\boks{0.2}
Alternatively, $\abc_5(\la)$ has one bead in each runner and thus $\la$ is 5-decomposable.
\begin{center}\begin{tikzpicture}[shift up/.style = {transform canvas={yshift=.5mm}},shift down/.style = {transform canvas={yshift=-.5mm}},]\matrix (m) [matrix of nodes, column sep = .75cm, row sep = 0.2cm, nodes = {anchor = center}]
{0 & 1 & 2 & 3 & 4\\ 5 & 6 & 7 & 8 & 9\\ 10 & 11 & 12 & 13 & 14\\ 15 & 16 & 17 & 18 & 19\\ };
\draw (m-1-1) circle (0.3cm);
\draw (m-3-2) circle (0.3cm);
\draw (m-1-3) circle (0.3cm);
\draw (m-4-4) circle (0.3cm);
\draw (m-2-5) circle (0.3cm);
\end{tikzpicture}\end{center}

\end{example}
We now elaborate on the perspective that views $r$-decomposable partitions as a rim-hook tableau with rectangular content. Let $\mu$ be a partition of $n$ and $\be$ be a composition of $n$.
Define a \textit{rim-hook tableau} (RHT) $T$ of shape $\mu$ and content $\beta$ to be a filling of cells of $\dg(\mu)$ with values $1,2,\ldots,\ell(\be)$ such that cells containing $i$ form a $\be_i$-ribbon, and the set of the cells containing $1,2,\ldots, i$ forms the partition shape $\dg(\mu^{(i)})$ for all $i = 1,2,\ldots, \ell(\be)$. Let $\mu^{(0)} = \vn$ and $\mu^{(\ell(\be))} = \mu$. Define the \textit{sign} $\sgn(T)$ of $T$ as $\prod\limits_{i=1}^{\ell(\be)}  (-1)^{\height(\mu^{(i)}/\mu^{(i-1)})}$. A rim-hook tableau of shape $\mu$ can also be viewed as an ordered removal of ribbons from $\dg(\mu)$ where we remove the largest labeled ribbon removed at each step.
\begin{example}
The rim-hook tableau
\boks{0.3}
\ytableausetup{nobaseline}
\[
T= \yt{1112,1222,134,334}
\]
has shape $(4,4,3,3)$ and content $(5,4,3,2)$. We find $\sgn(T) = (-1)^{2 + 1 + 1 + 1}$ where the height of the ribbon formed by cells containing the value 1 is $2$ while the ribbons formed by the cells containing 2, 3 and 4 respectively have height 1. The rim-hook tableau $T$ records the following step-by-step removal of ribbons from $(4,4,3,3)$
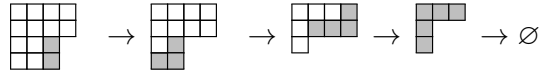
\boks{0.2}
\[
\y{4,4,2,2,}*[*(lightgray)]{4+0,4+0,2+1,2+1,}\: \to \:
\y{4,4,1,0,}*[*(lightgray)]{4+0,4+0,1+1,0+2,}\: \to \:
\y{3,1,1,0,}*[*(lightgray)]{3+1,1+3,1+0,0+0,}\: \to \:
\y{0,0,0,0,}*[*(lightgray)]{0+3,0+1,0+1,0+0,}\: \to \vn
\]
\end{example}
Recall that the sign of a ribbon $R$ removed from a partition $\mu$ can be computed from the abacus of $\mu$ by tracking how many beads are jumped over while moving the corresponding bead. In particular, if we jump over $k$ beads while removing $R$, $\sgn(R)= (-1)^k$. 

An \textit{$r$-permutation} $\pi$ is a list of $r$ integers where each value $0,1,\ldots, r-1$ appears exactly once. The set of these objects is denoted by $\SS_r$. Suppose beads occur in positions $b_0, b_1, \ldots, b_{r-1}$ in $\abc_r(\mu)$. We label each bead in position $b_i$ with the label $i$. Consider the RHT of shape $\mu$ and content $(\be_1, \be_2, \ldots, \be_{l})$ for some $l\geq 1$. Suppose we remove ribbons of sizes $\be_l, \be_{l-1}, \ldots, \be_1$ in order, and this results in the bead in position $b_j$ eventually moving to position $i$ for all $j = 0,1,\ldots, r-1$. When the bead in position $b_j$ ends in position $i$, it retains the label $j$. Thus, the labels of beads in positions $i = 0,1,\ldots, r-1$ now form an $r$-permutation which we denote by $\pi^T$. Note that this correspondence from the RHT to its $r$-permutation is not bijective.
\begin{example}\label{ex:pi-T}
In the following example, $\la = (14,14,6,5,5,3,2)$ and $\ab_7(\la)$ is presented below with beads labeled in ascending order of their positions.
\begin{center}\begin{tikzpicture}[shift up/.style = {transform canvas={yshift=.5mm}},shift down/.style = {transform canvas={yshift=-.5mm}},]\matrix (m) [matrix of nodes, column sep = .75cm, row sep = 0.2cm, nodes = {anchor = center}]
{0 & 1 & $2^0$ & 3 & $4^1$ & 5 & 6\\ $7^2$ & $8^3$ & 9 & $10^4$ & 11 & 12 & 13\\ 14 & 15 & 16 & 17 & 18 & $19^5$ & $20^6$\\};
\draw (m-1-3) circle (0.3cm);
\draw (m-1-5) circle (0.3cm);
\draw (m-2-1) circle (0.3cm);
\draw (m-2-2) circle (0.3cm);
\draw (m-2-4) circle (0.3cm);
\draw (m-3-6) circle (0.3cm);
\draw (m-3-7) circle (0.3cm);
\end{tikzpicture}\end{center}
Let $T$ be a RHT on $\la$ with content $(13,3,16,1,5,9,2)$.

\ytableausetup{nobaseline}
\boks{0.35}\[
T = \yt{11111111111113,22233333333333,333366,45556,55666,666,77}
\]
We slide the beads to the following positions:
\begin{center}\begin{tikzpicture}[shift up/.style = {transform canvas={yshift=.5mm}},shift down/.style = {transform canvas={yshift=-.5mm}},]\matrix (m) [matrix of nodes, column sep = .75cm, row sep = 0.2cm, nodes = {anchor = center}]
{$0^0$ & $1^4$ & $2^2$ & $3^1$ & $4^6$ & $5^3$ & $6^5$\\ 7 & 8 & 9 & 10 & 11 & 12 & 13\\ 14 & 15 & 16 & 17 & 18 & 19 & 20\\ };
\draw (m-1-1) circle (0.3cm);
\draw (m-1-2) circle (0.3cm);
\draw (m-1-3) circle (0.3cm);
\draw (m-1-4) circle (0.3cm);
\draw (m-1-5) circle (0.3cm);
\draw (m-1-6) circle (0.3cm);
\draw (m-1-7) circle (0.3cm);
\end{tikzpicture}\end{center}
We read off $\pi^T = (0,4,2,1,6,3,5)$.
\end{example}

An \textit{inversion} in a $r$-permutation $\pi =(\pi_0, \pi_1,\ldots, \pi_{r-1})$ is a pair $(i,j)$ such that $i<j$ and $\pi_i > \pi_j$. We denote the number of inversions in $\pi$ by $\inv(\pi)$. Define $\sgn(\pi) = (-1)^{\inv(\pi)}$. 
\begin{example}
Continuing from Example \ref{ex:pi-T}, we find that $\pi^T$ has inversions $(1,2), (1,3), (1,5), (2,3), (4,5), (4,6)$. Note that the coordinates of inversions are the zero-indexed positions in $\pi^T$. So, $\inv(\pi^T) = 5$, which means $\sgn(\pi^T) = -1$.
\end{example}

 For two $r$-permutations $\pi$ and $\rho$, define $\pi\cdot \rho$ to be the permutation whose $i$th entry is $\pi_{\rho_i}$ for $i = 0,1,\ldots, r-1$. 

\begin{lemma}[Thm. 7.31(a) in \cite{loehr-book}]\label{lem:sign-of-product}
Let $r$ be a positive integer. For two $r$-permutations $\pi$ and $\rho$, $\sgn(\pi\cdot \rho) = \sgn(\pi)\sgn(\rho)$.
\end{lemma}

%Sep 9 2026
\begin{prop}\label{prop:sign-of-abacus-is-rht}
For a rim-hook tableau $T$ with at most $r$ rows, $\sgn(T) = \sgn(\pi^T)$.
\end{prop}
\begin{proof}
We prove this statement via induction. Suppose the abacus of shape of $T$ has beads in positions $b_0, b_1,\ldots, b_{r-1}$. Suppose $T'$ is the RHT obtained after removing the largest labeled $t$-ribbon and the new positions of the beads are $b_0, b_1, \ldots, b_{i-k-1},b_i - t, b_{i-k}, b_{i-k+1},\ldots. b_{i-1}, b_{i+1},\ldots, b_{r-1}$ where we read off the permutation of this new abacus as $(0,1,\ldots,i-k-1, i, i-k, i-k+1, i-k+2, \ldots, i-1, i+1, \ldots r-1)$. This permutation has exactly $k$ inversions as $i$ appears to the left of $i-k, i-k+1, \ldots, i-1$. If we now proceed to remove ribbons from $T'$ to obtain the permutation $\pi^{T'}$, then $\pi^T = (0,1,\ldots, i, i-k, i-k+1, i-k+2, \ldots, i-1, i+1, \ldots r-1)\cdot \pi^{T'}$, which shows $\sgn(\pi^T) = (-1)^k \sgn(\pi^{T'})$. This recursion matches $\sgn(T) = (-1)^k \sgn(T')$ for the corresponding RHT. The sign $\sgn(\vn)$ of the empty RHT is 1, while the sign of $\pi^\vn = (0,1,\ldots, r-1)$ is also 1. Thus $\sgn(\pi^T) = \sgn(T)$ as $\sgn(\vn) = 1$.
\end{proof}
\begin{remark}
An alternate proof of Proposition \ref{prop:sign-of-abacus-is-rht} can be found in Lem. 4.4 of \cite{rht-konstant}.
\end{remark}
%\begin{example}Suppose we have a RHT such that $\pi^{T} = (0,1,2,3,6,4,5,7)$. Now, assume we remove a ribbon such that the bead in the second largest position slides to the fourth lowest position. This results in the permutation $\pi = (0,1,2,6,3,4,5,7)$. As we only swap the third and fourth position in $\pi$, $\pi^{T'} = (0,1,2,4,3,5,6,7)$. Thus, we find $\pi^T = \pi \cdot \pi^{T'}$.\end{example}
In an $r$-decomposable partition of $nr$, at each step, we remove an $r$-ribbon. The rim-hook tableau recording this removal has content $(r,r,\ldots, r)$ where $r$ occurs $n$ times. A content composition of $n > 1$ is called \textit{rectangular} if the set of values appearing in the composition is a singleton. The sign of the rim-hook tableau depends on the content, or equivalently, the choices of ribbons removed. The following lemma due to Littlewood is useful to us.
\begin{lemma}[Cor. 8 and 9 in \cite{white-bij}, Cor. 3.6 in \cite{tenner-involutions}, \cite{littlewood}]\label{lem:rectangular-content}
If $T$ and $T'$ have the same shape and the same rectangular content, then $\sgn(T) = \sgn(T')$.
\end{lemma}

Lemma \ref{lem:rectangular-content} implies that the sign of any RHT with content $(r,r,\ldots, r)$ with shape a $r$-decomposable partition $\mu$ depends only on $\mu$ and not the order we remove the $r$-ribbons in. We denote this unique sign by $\sgn_r(\mu)$. When $\mu$ has length strictly larger than $r$ or $\mu$ does not have an empty $r$-core, it is not $r$-decomposable. In this case, define $\sgn_r(\mu) = 0$.

\begin{lemma}\label{lem:sign-of-tableau}
Let $\mu$ be an $r$-decomposable partition and suppose beads in $\ab_r(\mu)$ occur in positions $b_0 < b_1 < \ldots < b_{r-1}$. Further suppose that $\pi$ is the $r$-permutation such that runner $i$ contains the bead $b_{\pi_i}$. Then $\sgn_r(\mu) =  \sgn(\pi)$.
\end{lemma}
\begin{proof}
We label each bead in position $b_i$ with the label $i$. We scan the runners from left to right, and we slide each bead in each runner one row up at a time until it reaches the top.  This sliding procedure corresponds to some RHT $T$ with rectangular content $(r,r,\ldots, r)$, and at the end of this procedure, the bead in runner $i$ with label $\pi_i$ occurs in position $i$ in $\ab_r(\vn)$. This means that the $r$-permutation $\pi^T$ obtained by reading the labels on the beads in row 0 is $\pi$. By Proposition \ref{prop:sign-of-abacus-is-rht}, $\sgn(\pi) = \sgn(T)$. As any RHT of rectangular content $(r,r,\ldots, r)$ and shape $\mu$ has the same sign, $\sgn_r(\mu) = \sgn(T) = \sgn(\pi)$.
\end{proof}

\noindent We now prove the $s$-expansion of $h_n[p_r]$ as outlined in Macdonald's book \cite{macd}.

\begin{prop}[I.8 Ex. 7 in \cite{macd}]\label{prop:hnpr-in-s}
If $n$ and $r$ are positive integers, then \[h_n[p_r] = \sum_{\mu\in \Par(nr)} \sgn_r(\mu) s_\mu.\]
\end{prop}
\begin{proof}
Recall that the generating function of the complete homogeneous symmetric functions~\cite[Eq. (2.5) in I.2]{macd} is given by 
\[
\sum_{n\geq 0} h_n(\x_*) z^n = 
\prod\limits_{i\geq 1} \left(1 - x_i z\right)^{-1}.
\]
Also recall that for $f\in \Sym$, $f[p_r]$ is computed by raising every variable in $f$ to the power $r$. Thus, $h_n[p_r](\x_*)$ is the coefficient of $z^{n}$ in the expansion $
\prod_{i\geq 1} \left(1 - x_i^r z\right)^{-1}$. Setting $z = t^r$ results in $h_n[p_r](\x_*)$ being the coefficient of $t^{rn}$ in the product $
\prod_{i\geq 1} \left(1 - x_i^r t^r\right)^{-1}$ which can be factorized using a primitive $r$th root of unity $\omega$ as
\[
\prod\limits_{i\geq 1} \left(1 - x_i^rt^r\right)^{-1} = \prod\limits_{i\geq 1} \prod\limits_{j=1}^r \left(1 - x_i \omega^jt\right)^{-1}.
\]

    The above expression lends itself to the Cauchy identity \cite[Eq. (4.3) in I.4]{macd} for Schur functions which is
    \[\prod\limits_{i,j\geq 1} \left(1 - x_iy_j\right)^{-1} = \sum\limits_{\lambda} s_{\lambda}(x_1, x_2, \ldots)s_{\lambda}(y_1, y_2, \ldots).\]
    
Setting $y_j= \omega^{j-1}t$ for $j = 1,2,\ldots, r$ and $y_j  = 0$ for $j > r$ yields
\[
 \prod\limits_{i\geq 1} \prod\limits_{j=1}^r \left(1 - x_i \omega^jt\right)^{-1} = \sum\limits_{\lambda} s_\lambda(x)s_{\lambda}(1, \omega, \omega^2, \ldots, \omega^{r-1})t^{|\lambda|},
\]
where the last equality follows from the homogeneity property of Schur functions.

By using the bialternant formulation of Schur functions \cite[Eq. (3.1) in I.3]{macd} and the Vandermonde determinant product formula~\cite[Thm. 12.65]{loehr-book}, we get
%\[
%s_{\lambda}(1, \omega, \ldots, \omega^{r-1}) = \frac{1}{\prod\limits_{1\leq j < k \leq r} \omega^{r-j} - \omega^{r-k}} \begin{vmatrix}
%    (\omega^0)^{\lambda_1 + r - 1} & (\omega^1)^{\lambda_1 + r - 1} & \ldots & (\omega^{r-1})^{\lambda_1 + r - 1}\\
%    (\omega^0)^{\lambda_2 + r - 2} & (\omega^1)^{\lambda_2 + r - 2} & \ldots & (\omega^{r-1})^{\lambda_2 + r - 2}\\
%    \vdots & \vdots& \ddots & \vdots\\
%    (\omega^0)^{\lambda_r} & (\omega^1)^{\lambda_r} & \ldots & (\omega^{r-1})^{\lambda_r}
%\end{vmatrix}
%\]
\[
s_{\lambda}(1, \omega, \ldots, \omega^{r-1}) = \prod\limits_{1\leq j < k \leq r} \frac{\omega^{\lambda_j + r - j} - \omega^{\lambda_k + r- k}}{\omega^{r-j} - \omega^{r-k}} \quad (*).
\]
The exponents of $\omega$ in the numerator $\omega^{\lambda_j + r - j} - \omega^{\lambda_k + r- k}$ are exactly the elements of $\mcH_r(\la)$, and thus correspond to the positions of beads in an abacus of $\la$ with $r$ beads. On the other hand, the exponents of $\omega$ in the denominator $\omega^{r-j} - \omega^{r-k}$ correspond to bead positions in an abacus of $\vn$ with $r$ beads. First suppose that $\la$ does not have an empty $r$-core. A partition has an empty $r$-core if and only if its abacus with $r$ runners and $r$ beads has one bead per runner. We deduce that the abacus with $r$ runners of $\la$ has multiple beads in some runner, and so there exists a pair of beads in positions $\lambda_{i_0}+r-i_0$ and $\lambda_{j_0}+r-j_0 $ separated by a multiple of $r$. Thus $\omega^{\lambda_{i_0}+r-i_0} = \omega^{\la_{j_0} + r - j_0}$. This makes the numerator of the product on the right hand side of $(*)$ equal to 0, and thus $s_\la(1, \omega, \ldots, \omega^{r-1})$ is equal to 0. As a Schur function with less than $\ell(\la)$ inputs is 0, $r\geq \ell(\la)$.

Now suppose that $\la$ does have an empty $r$-core. Label each bead in position $\la_i + r - i$ with the label $i$ for $i = 0,1,\ldots, r-1$. Suppose each bead in position $\la_i + r-i$ ends up in position $\pi_i$ where $0\leq \pi_i \leq r-1$ for $i = 0,1,\ldots, r-1$. Each factor in the product on the right hand side of $(*)$ is negative for exactly those pairs of beads for which $\la_i + r- i < \la_j + r- j$ and $\pi_i > \pi_j$. This is equivalent to changing the original permutation where $i$ appeared to the left of $j$ to the permutation where $i$ appears to right of $j$ as $\pi_i > \pi_j$. So each inversion in $\pi$ contributes a $-1$ to the product on the right hand side of $(*)$. From this, it follows that $s_\la(1, \omega, \ldots, \omega^{r-1}) = \sgn((\pi_0,\pi_1,\ldots, \pi_{r-1}))$ and $s_\la(1, \omega, \ldots, \omega^{r-1}) = \sgn_r(\la)$ by Lemma \ref{lem:sign-of-tableau}.
\end{proof}

\subsection{$h$-expansion of $h_n[p_r]$}\label{ssec:h-exp-hnpr}
In this subsection, we find the composition-indexed $h$-expansion of $h_n[p_r]$ by applying the Jacobi-Trudi identity on the $s$-expansion discussed in the previous section. We first recall the combinatorial interpretation of the Jacobi-Trudi identity due to E\u{g}ecio\u{g}lu and Remmel \cite{inv-kostka}. Recall that a special $k$-ribbon is a $k$-ribbon whose southwesternmost cell lies in the first column. A \textit{special rim-hook tableau} (SRHT) of shape $\mu$ and content $\beta$ is a rim-hook tableau of shape $\mu$ and content $\be$ such that the cells containing the same label form special rim-hooks. Equivalently, an SRHT is an RHT where every label appears at least once in the first column.
\begin{example}\label{ex:srht-ex1}
\ytableausetup{nobaseline}
\boks{0.3}
The SRHT \[T= \yt{111222,1222,2233,333,34,44}\] has shape $(6,4,4,3,2,2)$ and content $(4,8,6,3)$. We compute $\sgn(T) = (-1)^{1+2+2+1} = 1$.
\end{example}
Given a partition $\mu$ and content $\beta = (\be_1, \be_2,\ldots,\be_l)$, there is at most one SRHT of shape $\mu$ and content $\beta$. In other words, if an SRHT with a given partition shape and composition content exists, then it is unique. To see this, observe  that in order to remove a special $\be_l$-hook, we start at the bottom cell of the first column. From here, we take unit steps northwards or eastwards for a total of $\be_l$ steps such that the set of cells traversed forms a unique $\be_l$-ribbon. The steps we take are entirely dictated by the shape $\mu$, and so a special $\be_l$-ribbon either exists or does not. If we cannot form a $\be_l$-ribbon, we declare that no SRHT with shape $\mu$ and content $\be$ exists. Otherwise, we remove the $\be_l$-ribbon and continue this process for $\be_{l-1},\be_{l-2},\ldots, \be_1$. If at any point, we cannot remove a $\be_i$-ribbon, then we claim that the SRHT with the given shape and content does not exist. Otherwise, we have a unique SRHT $T_{\mu,\be}$. Define $t_{\mu,\be} = \sgn(T_{\mu,\be})$, and $t_{\mu,\be} = 0$ when no SRHT with shape $\mu$ and content $\be$ exists.
\begin{prop}[Composition refinement of Theorem 1 in \cite{inv-kostka}]\label{prop:h-exp-of-s}
For $\mu$ a partition of $n$, 
\[
s_\mu = \sum_{\be\in \Com(n)} t_{\mu,\be} h_\be.
\]
\end{prop}
\begin{proof}
Let $\mu$ and $\la$ be partitions. E\u{g}ecio\u{g}lu and Remmel~\cite{inv-kostka} show that the coefficient of $h_\mu$ in the partition-indexed $h$-expansion of the Schur function $s_\la$ can be computed by summing $\sgn(T)$ over all SRHT $T$ whose content compositions sort to $\mu$. Our claim is a composition-refined version of their result.
\end{proof}

Following Remark \ref{rem:special-hook-removal}, the removal of a special ribbon involves sliding a bead into the gap in the smallest position. We can also view this procedure as first choosing a bead, deleting said bead, and then placing a bead in the gap with the smallest index. We call this procedure \textit{special sliding}. If the chosen bead is to the left of the gap with the smallest index, then special sliding ensures that the bead stays in the same position, as deletion of said bead creates a gap of in the smallest index, and which results in the bead staying in its original position.

Let $r$ and $n$ be positive integers. In our subsequent discussions, all partitions $\mu$ are $r$-decomposable unless stated otherwise. 
%Recall that the abacus $\ab_r(\mu)$ contains $r$ runners with exactly one bead per runner. Let $\pi = (\pi_0, \pi_1,\ldots, \pi_{r-1})$ be a $r$-permutation. We construct a SRHT of shape $\mu$ by special sliding each bead in position $b_{\pi_i}$ in runner $\pi_i$ in $\ab_r(\la)$ to the position $i$ in $\ab_r(\vn)$. 

Define an \textit{$r$-weak composition of $n$} to be an $r$-tuple of non-negative integers that sum to $n$. Denote their set by $\WCom_r(n)$. For example $(2,0,3,3,0,1)$ is a 6-weak composition of $9$. Define $\comrev: \WCom_r(n)\to \Com(n)$ by sending a $r$-weak composition $w = (w_0,w_1,\ldots, w_{r-1})$ to the composition obtained by removing all zeros in $(w_{r-1},w_{r-2},\ldots, w_0)$ while preserving the order of entries. For example, $\comrev((2,0,3,3,0,1)) = (1,3,3,2)$. This operation is useful in mapping an $r$-weak composition to the content of some SRHT as we now explain.

Let $\mu$ be a partition of $n$ and $\be = (\be_1,\be_2,\ldots, \be_k)$ be a composition of $n$ such that the SRHT $T_{\mu,\be}$ of shape $\mu$ and content $\be$ exists. Suppose $r\geq \ell(\la)$. We associate an $r$-weak composition $w = (w_0, w_1, \ldots, w_{r-1})$ of $n$ with $T_{\mu,\be}$ as follows: remove the largest labeled special rim-hook of size $\beta_k$ while special sliding the corresponding bead on $\ab_r(\mu)$. As we special slide this bead, it must end in row 0, as there are only $r$ beads. If the new position of the bead is $i$, then let $w_i = \be_k$. Continue this process by removing special rim-hooks of sizes $\be_{k-1},\ldots, \be_1$ and assigning these sizes as values to the entries of $w$. For every unassigned $w_i$, set it equal to zero. Note that as at each step, the bead we special slide ends up in larger position than the bead slid in a previous step. This means that $w$ contains the entries $\be_k, \be_{k-1},\ldots, \be_1$ as a subsequence with rest of the entries zero. Conversely, suppose we start with an $r$-weak composition $w = (w_0, w_1, \ldots, w_{r-1})$. Try to construct $\mu$ such that $\ab_r(\mu)$ has beads in positions $b_i = i + w_i$. If two beads end up in the same position or in the same runner, define $t_{w} = 0$. Otherwise, the resulting abacus encodes a special rim-hook tableau of shape $\mu$ and content $\be = \comrev(w)$ which can be seen by running the bead sliding process in reverse as before. In this case, define $t_{w} = t_{\mu,\be}$. Denote the abacus with beads in positions $i + w_i$ by $\ab_r(w)$.
\begin{example}\label{srht-ex-abacus}
We showcase the above procedure by continuing with Example \ref{ex:srht-ex1}. Recall that in that example, we had a unique SRHT of shape $\la = (6,4,4,3,2,2)$ and content $\be = (4,8,6,3)$. As $\la$ is a 7-decomposable partition, we choose $r= 7$. We start with $\ab_7(\la)$ 
\begin{center}
\begin{tikzpicture}[shift up/.style = {transform canvas={yshift=.5mm}},shift down/.style = {transform canvas={yshift=-.5mm}},]\matrix (m) [matrix of nodes, column sep = .75cm, row sep = 0.2cm, nodes = {anchor = center}]
{0 & 1 & 2 & 3 & 4 & 5 & 6\\ 7 & 8 & 9 & 10 & 11 & 12 & 13\\ };
\draw (m-2-6) circle (0.3cm);
\draw (m-2-3) circle (0.3cm);
\draw (m-2-2) circle (0.3cm);
\draw (m-1-7) circle (0.3cm);
\draw (m-1-5) circle (0.3cm);
\draw (m-1-4) circle (0.3cm);
\draw (m-1-1) circle (0.3cm);
\end{tikzpicture}
\end{center} and special slide the bead in position $4$ $\be_4 = 3$ steps which gives
\begin{center}
\begin{tikzpicture}[shift up/.style = {transform canvas={yshift=.5mm}},shift down/.style = {transform canvas={yshift=-.5mm}},]\matrix (m) [matrix of nodes, column sep = .75cm, row sep = 0.2cm, nodes = {anchor = center}]
{0 & 1 & 2 & 3 & 4 & 5 & 6\\ 7 & 8 & 9 & 10 & 11 & 12 & 13\\ };
\draw (m-2-6) circle (0.3cm);
\draw (m-2-3) circle (0.3cm);
\draw (m-2-2) circle (0.3cm);
\draw (m-1-7) circle (0.3cm);
\draw (m-1-4) circle (0.3cm);
\draw (m-1-2) circle (0.3cm);
\draw (m-1-1) circle (0.3cm);
\end{tikzpicture}
\end{center} 
As the bead ends up in the first runner, we assign $w_1 = 3$. Note that we jump over 1 bead. Now, we special slide the bead in position 8 to position 2 $\be_3 = 6$ steps while jumping over 2 beads which gives
\begin{center}
\begin{tikzpicture}[shift up/.style = {transform canvas={yshift=.5mm}},shift down/.style = {transform canvas={yshift=-.5mm}},]\matrix (m) [matrix of nodes, column sep = .75cm, row sep = 0.2cm, nodes = {anchor = center}]
{0 & 1 & 2 & 3 & 4 & 5 & 6\\ 7 & 8 & 9 & 10 & 11 & 12 & 13\\ };
\draw (m-2-6) circle (0.3cm);
\draw (m-2-3) circle (0.3cm);
\draw (m-1-7) circle (0.3cm);
\draw (m-1-4) circle (0.3cm);
\draw (m-1-3) circle (0.3cm);
\draw (m-1-2) circle (0.3cm);
\draw (m-1-1) circle (0.3cm);
\end{tikzpicture}
\end{center}
We set $w_2 = 6$. We special slide the bead in position 12 to position 4 while jumping over 2 beads. As we end up in the 4th runner, $w_4 = 12 - 4 = 8$. We obtain the abacus
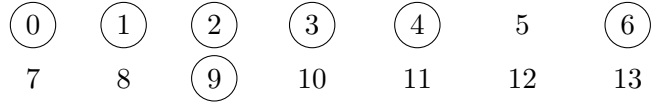
\begin{center}\begin{tikzpicture}[shift up/.style = {transform canvas={yshift=.5mm}},shift down/.style = {transform canvas={yshift=-.5mm}},]\matrix (m) [matrix of nodes, column sep = .75cm, row sep = 0.2cm, nodes = {anchor = center}]
{0 & 1 & 2 & 3 & 4 & 5 & 6\\ 7 & 8 & 9 & 10 & 11 & 12 & 13\\ };
\draw (m-2-3) circle (0.3cm);
\draw (m-1-7) circle (0.3cm);
\draw (m-1-5) circle (0.3cm);
\draw (m-1-4) circle (0.3cm);
\draw (m-1-3) circle (0.3cm);
\draw (m-1-2) circle (0.3cm);
\draw (m-1-1) circle (0.3cm);
\end{tikzpicture}\end{center}
Finally, we special slide the bead in position 9 to position 5 while jumping over 1 bead to obtain $\ab_7(\vn)$:
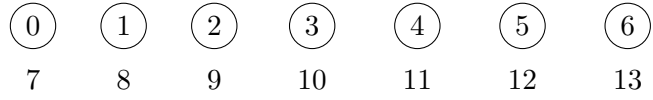
\begin{center}\begin{tikzpicture}[shift up/.style = {transform canvas={yshift=.5mm}},shift down/.style = {transform canvas={yshift=-.5mm}},]\matrix (m) [matrix of nodes, column sep = .75cm, row sep = 0.2cm, nodes = {anchor = center}]
{0 & 1 & 2 & 3 & 4 & 5 & 6\\ 7 & 8 & 9 & 10 & 11 & 12 & 13\\ };
\draw (m-1-7) circle (0.3cm);
\draw (m-1-6) circle (0.3cm);
\draw (m-1-5) circle (0.3cm);
\draw (m-1-4) circle (0.3cm);
\draw (m-1-3) circle (0.3cm);
\draw (m-1-2) circle (0.3cm);
\draw (m-1-1) circle (0.3cm);
\end{tikzpicture}\end{center}
Because we end up in runner 5, $w_5 = 4$. So, we find $(w_0, w_1, \ldots, w_7) = (0,3,6,0,8,4,0)$. The total number of jumps are $1+2+2+1 = 6$, and so $t_w = t_{\mu,\be} = (-1)^6 = 1$.
\end{example}

%Suppose we have the partition $\la$ and a list $w = (w_1, w_2, \ldots, w_k)$ of non-negative integers such that we obtain $\ab^k_r(\vn)$ from $\ab^k_r(\la)$ by special sliding beads successively from position $i$ in $\ab^k_r(\la)$ to position $i + w_i$ in $\ab^k_r(\la)$. For any list $\be = (\be_1, \be_2, \ldots, \be_l)$, define $\rev(\beta) = (\be_l, \be_{l-1},\ldots, \be_1)$. Define $\comrev(w)$ to be the composition obtained by removing all zeroes in $\rev(w)$ and preserving the order of the entries. For $i = 1,2,\ldots, k$, if in the $i$th step, we slide the bead by $\al_i$ positions, then with this bead sliding procedure we associate the composition $\rev(\alpha) = (\al_k,\al_{k-1}, \ldots, \al_1)$. The composition is reversed as the special ribbons are removed in the descending order of labels.
For an $r$-weak composition $w$, define $h_w = h_{w_0}h_{w_1}\ldots h_{w_{r-1}}$. Note that $h_w = h_{\comrev(w)}$ due to commutativity of multiplication and $h_0 = 1$. 
\begin{lemma}\label{lem:h-exp-of-s-w}
For an $r$-decomposable partition of $\mu$,
\[
s_\mu = \sum_{\substack{w\in \WCom_r(n)\\ \ab_r(w) = \ab_r(\mu)}} t_w h_w.
\]
\end{lemma}
\begin{proof}
In the $h$-expansion of $s_\mu$ (Proposition \ref{prop:h-exp-of-s}), every term $h_\beta$ that appears on the right hand side with a non-zero coefficient corresponds to an SRHT of shape $\mu$ and content $\beta$. With such an SRHT we associate an $r$-weak composition $w$ satisfying $t_w = t_{\mu,\be}$ as discussed before. The indexing compositions on the right hand side of Proposition \ref{prop:h-exp-of-s} are in bijection with $r$-weak compositions of $n$ satisfying $\ab_r(w) = \ab_r(\mu)$. This shows $s_\mu = \sum\limits_{\substack{w\in \WCom_r(n)\\ \ab_r(w) = \ab_r(\mu)}} t_w h_w$. 
\end{proof}

%27 Aug

We implement our strategy of finding the $h$-expansion of $h_n[p_r]$ using abaci. We first construct an abacus with $r$-runners for each term $s_\mu$ of the $s$-expansion of $h_n[p_r]$. Recall that these abaci have exactly one bead per runner and the coefficient of $s_\mu$ is $\sgn_r(\mu)$. Then for each such abacus, successively special sliding beads gives us an $r$-weak composition $w$ through the process described above. The sign of this process is $t_w$ and the resulting sign of $h_w$ in this $r$-weak composition-indexed $h$-expansion of $h_n[p_r]$ is thus computed as the product $\sgn_r(\mu)t_w$.

\begin{lemma}\label{lem:sign-pirho}
Suppose $\mu$ is an $r$-decomposable partition of $nr$ such that $\ab_r(\mu)$ has beads in positions $b_0 < b_1 <\ldots < b_{r-1}$. Suppose $\pi$ is the $r$-permutation such that runner $i$ contains the bead $b_{\pi_i}$. Let $\rho$ be a permutation and $w$ be an $r$-weak composition $w = (w_0, w_1,\ldots, w_{r-1})$ such that $S$ is the special rim-hook tableau obtained by special sliding the bead in runner $\rho_i$ $w_i$ steps to position $i$. Then $\pi^S = \pi \cdot \rho$ and $t_w = \sgn_r(\mu) \sgn(\rho)$.
\end{lemma}
\begin{proof}
As we remove special ribbons from $\mu$ according to $S$ to obtain $\vn$, we special slide the bead in runner $\rho_i$, that is, in position $b_{\pi_{\rho_i}}$ to position $i$ in $\ab_r(\vn)$. We read off the permutation formed by the labels of beads in row 0 as $\pi^S = (\pi_{\rho_0}, \pi_{\rho_1},\ldots, \pi_{\rho_{r-1}})$ which is equal to $\pi\cdot \rho$. By Lemma \ref{lem:sign-of-product}, $\sgn(\pi^S) = \sgn(\pi)\sgn(\rho)$. By Lemma \ref{lem:sign-of-tableau}, $\sgn(\pi) = \sgn_r(\mu)$ and as we obtain $w$ from $S$, by definition $t_w = \sgn(S)$. By Proposition \ref{prop:sign-of-abacus-is-rht}, $\sgn(S) = \sgn(\pi^S) = \sgn(\pi)\sgn(\rho)$.  This shows $t_w = \sgn_r(\mu) \sgn(\rho)$
\end{proof}
%Aug 28
Continuing the notation from Lemma \ref{lem:sign-pirho}, we express the permutation $\rho$ in terms of the $r$-weak composition $w$.
We special slide the bead in runner $\rho_i$ to position $i$ $w_i$ steps, which means that position $i + w_i$ is in runner $\rho_i$. All positions in a given runner in an abacus with $r$ runners are in the same equivalence class modulo $r$. Thus, $\rho_i$ is the remainder obtained upon dividing $i+w_i$ by $r$. For an $r$-weak composition $w$, define $w\% r$ to be the operation that replaces each entry of $w$ with its remainder upon division by $r$. For $r$-weak compositions $w$ and $w'$, define $w+w'$ as the $r$-weak composition whose $i$th entry is $w_i + w'_i$ for $i = 0,1,\ldots, r-1$. Define $\de_r = (0,1,2,\ldots, r-1)$. 

The resultant list $(w + \de_r)\%r$ has as its $i$th entry the index of the runner containing the bead the special slides to $i$. This is exactly the permutation $\rho$. If $\ab_r(w)$ is an abacus of a partition with exactly one bead per runner, then $(w + \de_r)\%r$ is a $r$-permutation. For an $r$-weak composition $w$, suppose $(w + \delta_r)\%r$ is a permutation. Then define
\[
\hat{s}_r(w) = (-1)^{\inv((w + \delta_r)\%r)}.
\]
On the other hand, if $(w + \de_r)\%r$ is not a permutation, we set $\hat{s}_r(w) = 0$.

\begin{example}
Choose $r = 7$. Consider the $7$-weak composition $w = (3,12,13,6,7,2,6)$ of 49. We find $w + \de_7 = (3,13,15,9,11,7,12)$ which gives $(w+\de_7)\%7 = (3,6,1,2,4,0,5)$. This resulting permutation contains 11 inversions and thus $\hat{s}_7(w) = (-1)^{11} = -1$.
 On the other hand, consider the 7-weak composition $w' = (8,15,0,16,2,3,5)$ of 49. We find $w' + \de_7 = (8,16,2,19,6,8,11)$ and $(w'+\de_7)\%7= (1,2,2,5,6,1,4)$. The latter is not a permutation and thus $\hat{s}_7(w') = 0$.
\end{example}

\begin{theorem}\label{thm:h-exp-of-hnpr}
Let $n,r\geq 1$ be integers. Then, \[h_n[p_r] = \sum_{w\in \WCom_r(nr)} \hat{s}_r(w) h_w.\]
\end{theorem}
\begin{proof}
Combining the expansions from Proposition \ref{prop:hnpr-in-s} and Lemma \ref{lem:h-exp-of-s-w}, we get
\[
h_n[p_r] = \sum_{\mu\in \Par(nr)} \sum_{\substack{w\in \WCom_r(n)\\ \ab_r(w) = \ab_r(\mu)}}  \sgn_r(\mu) t_w h_w. 
\]
We choose an $r$-weak composition $w$ such that $\ab_r(w) = \ab_r(\mu)$ contains exactly one bead in each runner. Suppose $\rho = (w+\de_r)\%r$ is the permutation such that $i+w_i$ lies in runner $\rho_i$. In Lemma \ref{lem:sign-pirho}, we found $t_w = \sgn_r(\mu) \sgn(\rho)$. Multiplying $\sgn_r(\mu)$ on both sides, we find $\sgn_r(\mu)t_w = \sgn(\rho)$ as $\sgn_r(\mu) \in \{-1,1\}$. It follows that $\sgn_r(\mu)t_w = \hat{s}_r(w)$. This shows that $h_n[p_r] = \sum_{w} \hat{s}_r(w) h_w$ where the sum is over all $r$-weak compositions of $nr$ for which $\ab_r(w)$ contains one bead per runner, or equivalently, $(w+\de_r)\%r$ is a permutation.

 If $(w+\de_r)\%r$ is not a permutation, then $j+w_j$ and $k+w_k$ are in the same runner. This means $\mu$ is not $r$-decomposable. Thus, both $\sgn_r(\mu)$ and $\hat{s}_r(w)$ are zero and the equality $\sgn_r(\mu)t_w = \hat{s}_r(w)$ still holds.
\end{proof}

We can now combine like terms to find a partition-indexed expansion of $h_n[p_r]$. First, let $\beta$ be a composition of length at most $r$. Define $\hat{s}_r(\be) = \sum_{w}\hat{s}_r(w)$ where the sum is over all $r$-weak compositions $w$ which yield $\be$ when all zeroes are removed and the order of entries is preserved. For instances, $\hat{s}_3((1,2)) = \hat{s}_3((1,2,0)) + \hat{s}_3((1,0,2)) + \hat{s}_3((0,1,2))$. For a composition $\beta$ of length strictly larger than $r$, define $\hat{s}_r(\be) = 0$. For a partition $\mu$, define $\hat{s}_r(\mu) = \sum_{\be} \hat{s}_r(\be)$ where the sum is over all compositions $\be$ that sort to $\mu$.
\begin{corollary}\label{cor:h-expansions}
For positive integers $n$ and $r$, $
h_n[p_r] = \sum_{\be\in \Com(nr)} \hat{s}_r(\be) h_\be.$
The coefficient of $h_\mu$ in the partition-indexed $h$-expansion of $h_n[p_r]$ is $\hat{s}_r(\mu)$.
\end{corollary}
\begin{proof}
As $h_0 = 1$, $h_w = h_{w'}$ holds for all $r$-weak compositions $w$ and $w'$ which give the same composition upon removal of zeros while preserving the order of entries. This gives the composition-indexed expansion $
h_n[p_r] = \sum_{\be\in \Com(nr)} \hat{s}_r(\be) h_\be.$
As multiplication is commutative and $h_0= 1$, $h_w = h_{w'}$ for $r$-weak compositions $w$ and $w'$ which have the same multisets of non-zero entries. We may represent the equivalence class of such $r$-weak compositions $w$ by a partition $\mu$ and all such $r$-weak compositions $w$ satisfy $h_\mu = h_w$. This shows $h_n[p_r] = \sum_{\mu\in \Par(nr)} \hat{s}_r(\mu) h_\mu$.
\end{proof}
\begin{example}\label{ex:hnpr-full}
We now find the coefficient of $h_{(4,1,1)}$ in the partition-indexed $h$-expansion of $h_2[p_3]$. As the length of $(4,1,1)$ is already $r = 3$, we do not pad it with any zeroes.
The rearrangements of $(4,1,1)$ are $w_1 =(4,1,1)$, $w_2 = (1,4,1)$ and $w_3 = (1,1,4)$. Adding $\delta_3$ to each gives $(4,2,3)$, $(1,5,3)$, and $(1,2,6)$, respectively. Upon reducing the entries modulo 3, we get the lists $(1,2,0)$, $(1,2,0)$ and $(1,2,0)$. For $i = 1,2,3$, $\hat{s}_3(w_i) = 1$, and so the coefficient of $h_{(4,1,1)}$ is 3.
\end{example}
\begin{example}
    The coefficient of $h_{(4,2)}$ in the $h$-expansion of $h_2[p_3]$ is $-3$ computed according to the following table:
    \begin{table}[H]
        \centering
        \begin{tabular}{c|c|c|c}
             $w$ & $w + \delta_3$ & $(w + \de_3)\%3$ & $\hat{s}_r((w+\delta_3)\% 3)$\\
             \hline
             $(4,2,0)$ & $(4,3,2)$ & $(1,0,2)$ & $(-1)^1 = -1$  \\
              $(4,0,2)$ & $(4,1,4)$ & $(1,0,1)$ & $0$  \\
              $(0,2,4)$ & $(0,3,6)$ & $(0,0,0)$ & $0$  \\
              $(0,4,2)$ & $(0,5,4)$ & $(0,2,1)$ & $(-1)^1 = -1$  \\
              $(2,0,4)$ & $(2,1,6)$ & $(2,1,0)$ & $(-1)^3 = -1$  \\$(2,4,0)$ & $(2,5,2)$ & $(2,2,2)$ & $0$ 
        \end{tabular}
    \end{table}
By adding up the signs, we see that the sum is $-3$ which is the desired coefficient.
\end{example}
\begin{example}\label{ex:hnpr-expansion}
For $n = 2, r = 3$, 
\[
h_2[p_3] = h_{2,2,2} - 3 h_{3,2,1} + 3 h_{3,3} + 3 h_{4,1,1} - 3 h_{4,2} - 3 h_{5,1} + 3 h_{6}.
\]
We also find
\begin{align*}
h_4[p_3] &= h_{4,4,4} - 3 h_{5,4,3} + 3 h_{5,5,2} + 3 h_{6,3,3} - 3 h_{6,4,2} - 3 h_{6,5,1}\\& + 3 h_{6,6} - 3 h_{7,3,2} + 6 h_{7,4,1} - 3 h_{7,5} + 3 h_{8,2,2} - 3 h_{8,3,1} - 3 h_{8,4} - 3 h_{9,2,1}\\& + 6 h_{9,3} + 3 h_{10,1,1} - 3 h_{10,2} - 3 h_{11,1} + 3 h_{12}.
\end{align*}
\end{example}
\begin{remark}\label{rem:multiple-of-r}
We observe that the coefficient computed in Example \ref{ex:hnpr-full} is $-3 = -r$. We also see that the coefficients in Example \ref{ex:hnpr-expansion} are divisible by $r$. We prove that the coefficients in the $h$-expansion of $h_n[p_r]$ are divisible by $r$ when $r$ is a prime in Section \ref{ssec:cyclic-invariance} by showing that if we cyclically shift the entries of $w$ to give $w'$, then $\hat{s}_r(w) = \hat{s}_r(w')$.
\end{remark}
\section{The $h^\ox$ and $e^\ox$ expansions of $H,E,E^+$}\label{sec:he-tensor-expa}
We now use Theorem \ref{thm:h-exp-of-hnpr} and Corollary \ref{cor:h-expansions} to prove results about bases of polysymmetric functions. We also provide combinatorial interpretations of the expansions in terms of objects we call polywrapping block tabloids.

\subsection{The transition matrices $\mcM_n(H,h^\otimes)$ and $\mcM_n(E,e^\otimes)$}\label{ssec:Hh-Ee}
We have already defined $\hat{s}_r$ for $r$-weak compositions, compositions, and partitions. Now we define $\hat{s}_r$ for ordered types and types. Recall an ordered type is an object $\de = 1^{\de|_1} 2^{\de|_2}\ldots$ where each $\de|_i$ is a composition and the size of $\delta$ is $|\delta| = \sum_i i\area(\delta|_i)$. A type is an ordered type where each $\de|_i$ is a partition. The function $\psort$ maps an ordered type $\de$ to a type $\psort(\de)$ by mapping each $\de|_i$ to $\sort(\de|_i)$. For an ordered type $\delta$, define $\hat{s}_r(\delta) = \prod_i \hat{s}_r(\de|_i)$. For a type $\tau$, define $\hat{s}_r(\tau) = \sum_{\delta} \hat{s}_r(\de)$, where the sum is over all ordered types $\de$ satisfying $\psort(\de) = \tau$.

Define an \textit{$r$-ordered type} $\rho$ to be an ordered type such that $\area(\rho|_i)$ is divisible by $r$.
%and $\ell(\rho|_i)\leq r$ for all $i\geq 1$. 
Denote the set of $r$-ordered types of size $n$ by $\XCom_r(n)$.

%29 Aug

\begin{prop}\label{prop:r-ordered-types-exp}
Let $d,r\geq 1$. Then
\[
H_{d^r} = \sum_{\substack{\rho\in \XCom_r(dr)\\ \ell(\rho|_i)\leq r, \forall i}} \hat{s}_r(\rho) h^\otimes_\rho.
\]
\end{prop}
\begin{proof}
From part (1) of Proposition \ref{prop:dr-expansions}, 
\[
H_{d^r} = \sum\limits_{\la\in \Par(d)} h_{m_1(\la)}[p_r] \otimes h_{m_2(\la)}[p_r] \otimes \cdots.
\]

Using the composition-indexed expansion of $h_n[p_r]$ from Corollary \ref{cor:h-expansions},
\[
h_n[p_r] =\sum_{\beta\in \Com(nr)} \hat{s}_r(\be) h_\be,
\]
we get
\begin{align*}
 H_{d^r} &= \sum\limits_{\la\in \Par(d)} \left(\sum_{\be^{(1)}\in \Com(m_1(\la)r)} \hat{s}_r(\be^{(1)}) h_{\be^{(1)}} \right) \otimes \left(\sum_{\be^{(2)}\in \Com(m_2(\la)r))} \hat{s}_r(\be^{(2)}) h_{\be^{(2)}} \right)\otimes\cdots\\ &\otimes\left(\sum_{\be^{(k)}\in \Com(m_k(\la)r)} \hat{s}_r(\be^{(k)}) h_{\be^{(k)}} \right)\otimes \cdots
\end{align*}
This simplifies to
\begin{align*}
H_{d^r} = & \sum\limits_{\la\in \Par(d)} \sum_{\be^{(1)}\in \Com(m_1(\la)r)} \sum_{\be^{(2)}\in \Com(m_2(\la)r)}\cdots\\ &\ldots \sum_{\be^{(k)}\in \Com(m_k(\la)r)}\ldots \left(\prod_i \hat{s}_r(\be^{(i)})\right) h_{\be^{(1)}}\otimes h_{\be^{(2)}}\otimes\cdots \: (*).
\end{align*}
With each term of the above expansion, we associate the $r$-ordered type $\rho$ that satisfies $\rho|_i = \be^{(i)}$. On the other hand, assume that $\rho$ is an $r$-ordered type of $dr$. This means that there exist non-negative integers $(a_i)_{\geq 1}$ such that $\area(\rho|_i) = a_i r$. Note $|\rho| = dr = \sum_{i\geq 1} i a_ir$. Thus $\rho$ is described by a solution to $\sum_{i\geq 1} i a_i = d$ which is exactly some partition $\la$ of $d$ such that $a_i = m_i(\la)$ for all $i$. So, $\rho$ indexes the term on the right for which $\be^{(i)} = \rho|_i$. The condition $\ell(\rho|_i)\leq r$ for all $i\geq 1$ follows from the observation that $\hat{s}_r(\rho) = 0$ if there exists a $k$ such that $\ell(\rho|k) > r$. This shows that the sum $H_{d^r} = \sum_{\rho\in \XCom_r(dr)} \hat{s}_r(\rho) h^\otimes_\rho$ where $\ell(\rho|_i) \leq r$ for all $i$ on the right hand side of $(*)$. 
\end{proof}
\begin{remark}
We may express the result of Proposition \ref{prop:r-ordered-types-exp} as sum over all ordered types, that is,
\[
H_{d^r} = \sum_{\de\in \XCom(dr)} \hat{s}_r(\de) h^\otimes_\de.
\]
First we may drop the condition that $\ell(\de_i)\leq r$ for all $i$ as $\hat{s}_r(\alpha) = 0$ for any composition with length greater than $\alpha$.
In Proposition \ref{prop:must-be-multiple-of-r} (proved later), we show that $\hat{s}_r(\alpha)$ is zero for any composition $\al$ for which $\area(\al)$ is not a multiple of $r$. For an ordered type $\delta$, $\hat{s}_r(\de) = \prod_{i\geq 1}\hat{s}_r(\delta|_i)$. If $\de$ is not an $r$-ordered type, then for some $k$, $\de|_k$ must be such that $\area(\de|_k)$ is not divisible by $r$. This means $\hat{s}_r(\de|_k) = 0$, which implies $\hat{s}_r(\de) = 0$.
\end{remark}

We now interpret the expansion in Proposition \ref{prop:r-ordered-types-exp} combinatorially using polywrapping block tabloids. Define a \textit{polywrapping $D$ of block $d^r$} of ordered content $\de\in \XCom_r(dr)$ as follows:
\begin{enumerate}
\item Express $\bdg(d^r)$ as a union of block diagrams of $(k\area(\delta|_k)/r)^r$ for $k = 1,2,\ldots$. We arrange the block diagrams in ascending order of $i$ from top to bottom.
\item For each $k$, we construct a $k$-wrapping $D_k$ of ordered content $\de|_k$ of $(k\area(\delta|_k)/r)^r$. This construction implies that rectangles of larger heights appear below rectangles of smaller heights in the polywrapping of $d^r$.
\item Define $\psgn(D) = \prod_{k\geq 1} \hat{s}_r(\delta|_k) = \hat{s}_r(\de)$.
\end{enumerate}
For a polywrapping $D$, define $\pocon(D)$ to be ordered type that is its ordered content.
\begin{example}\label{ex:polywrapping-block}
The following object is a polywrapping of $11^5$ with ordered content $1^{3,4,3}2^{8,3,3,1}3^{2,1,2}$:
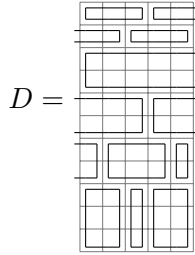
\begin{center}
$D = $ \raisebox{40.0pt}{\begin{tikzpicture}[scale =0.300,baseline=(current bounding box.north)]
\draw[step = 1.0, gray, thin] (0,0) grid (5,-11);
\draw (0.250,-0.250) -- (0.250,-0.750);
\draw (0.250,-0.250) -- (1.25,-0.250);
\draw (0.250,-0.750) -- (1.25,-0.750);
\draw (1.25,-0.250) -- (2.25,-0.250);
\draw (1.25,-0.750) -- (2.25,-0.750);
\draw (2.25,-0.250) -- (2.75,-0.250);
\draw (2.25,-0.750) -- (2.75,-0.750);
\draw (2.75,-0.250) -- (2.75,-0.750);
\draw (3.25,-0.250) -- (3.25,-0.750);
\draw (3.25,-0.250) -- (4.25,-0.250);
\draw (3.25,-0.750) -- (4.25,-0.750);
\draw (4.25,-0.250) -- (5.25,-0.250);
\draw (4.25,-0.750) -- (5.25,-0.750);
\draw (-0.25,-1.25) -- (0.25,-1.25);
\draw (-0.25,-1.75) -- (0.25,-1.75);
\draw (0.250,-1.25) -- (1.25,-1.25);
\draw (0.250,-1.75) -- (1.25,-1.75);
\draw (1.25,-1.25) -- (1.75,-1.25);
\draw (1.25,-1.75) -- (1.75,-1.75);
\draw (1.75,-1.25) -- (1.75,-1.75);
\draw (2.25,-1.25) -- (2.25,-1.75);
\draw (2.25,-1.25) -- (3.25,-1.25);
\draw (2.25,-1.75) -- (3.25,-1.75);
\draw (3.25,-1.25) -- (4.25,-1.25);
\draw (3.25,-1.75) -- (4.25,-1.75);
\draw (4.25,-1.25) -- (4.75,-1.25);
\draw (4.25,-1.75) -- (4.75,-1.75);
\draw (4.75,-1.25) -- (4.75,-1.75);
\draw (0.250,-2.25) -- (0.250,-3.75);
\draw (0.250,-2.25) -- (1.25,-2.25);
\draw (0.250,-3.75) -- (1.25,-3.75);
\draw (1.25,-2.25) -- (2.25,-2.25);
\draw (1.25,-3.75) -- (2.25,-3.75);
\draw (2.25,-2.25) -- (3.25,-2.25);
\draw (2.25,-3.75) -- (3.25,-3.75);
\draw (3.25,-2.25) -- (4.25,-2.25);
\draw (3.25,-3.75) -- (4.25,-3.75);
\draw (4.25,-2.25) -- (5.25,-2.25);
\draw (4.25,-3.75) -- (5.25,-3.75);
\draw (-0.25,-4.25) -- (0.25,-4.25);
\draw (-0.25,-5.75) -- (0.25,-5.75);
\draw (0.250,-4.25) -- (1.25,-4.25);
\draw (0.250,-5.75) -- (1.25,-5.75);
\draw (1.25,-4.25) -- (2.25,-4.25);
\draw (1.25,-5.75) -- (2.25,-5.75);
\draw (2.25,-4.25) -- (2.75,-4.25);
\draw (2.25,-5.75) -- (2.75,-5.75);
\draw (2.75,-4.25) -- (2.75,-5.75);
\draw (3.25,-4.25) -- (3.25,-5.75);
\draw (3.25,-4.25) -- (4.25,-4.25);
\draw (3.25,-5.75) -- (4.25,-5.75);
\draw (4.25,-4.25) -- (5.25,-4.25);
\draw (4.25,-5.75) -- (5.25,-5.75);
\draw (-0.25,-6.25) -- (0.25,-6.25);
\draw (-0.25,-7.75) -- (0.25,-7.75);
\draw (0.250,-6.25) -- (0.750,-6.25);
\draw (0.250,-7.75) -- (0.750,-7.75);
\draw (0.750,-6.25) -- (0.750,-7.75);
\draw (1.25,-6.25) -- (1.25,-7.75);
\draw (1.25,-6.25) -- (2.25,-6.25);
\draw (1.25,-7.75) -- (2.25,-7.75);
\draw (2.25,-6.25) -- (3.25,-6.25);
\draw (2.25,-7.75) -- (3.25,-7.75);
\draw (3.25,-6.25) -- (3.75,-6.25);
\draw (3.25,-7.75) -- (3.75,-7.75);
\draw (3.75,-6.25) -- (3.75,-7.75);
\draw (4.25,-6.25) -- (4.25,-7.75);
\draw (4.25,-6.25) -- (4.75,-6.25);
\draw (4.25,-7.75) -- (4.75,-7.75);
\draw (4.75,-6.25) -- (4.75,-7.75);
\draw (0.250,-8.25) -- (0.250,-10.8);
\draw (0.250,-8.25) -- (1.25,-8.25);
\draw (0.250,-10.8) -- (1.25,-10.8);
\draw (1.25,-8.25) -- (1.75,-8.25);
\draw (1.25,-10.8) -- (1.75,-10.8);
\draw (1.75,-8.25) -- (1.75,-10.8);
\draw (2.25,-8.25) -- (2.25,-10.8);
\draw (2.25,-8.25) -- (2.75,-8.25);
\draw (2.25,-10.8) -- (2.75,-10.8);
\draw (2.75,-8.25) -- (2.75,-10.8);
\draw (3.25,-8.25) -- (3.25,-10.8);
\draw (3.25,-8.25) -- (4.25,-8.25);
\draw (3.25,-10.8) -- (4.25,-10.8);
\draw (4.25,-8.25) -- (4.75,-8.25);
\draw (4.25,-10.8) -- (4.75,-10.8);
\draw (4.75,-8.25) -- (4.75,-10.8);
\end{tikzpicture}}\end{center}
We construct $D$ as follows. Express the block diagram of $11^5$ as a union of block diagrams of $(1\cdot 10/5)^5$, $(2\cdot 15/5)^5$ and $(3\cdot 5/5)^5$. Then the polywrapping is obtained by independently constructing the 1-wrapping of $2^5$, the 2-wrapping of $6^5$ and the 3-wrapping of $3^5$. We compute $\hat{s}_5((3,4,3)) = 2$, $\hat{s}_5((8,3,3,1)) = -1$, and $\hat{s}_5((2,1,2)) = 2$. This gives $\psgn(D) = 2 \cdot -1 \cdot 2 = -4$.
\end{example}

We call a polywrapping \textit{$r$-bounded} if the number of rectangles of each height is at most $r$. In Example \ref{ex:polywrapping-block}, we have shown a 5-bounded polywrapping.
\begin{remark}
We have defined polywrappings of $d^r$ and $r$-bounded polywrappings of $d^r$ separately even though at this point only use the latter definition. The former notion of polywrappings (without any bound on number of rectangles) plays an important role in Section \ref{ssec:He-Eh}.
\end{remark}
\begin{prop}\label{prop:H-dr-in-h-combo}
For integers $d,r \geq 1$,
\[
H_{d^r} = \sum_{D} \psgn(D) h^\otimes_{\pocon(D)}.
\]
where the sum is over all $r$-bounded polywrappings $D$ of the block $d^r$.
\end{prop}
\begin{proof}
For each $r$-ordered type $\rho$ with $\ell(\rho_i)\leq i$ for all $i$, we construct a polywrapping $D$ of block $d^r$ with ordered content $\rho$. From our definitions, it follows that $\psgn(D) = \hat{s}_r(\rho)$ and that $D$ is $r$-bounded as $\ell(\rho_i)\leq i$ for all $i$. So we may rewrite the right hand side of Proposition \ref{prop:r-ordered-types-exp} as the sum $\sum_{D} \psgn(D) h^\ox_{\pocon(D)}$ where the sum is over all $r$-bounded polywrappings $D$ arising from $r$-ordered types $\rho$ as discussed. To complete the proof, we need to show that each $r$-bounded polywrapping of $d^r$ produces one term on the right hand side of Proposition \ref{prop:r-ordered-types-exp}. 

Now assume we have an $r$-bounded polywrapping $D$ of $d^r$ which can be expressed as a union of $k$-wrappings. We read the rectangles from left to right. The list of lengths of rectangle of a given height $k\geq 1$ read in this prescribed order form a composition, which we call $\be^{(k)}$. Then, the ordered content of $D$ is the ordered type $\de= 1^{\be^{(1)}} 2^{\be^{(2)}}\ldots$. Note that as the union of rectangles of a given height cover an array of width $r$, each composition $\be^{(k)}$ has size divisible by $r$. As the number of rectangles in $D_k$ is equal to the length of the composition $\be^{(k)}$, the length of $\be^{(k)}$ is at most $r$.  Thus, $\de$ is an $r$-ordered type satisfying $\ell(\rho|_k)\leq r$ for all $k$. So the polywrapping $D$ gives rise to the term $h^\ox_\de$. It follows that $\psgn(D) = \hat{s}_r(\de)$.
\end{proof}

Let $\sigma$ and $\tau$ be types of a positive integer $n$. Define a \textit{bounded polywrapping block tabloid} $T$ of shape $\si = d_1^{r_1}d_2^{r_2}\ldots d_k^{r_k}$ and content $\tau = e_1^{s_1}e_2^{s_2}\ldots e_l^{s_l}$ as a list $(D^{(1)},D^{(2)},\ldots, D^{(k)})$ such that each $D^{(i)}$ is a $r_i$-bounded polywrapping of $d_i^{r_i}$ of content $\de^{(i)}$ and the multiset union $
\bigcup_{i=1}^k\psort(\ocon(D^{(i)}))$ equals $\tau$.

Define $\psgn(T) = \prod_{i=1}^k \psgn(D^{(k)}) = \prod_{i=1}^k \hat{s}_r(\de^{(i)})$.

Informally, we use rectangles of heights $e_i$ and length $s_i$ for $i = 1,2,\ldots, l$ respectively to construct simultaneous $r_i$-bounded polywrappings $D^{(i)}$ of blocks $d_i^{r_i}$ for all $i = 1,2,\ldots, k$.

\begin{example}\label{ex:pw-block-tabloid}
The following  is a bounded polywrapping block tabloid of shape $5^4 5^4 3^3$ with content $3^{3,1} 2^{3,3,2} 1^{6,3,3,3,2,2,2}$.
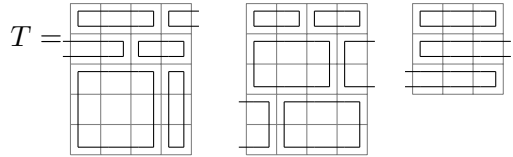
\begin{center}
$T = $\raisebox{16.0pt}{\begin{tikzpicture}[scale =0.400,baseline=(current bounding box.north)]
\draw[step = 1.0, gray, thin] (0,0) grid (4,-5);
\draw (0.250,-0.250) -- (0.250,-0.750);
\draw (0.250,-0.250) -- (1.25,-0.250);
\draw (0.250,-0.750) -- (1.25,-0.750);
\draw (1.25,-0.250) -- (2.25,-0.250);
\draw (1.25,-0.750) -- (2.25,-0.750);
\draw (2.25,-0.250) -- (2.75,-0.250);
\draw (2.25,-0.750) -- (2.75,-0.750);
\draw (2.75,-0.250) -- (2.75,-0.750);
\draw (3.25,-0.250) -- (3.25,-0.750);
\draw (3.25,-0.250) -- (4.25,-0.250);
\draw (3.25,-0.750) -- (4.25,-0.750);
\draw (-0.25,-1.25) -- (0.25,-1.25);
\draw (-0.25,-1.75) -- (0.25,-1.75);
\draw (0.250,-1.25) -- (1.25,-1.25);
\draw (0.250,-1.75) -- (1.25,-1.75);
\draw (1.25,-1.25) -- (1.75,-1.25);
\draw (1.25,-1.75) -- (1.75,-1.75);
\draw (1.75,-1.25) -- (1.75,-1.75);
\draw (2.25,-1.25) -- (2.25,-1.75);
\draw (2.25,-1.25) -- (3.25,-1.25);
\draw (2.25,-1.75) -- (3.25,-1.75);
\draw (3.25,-1.25) -- (3.75,-1.25);
\draw (3.25,-1.75) -- (3.75,-1.75);
\draw (3.75,-1.25) -- (3.75,-1.75);
\draw (0.250,-2.25) -- (0.250,-4.75);
\draw (0.250,-2.25) -- (1.25,-2.25);
\draw (0.250,-4.75) -- (1.25,-4.75);
\draw (1.25,-2.25) -- (2.25,-2.25);
\draw (1.25,-4.75) -- (2.25,-4.75);
\draw (2.25,-2.25) -- (2.75,-2.25);
\draw (2.25,-4.75) -- (2.75,-4.75);
\draw (2.75,-2.25) -- (2.75,-4.75);
\draw (3.25,-2.25) -- (3.25,-4.75);
\draw (3.25,-2.25) -- (3.75,-2.25);
\draw (3.25,-4.75) -- (3.75,-4.75);
\draw (3.75,-2.25) -- (3.75,-4.75);
\end{tikzpicture}} \quad \raisebox{16.0pt}{\begin{tikzpicture}[scale =0.400,baseline=(current bounding box.north)]
\draw[step = 1.0, gray, thin] (0,0) grid (4,-5);
\draw (0.250,-0.250) -- (0.250,-0.750);
\draw (0.250,-0.250) -- (1.25,-0.250);
\draw (0.250,-0.750) -- (1.25,-0.750);
\draw (1.25,-0.250) -- (1.75,-0.250);
\draw (1.25,-0.750) -- (1.75,-0.750);
\draw (1.75,-0.250) -- (1.75,-0.750);
\draw (2.25,-0.250) -- (2.25,-0.750);
\draw (2.25,-0.250) -- (3.25,-0.250);
\draw (2.25,-0.750) -- (3.25,-0.750);
\draw (3.25,-0.250) -- (3.75,-0.250);
\draw (3.25,-0.750) -- (3.75,-0.750);
\draw (3.75,-0.250) -- (3.75,-0.750);
\draw (0.250,-1.25) -- (0.250,-2.75);
\draw (0.250,-1.25) -- (1.25,-1.25);
\draw (0.250,-2.75) -- (1.25,-2.75);
\draw (1.25,-1.25) -- (2.25,-1.25);
\draw (1.25,-2.75) -- (2.25,-2.75);
\draw (2.25,-1.25) -- (2.75,-1.25);
\draw (2.25,-2.75) -- (2.75,-2.75);
\draw (2.75,-1.25) -- (2.75,-2.75);
\draw (3.25,-1.25) -- (3.25,-2.75);
\draw (3.25,-1.25) -- (4.25,-1.25);
\draw (3.25,-2.75) -- (4.25,-2.75);
\draw (-0.25,-3.25) -- (0.25,-3.25);
\draw (-0.25,-4.75) -- (0.25,-4.75);
\draw (0.250,-3.25) -- (0.750,-3.25);
\draw (0.250,-4.75) -- (0.750,-4.75);
\draw (0.750,-3.25) -- (0.750,-4.75);
\draw (1.25,-3.25) -- (1.25,-4.75);
\draw (1.25,-3.25) -- (2.25,-3.25);
\draw (1.25,-4.75) -- (2.25,-4.75);
\draw (2.25,-3.25) -- (3.25,-3.25);
\draw (2.25,-4.75) -- (3.25,-4.75);
\draw (3.25,-3.25) -- (3.75,-3.25);
\draw (3.25,-4.75) -- (3.75,-4.75);
\draw (3.75,-3.25) -- (3.75,-4.75);
\end{tikzpicture}}\quad \raisebox{16.0pt}{\begin{tikzpicture}[scale =0.400,baseline=(current bounding box.north)]
\draw[step = 1.0, gray, thin] (0,0) grid (3,-3);
\draw (0.250,-0.250) -- (0.250,-0.750);
\draw (0.250,-0.250) -- (1.25,-0.250);
\draw (0.250,-0.750) -- (1.25,-0.750);
\draw (1.25,-0.250) -- (2.25,-0.250);
\draw (1.25,-0.750) -- (2.25,-0.750);
\draw (2.25,-0.250) -- (2.75,-0.250);
\draw (2.25,-0.750) -- (2.75,-0.750);
\draw (2.75,-0.250) -- (2.75,-0.750);
\draw (0.250,-1.25) -- (0.250,-1.75);
\draw (0.250,-1.25) -- (1.25,-1.25);
\draw (0.250,-1.75) -- (1.25,-1.75);
\draw (1.25,-1.25) -- (2.25,-1.25);
\draw (1.25,-1.75) -- (2.25,-1.75);
\draw (2.25,-1.25) -- (3.25,-1.25);
\draw (2.25,-1.75) -- (3.25,-1.75);
\draw (-0.25,-2.25) -- (0.25,-2.25);
\draw (-0.25,-2.75) -- (0.25,-2.75);
\draw (0.250,-2.25) -- (1.25,-2.25);
\draw (0.250,-2.75) -- (1.25,-2.75);
\draw (1.25,-2.25) -- (2.25,-2.25);
\draw (1.25,-2.75) -- (2.25,-2.75);
\draw (2.25,-2.25) -- (2.75,-2.25);
\draw (2.25,-2.75) -- (2.75,-2.75);
\draw (2.75,-2.25) -- (2.75,-2.75);
\end{tikzpicture}}
\end{center}
We compute $\psgn(T) = (1\cdot -1)(-2\cdot 1)(3) = 6$.
\end{example}

Define $\pwrap(d^r)$ as the set of all $r$-bounded polywrappings of the block $d^r$. We may write Proposition \ref{prop:H-dr-in-h-combo} as $H_{d^r} = \sum_{D\in \pwrap(d^r)} \psgn(D) h^\otimes_{\pocon(D)}$.
\begin{theorem}\label{thm:H-in-htensor}
Let $\sigma$ and $\tau$ be types of an integer $n\geq 1$. The coefficient of $h^\otimes_\tau$ in the $h^\otimes$-expansion of $H_\sigma$ is $\sum_T \psgn(T)$ where the sum is over all bounded polywrapping block tabloids $T$ of shape $\si$ and content $\tau$.
\end{theorem}
\begin{proof}
From Proposition \ref{prop:H-dr-in-h-combo},
\begin{align*}
H_{d_1^{r_1}} H_{d_2^{r_2}}\ldots H_{d_k^{r_k}} = 
&\left( \sum_{D^{(1)}\in \pwrap(d_1^{r_1})} \psgn(D^{(1)}) h^\otimes_{\pocon(D^{(1)})}\right)\\
&\left( \sum_{D^{(2)}\in \pwrap(d_2^{r_2})} \psgn(D^{(2)}) h^\otimes_{\pocon(D^{(2)})}\right)\ldots\\
&\left( \sum_{D^{(k)}\in \pwrap(d_k^{r_k})} \psgn(D^{(k)}) h^\otimes_{\pocon(D^{(k)})}\right).
\end{align*}
For a fixed $\tau$, the coefficient of $h^\ox_\tau$ is $\sum_{\vec{D}} \prod_{i=1}^k \psgn(D^{(i)})$, where the sum over all lists $\vec{D} = (D^{(1)}, D^{(2)}, \ldots, D^{(k)})$ satisfying $\bigcup_{i=1}^k\psort(\ocon(D^{(i)})) = \tau$ as multisets. As each $D^{(i)}$ is an $r_i$-bounded polywrapping of $d_i^{r_i}$, the list $(D^{(1)}, D^{(2)}, \ldots, D^{(k)})$ is a bounded polywrapping block tabloid of shape $\si$ and content $\tau$. So the sum in the coefficient is indexed by bounded polywrapping block tabloids $T$ of shape $\si$ and content $\tau$, and the summands are $\psgn(T)$. On the other hand, if we are given a bounded polywrapping block tabloid $T$ of shape $\si$ and content $\tau$, then we may define $D^{(i)}$ as the polywrapping of the block $d_i^{r_i}$, and thus each bounded polywrapping block tabloid appears as an index of a term on the right hand side of the above product. This shows that the coefficient of $h^\otimes_\tau$ is $\sum_T \psgn(T)$ where the sum runs over the set of all bounded polywrapping block tabloids $T$.
\end{proof}

Recall that for a type $\tau$, $\area(\tau) = \sum_{i\geq 1} \area(\tau|_i)$. Also recall Corollary \ref{cor:Omega-of-htensor} which states that $\Omega(h^\otimes_\tau) = (-1)^{\area(\tau)} e^\otimes_{\tau}$.
\begin{prop}\label{prop:Edr-in-etensor}
Let $d\geq 0$ and $r\geq 1$. Then
\[
E_{d^r} = \sum_{\substack{\rho\in \XCom_r(dr)\\ \ell(\rho|_i)\leq r, \forall i}} (-1)^{\area(\rho)}\hat{s}_r(\rho) e^\otimes_\rho.
\]
\end{prop}
\begin{proof}
The statement follows by applying $\Omega$ to Proposition \ref{prop:r-ordered-types-exp} and using Corollary \ref{cor:Omega-of-htensor}.
\end{proof}
\begin{theorem}\label{thm:E-in-etensor}
Let $\sigma$ and $\tau$ be types of an integer $n\geq 0$. The coefficient of $e^\otimes_\tau$ in the $e^\otimes$-expansion of $E_\sigma$ is $(-1)^{\area(\tau)}\sum_T \psgn(T)$ where the sum is over all polywrapping block tabloids of shape $\si$ and content $\tau$.
\end{theorem}
\begin{proof}
By definition, $\Omega(H_\si) = E_\si$. The coefficient of $e^\otimes_\tau$ in $E_\si$ as claimed follows from an application of $\Omega$ on the $h^\otimes_\tau$-expansion of $H_\si$ (Theorem \ref{thm:H-in-htensor}) and Corollary \ref{cor:Omega-of-htensor}.
\end{proof}

%30 Aug
\subsection{The transition matrices $\mcM_n(H,e^\ox)$ and $\mcM_n(E, h^\ox)$}\label{ssec:He-Eh}
We use the composition-indexed $e$-expansion of $h_\be$ for $\be\in \Com$ to derive the transition matrices $\mcM_n(H,e^\ox)$ and $\mcM_n(E, h^\ox)$. We first recall a partial order on compositions. We say that a composition $\be$ \textit{refines} $\al$, denoted $\be\leq \al$, if $\be$ is obtained by replacing each entry $\al_i$ with entries of some composition of $\al_i$ in order. For instance, $(2,1,2,3,2,1)$ refines $(5,3,3)$ as we replace $5$ by $(2,1,2)$, $3$ by $(3)$, and $3$ by $(2,1)$. Equivalently, we say that $\al$ \textit{coarsens} $\be$. The maximum element in the partially ordered set $(\Com(n),\leq)$ is $(n)$ and the minimum element is $(1,1,\ldots, 1)$.

For ordered types $\gamma$ and $\delta$ of $n$, \textit{$\gamma$ refines $\de$}, written $\gamma\leq \de$, if $\gamma|_i \leq \delta|_i$ for all $i\geq 1$.
\begin{lemma}\label{lem:hten-to-eten}
For an ordered type $\de$, 
\begin{enumerate}
\item $h^\otimes_\de =\sum\limits_{\ga\leq \de} (-1)^{\area(\ga)- \ell(\ga)} e^\otimes_\ga$.
\item $e^\otimes_\de =\sum\limits_{\ga\leq \de} (-1)^{\area(\ga)- \ell(\ga)} h^\otimes_\ga$.
\end{enumerate}
\end{lemma}
\begin{proof}
From \cite[Pg. 26]{qsym-book}, we know that for all compositions $\al$,
\[
h_\al = \sum_{\be\leq \al} (-1)^{n-\ell(\be)} e_\be \quad (*).
\]
If we apply $(*)$ to $h^\ox_\de = h_{\de|_1}\otimes h_{\de|_2}\otimes\ldots$, we get 
\[
h^\ox_\de = \left(  \sum_{\gamma^{(1)}\leq \de|_1} (-1)^{\area(\ga^{(1)}) - \ell(\ga^{(1)})} e_{\ga^{(1)}}\right)\otimes \left(  \sum_{\gamma^{(2)}\leq \de|_2} (-1)^{\area(\ga^{(2)}) - \ell(\ga^{(2)})} e_{\ga^{(2)}}\right) \ox\cdots.
\]
Multiplying out the terms gives 
\[
h^\ox_\de = \sum_{\gamma^{(1)}\leq \de|_1} \sum_{\gamma^{(2)}\leq \de|_2} \ldots \left(\prod_{i\geq 1} (-1)^{\area(\ga^{(i)}) - \ell(\ga^{(i)})}\right)  e_{\ga^{(1)}} \otimes  e_{\ga^{(2)}}\otimes \cdots.
\]
Define an ordered type $\gamma$ such that $\ga|_i =\ga^{(i)}$. Then $\gamma\leq \delta$ by construction, and $e_\ga^\otimes$ has the coefficient $(-1)^{\area(\ga) - \ell(\ga)}$ as $\sum_{i\geq 1} \area(\ga^{(i)}) = \sum_{i\geq 1} \area(\ga|_i) = \area(\ga)$ and $\sum_{i\geq 1} \ell(\ga^{(i)}) = \sum_{i\geq 1} \ell(\ga|_i)= \ell(\ga)$. Thus, each term in the above sum is of the form $(-1)^{\area(\ga) - \ell(\ga)}e_\ga^\ox$ for some $\ga\leq \de$. To see that all $\gamma$ that refine $\de$ appear in the above sum, observe that the condition $\ga\leq \de$ implies $\gamma|_i \leq \delta|_i$ for all $i\geq 1$, and thus $\gamma$ indexes the term $e_{\ga^{(1)}} \otimes  e_{\ga^{(2)}}\otimes \cdots$ for which $\gamma|_i = \ga^{(i)}$. These term $e^\ox_\ga$ appears with the coefficient $(-1)^{\area(\ga)- \ell(\ga)}$ as discussed.

If we apply $\Omega$ on both sides of (1) and use Proposition \ref{cor:Omega-of-htensor}, we find
\[
(-1)^{\area(\de)} h^\otimes_\de = \sum_{\ga\leq \de} (-1)^{\area(\ga)-\ell(\ga)}(-1)^{\area(\ga)} e^\otimes_\ga.
\]
As $\ga$ refines $\de$, $\gamma|_i \leq \delta|_i$ for all $i\geq 1$, which means that both $\gamma|_i$ and $\delta|_i$ are compositions of the same integers. Thus $\area(\ga|_i) = \area(\de|_i)$ for all $i\geq 1$. and $\area(\ga) = \area(\de)$. Rearranging signs proves (2).
\end{proof}

We now use Lemma \ref{lem:hten-to-eten} to find the $e^\ox$-expansion of $H_{d^r}$.

\begin{prop}\label{prop:H-in-eten}
For integers $d,r\geq1$, 
\[
H_{d^r} = \sum_{\gamma\in \XCom_r(dr)} \left[ (-1)^{\area(\ga)-\ell(\ga)} \left(\sum_{\rho \geq \gamma}\hat{s}_r(\rho)\right)\right] e^\otimes_\gamma.
\]
\end{prop}
\begin{proof}
We begin with the expansion in Proposition \ref{prop:r-ordered-types-exp}:
\[
H_{d^r} = \sum_{\substack{\rho\in \XCom_r(dr)\\ \ell(\rho|_i)\leq r, \forall i}} \hat{s}_r(\rho) h^\otimes_\rho.
\]
and use Lemma \ref{lem:hten-to-eten} to write
\[
H_{d^r} = \sum_{\substack{\rho\in \XCom_r(dr)\\ \ell(\rho|_i)\leq r, \forall i}}\hat{s}_r(\rho) \sum_{\gamma\leq \rho} (-1)^{\area(\ga) - \ell(\ga)} e^\ox_\ga.
\]
Note that if $\gamma\leq \rho$ and $\rho$ is an $r$-ordered type, then $\area(\ga|_i) = \area(\rho|_i)$ is divisible by $r$ for all $i\geq 1$. Thus, $\ga$ is also an $r$-ordered type. However, when a composition $\be$ refines $\al$, then the length of $\be$ is at least the length of $\al$. Thus, the length of $\gamma$ is not bounded by $r$. So we may swap the sums to write
\[
H_{d^r} = \sum_{\ga\in \XCom_r(dr)}  \left[ \left(\sum_{\substack{\rho \geq \gamma\\ \ell(\rho|_i)\leq r, \forall i}} \hat{s}_r(\rho) \right) (-1)^{\area(\ga) - \ell(\ga)}\right] e^\ox_\ga.
\]
By definition, $\hat{s}_r(\rho) = 0$ when $\ell(\rho|_k) > r$ for some $k\geq 1$. So we may write the above sum as
\[
H_{d^r} = \sum_{\ga\in \XCom_r(dr)}  \left[ \left(\sum_{\rho \geq \gamma} \hat{s}_r(\rho) \right) (-1)^{\area(\ga) - \ell(\ga)}\right] e^\ox_\ga.
\]
\end{proof}
\renewcommand{\ast}{^*}

We now define a \textit{labeled polywrapping of the block $d^r$ of ordered content $\ga\in \XCom_r(dr)$}. Start with a polywrapping $D$ of $d^r$. Define $D_k$ to be the $k$-wrapping of some contiguous rows of $d^r$. For all $k\geq 1$, label the rectangles of height $k$ with labels from $\{1,2,\ldots, r\}$ such that when we read through the rectangles involved in the $k$-wrapping of $D_k$ from left to right, the labels increase weakly starting at 1 and are consecutive. We emphasize that not all values in $\{1,2,\ldots, r\}$ need to appear as labels. We call the labeled polywrapping thus obtained $D\ast$. We associate with $D\ast$ an ordered type $\de$, called the \textit{ordered type of the labeling}, as follows: for each $k$, define $\delta|_k$ as the composition whose $i$th entry is the total number of boxes contained in the rectangles labeled $i$ divided by $k$. Define $\psgn\ast(D\ast) = (-1)^{\area(\ga) - \ell(\ga)}\hat{s}_r(\de)$. Write $\pocon(D\ast) = \ga$.
\begin{example}
The following object $D^*$ is a labeled polywrapping of the block $11^5$ with the underlying 5-bounded polywrapping taken from Example \ref{ex:polywrapping-block}. We extract the compositions $\delta|_1 = (7,3)$, $\delta|_2 = (8,6,1)$ and $\delta|_3 = (2,1,2)$. To elaborate, we compute the first entry of $\delta|_2$ by noting that the rectangle of heights 2 labeled 1 occupies 16 squares and dividing this number by $k = 2$. The second entry of $\delta_2$ follows by dividing the total number of squares (16) occupied by the rectangles of height 2 with label 2 by $k =2$. The third entry is computed similarly.
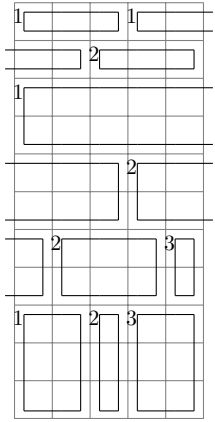
\begin{center}
\raisebox{20.0pt}{\begin{tikzpicture}[scale =0.500,baseline=(current bounding box.north)]
\draw[step = 1.0, gray, thin] (0,0) grid (5,-11);
\draw (0.250,-0.250) -- (0.250,-0.750);
\draw ( 0.1 ,-0.35) node[scale =0.750] {1};
\draw (0.250,-0.250) -- (1.25,-0.250);
\draw (0.250,-0.750) -- (1.25,-0.750);
\draw (1.25,-0.250) -- (2.25,-0.250);
\draw (1.25,-0.750) -- (2.25,-0.750);
\draw (2.25,-0.250) -- (2.75,-0.250);
\draw (2.25,-0.750) -- (2.75,-0.750);
\draw (2.75,-0.250) -- (2.75,-0.750);
\draw (3.25,-0.250) -- (3.25,-0.750);
\draw ( 3.1 ,-0.35) node[scale =0.750] {1};
\draw (3.25,-0.250) -- (4.25,-0.250);
\draw (3.25,-0.750) -- (4.25,-0.750);
\draw (4.25,-0.250) -- (5.25,-0.250);
\draw (4.25,-0.750) -- (5.25,-0.750);
\draw (-0.25,-1.25) -- (0.25,-1.25);
\draw (-0.25,-1.75) -- (0.25,-1.75);
\draw (0.250,-1.25) -- (1.25,-1.25);
\draw (0.250,-1.75) -- (1.25,-1.75);
\draw (1.25,-1.25) -- (1.75,-1.25);
\draw (1.25,-1.75) -- (1.75,-1.75);
\draw (1.75,-1.25) -- (1.75,-1.75);
\draw (2.25,-1.25) -- (2.25,-1.75);
\draw ( 2.1 ,-1.35) node[scale =0.750] {2};
\draw (2.25,-1.25) -- (3.25,-1.25);
\draw (2.25,-1.75) -- (3.25,-1.75);
\draw (3.25,-1.25) -- (4.25,-1.25);
\draw (3.25,-1.75) -- (4.25,-1.75);
\draw (4.25,-1.25) -- (4.75,-1.25);
\draw (4.25,-1.75) -- (4.75,-1.75);
\draw (4.75,-1.25) -- (4.75,-1.75);
\draw (0.250,-2.25) -- (0.250,-3.75);
\draw ( 0.1 ,-2.35) node[scale =0.750] {1};
\draw (0.250,-2.25) -- (1.25,-2.25);
\draw (0.250,-3.75) -- (1.25,-3.75);
\draw (1.25,-2.25) -- (2.25,-2.25);
\draw (1.25,-3.75) -- (2.25,-3.75);
\draw (2.25,-2.25) -- (3.25,-2.25);
\draw (2.25,-3.75) -- (3.25,-3.75);
\draw (3.25,-2.25) -- (4.25,-2.25);
\draw (3.25,-3.75) -- (4.25,-3.75);
\draw (4.25,-2.25) -- (5.25,-2.25);
\draw (4.25,-3.75) -- (5.25,-3.75);
\draw (-0.25,-4.25) -- (0.25,-4.25);
\draw (-0.25,-5.75) -- (0.25,-5.75);
\draw (0.250,-4.25) -- (1.25,-4.25);
\draw (0.250,-5.75) -- (1.25,-5.75);
\draw (1.25,-4.25) -- (2.25,-4.25);
\draw (1.25,-5.75) -- (2.25,-5.75);
\draw (2.25,-4.25) -- (2.75,-4.25);
\draw (2.25,-5.75) -- (2.75,-5.75);
\draw (2.75,-4.25) -- (2.75,-5.75);
\draw (3.25,-4.25) -- (3.25,-5.75);
\draw ( 3.1 ,-4.35) node[scale =0.750] {2};
\draw (3.25,-4.25) -- (4.25,-4.25);
\draw (3.25,-5.75) -- (4.25,-5.75);
\draw (4.25,-4.25) -- (5.25,-4.25);
\draw (4.25,-5.75) -- (5.25,-5.75);
\draw (-0.25,-6.25) -- (0.25,-6.25);
\draw (-0.25,-7.75) -- (0.25,-7.75);
\draw (0.250,-6.25) -- (0.750,-6.25);
\draw (0.250,-7.75) -- (0.750,-7.75);
\draw (0.750,-6.25) -- (0.750,-7.75);
\draw (1.25,-6.25) -- (1.25,-7.75);
\draw ( 1.1 ,-6.35) node[scale =0.750] {2};
\draw (1.25,-6.25) -- (2.25,-6.25);
\draw (1.25,-7.75) -- (2.25,-7.75);
\draw (2.25,-6.25) -- (3.25,-6.25);
\draw (2.25,-7.75) -- (3.25,-7.75);
\draw (3.25,-6.25) -- (3.75,-6.25);
\draw (3.25,-7.75) -- (3.75,-7.75);
\draw (3.75,-6.25) -- (3.75,-7.75);
\draw (4.25,-6.25) -- (4.25,-7.75);
\draw ( 4.1 ,-6.35) node[scale =0.750] {3};
\draw (4.25,-6.25) -- (4.75,-6.25);
\draw (4.25,-7.75) -- (4.75,-7.75);
\draw (4.75,-6.25) -- (4.75,-7.75);
\draw (0.250,-8.25) -- (0.250,-10.8);
\draw ( 0.1 ,-8.35) node[scale =0.750] {1};
\draw (0.250,-8.25) -- (1.25,-8.25);
\draw (0.250,-10.8) -- (1.25,-10.8);
\draw (1.25,-8.25) -- (1.75,-8.25);
\draw (1.25,-10.8) -- (1.75,-10.8);
\draw (1.75,-8.25) -- (1.75,-10.8);
\draw (2.25,-8.25) -- (2.25,-10.8);
\draw ( 2.1 ,-8.35) node[scale =0.750] {2};
\draw (2.25,-8.25) -- (2.75,-8.25);
\draw (2.25,-10.8) -- (2.75,-10.8);
\draw (2.75,-8.25) -- (2.75,-10.8);
\draw (3.25,-8.25) -- (3.25,-10.8);
\draw ( 3.1 ,-8.35) node[scale =0.750] {3};
\draw (3.25,-8.25) -- (4.25,-8.25);
\draw (3.25,-10.8) -- (4.25,-10.8);
\draw (4.25,-8.25) -- (4.75,-8.25);
\draw (4.25,-10.8) -- (4.75,-10.8);
\draw (4.75,-8.25) -- (4.75,-10.8);
\end{tikzpicture}}
\end{center}
We find $\ga = \pocon(D^*) = 1^{3,4,3}2^{8,3,3,1}3^{2,1,2}$. Also we compute $\hat{s}_5(\delta)$ to be $(-3)(1)(2) = -6$, $\area(\gamma) = 30$ and $\ell(\ga) = 10$. So $\psgn\ast(D\ast) = (-1)^{30 - 10} (-6) = -6$.
\end{example}
%Aug 31
\begin{prop}\label{prop:Hdr-in-eten}
For integers $d,r\geq 1$, 
\[
H_{d^r} = \sum_{D\ast} \psgn\ast(D\ast) e^\ox_{\pocon(D\ast)}
\]
\end{prop}
\begin{proof}
From Proposition \ref{prop:H-in-eten}, 
\[
H_{d^r} = \sum_{\gamma\in \XCom_r(dr)} \sum_{\rho \geq \gamma} (-1)^{\area(\ga)-\ell(\ga)} \hat{s}_r(\rho) e^\otimes_\gamma.
\]
As described in the proof of Proposition \ref{prop:H-dr-in-h-combo}, we can associate a polywrapping of block $d^r$ with each ordered type $\ga$. Now suppose $\rho \geq \gamma$. For each $i$, $\rho|_i \geq \gamma|_i$. For a fixed $i$, denote $\rho|_i = \al = (\al_1,\al_2,\ldots, \al_k)$ and $\gamma|_i = \be = (\be_1,\be_2, \ldots, \be_l)$. Suppose there exist integers $0 = r_0 < r_1 < \ldots < r_k$ such that $\al_i = \be_{r_{i-1}+1} + \beta_{r_{i-1}+2} + \ldots + \be_{r_i}$. Then from left to right, label the rectangles $1,2,\ldots, r_1$ by 1, rectangles $r_1 +1, r_1 + 2,\ldots, r_2$ by 2, and so on labeling the rectangles $r_{k-1} + 1, r_{k-1} + 2,\ldots, r_k$ by $k$. This results in the labeled polywrapping $D\ast$ for which $\psgn\ast(D\ast) = (-1)^{\area(\ga)-\ell(\ga)} \hat{s}_r(\rho)$ holds. The converse can be argued similarly, and gives us the claim.
\end{proof}
Define a \textit{labeled polywrapping block tabloid $T\ast$} of shape $\si = d_1^{r_1}d_2^{r_2}\ldots d_k^{r_k}$ and content $\tau$ as a list $(D\ast_{(1)}, D\ast_{(2)}, \ldots, D\ast_{(k)})$ such that each $D\ast_{(i)}$ is a labeled polywrapping of $d_i^{r_i}$ of ordered content $\ga^{(i)}$, and the multiset union $\bigcup_{i=1}^k \psort(\ocon(\de^{(i)}))$ equals $\tau$. Define $\psgn\ast(T\ast) = \prod_{i=1}^k \psgn\ast(D\ast_{(i)})$.
\begin{theorem}\label{thm:H-in-etensor}
For types $\sigma$ and $\tau$ of a positive integer $n$, the coefficient of $e^\ox_\tau$ in the $e^\ox$-expansion of $H_\si$ is $\sum_{T\ast} \psgn\ast(T\ast)$ where the sum is over all labeled polywrapping tabloids of shape $\si$ and content $\tau$.
\end{theorem}
\begin{proof}
The proof is similar to the proof of Theorem \ref{thm:H-in-htensor}
\end{proof}

\begin{prop}\label{prop:Edr-in-htensor}
For integers $d,r\geq 1$, \[
E_{d^r} = \sum_{\gamma\in \XCom_r(dr)} \left[ (-1)^{\ell(\ga)} \left(\sum_{\de\geq \gamma}\hat{s}_r(\de)  \right)h^\otimes_\gamma\right]\]
\end{prop}
\begin{proof}
The statement follows immediately by applying $\Omega$ on Proposition \ref{prop:H-in-eten} and using Corollary \ref{cor:Omega-of-htensor}.
\end{proof}
\begin{theorem}
For types $\sigma$ and $\tau$ of a positive integer $n$, the coefficient of $h^\ox_\tau$ in the $h^\ox$-expansion of $E_\si$ is $(-1)^{\area(\tau)}\sum_{T\ast} \psgn\ast(T\ast)$ where the sum is over all labeled polywrapping block tabloids of shape $\si$ and content $\tau$.
\end{theorem}
\begin{proof}
The proof is similar to that of Theorems \ref{thm:H-in-htensor} and \ref{thm:H-in-etensor}.
\end{proof}

%15 Sept 2026
\subsection{The transition matrices $\mcM_n(E^+, e^\otimes)$ and $\mcM_n(E^+,h^\ox)$}\label{ssec:Ue-Uh}
 We first find the $e$-expansion of $e_n[p_r]$ and then use it to find the $e^\ox$-expansion of $E^+_{d^r}$, and then find $\mcM_n(E^+,e^\ox)$ for $n\geq 0$. Then we use the $h^\ox$-expansion of $e^\ox_\tau$ to find the $h^\ox$-expansion of $E^+_\si$.

Recall that $\XCom_r(n)$ is the set of all $r$-ordered types $\de$ of $n$, which we recall are ordered types $\de$ for which $\area(\de|_i)$ is a multiple of $r$ for all $i\geq 1$. Also, recall $\area(\de) = \sum_{i\geq 1} \area(\de|_i)$.
\begin{prop}
For integers $d, r\geq 1$, 
\[
E^+_{d^r} = \sum_{\substack{\rho\in \XCom_r(dr)\\ \ell(\rho|_i)\leq r, \forall i}} (-1)^{\frac{(r-1)\area(\rho)}{r}} \hat{s}_r(\rho) e^\otimes_\rho.
\]
\end{prop}
\begin{proof}
From Proposition \ref{prop:dr-expansions}, $E^+_{d^r} = \sum\limits_{\la\in \Par(d)} e_{m_1(\la)}[p_r] \otimes e_{m_2(\la)}[p_r] \otimes \cdots$. To find the $e^\ox$-expansion of $E^+_{d^r}$, we need to know the $e$-expansion of $e_n[p_r]$ for all integers $n\geq 0$ and $r\geq 1$. From \cite[Ex. I.8 1(c)]{macd}, $e_n[p_r] = (-1)^{n(r-1)} \omega(h_n[p_r])$. Combining this with Corollary \ref{cor:h-expansions}, we find
\[
e_n[p_r] = \sum_{\be\in \Com(nr)} (-1)^{n(r-1)} \hat{s}_r(\be)e_\be.\]
where we may view the sum as being over all compositions $\be$ of $nr$ with length at most $r$, as for the rest, $\hat{s}_r(\be) = 0$. This yields
\begin{align*}
 E^+_{d^r} &= \sum\limits_{\la\in \Par(d)} \left(\sum_{\be^{(1)}\in \Com(m_1(\la)r)} (-1)^{m_1(\la)(r-1)}\hat{s}_r(\be^{(1)}) e_{\be^{(1)}} \right) \otimes\\ &\left(\sum_{\be^{(2)}\in \Com(m_2(\la)r))} (-1)^{m_2(\la)(r-1)}\hat{s}_r(\be^{(2)}) e_{\be^{(2)}} \right)\otimes\cdots\\ &\otimes\cdots \left(\sum_{\be^{(k)}\in \Com(m_k(\la)r)} (-1)^{m_k(\la)(r-1)}\hat{s}_r(\be^{(k)}) e_{\be^{(k)}} \right)\otimes \cdots
\end{align*}
Note that $m_i(\la) = \area(\be|_i)/r$ in the notation of the above sum. So,
\begin{align*}
E^+_{d^r} = & \sum\limits_{\la\in \Par(d)} \sum_{\be^{(1)}\in \Com(m_1(\la)r)} \sum_{\be^{(2)}\in \Com(m_2(\la)r)}\cdots\\ &\ldots \sum_{\be^{(k)}\in \Com(m_k(\la)r)}\ldots \left(\prod_{i\geq 1} (-1)^{\area(\be^{(i)})\frac{r-1}{r}}\hat{s}_r(\be^{(i)})\right) e_{\be^{(1)}}\otimes e_{\be^{(2)}}\otimes\cdots.
\end{align*}
As in the proof of Proposition \ref{prop:r-ordered-types-exp}, the above expression simplifies to a sum over $r$-ordered types $\rho$.  This yields
\[
E^+_{d^r} = \sum_{\substack{\rho\in \XCom_r(dr)\\ \ell(\rho|_i)\leq r, \forall i}} (-1)^{\frac{(r-1)\area(\rho)}{r}} \hat{s}_r(\rho) e^\otimes_\rho.
\]
\end{proof}

For a polywrapping $D$ of a block $d^r$ with ordered content $\pocon(D) = \rho$, define $\psgn_+(D) = (-1)^{\frac{(r-1)\area(\rho)}{r}}\hat{s}_r(\rho)$. 
\begin{prop}\label{prop:Edr-in-e-combinatorial}
For integers $d,r \geq 1$, 
\[
E^+_{d^r} = \sum_{D} \psgn_+(D) e^\otimes_{\pocon(D)}
\]
where the sum is over all $r$-bounded polywrappings $D$ of block $d^r$
\end{prop}
\begin{proof}
The proof is similar to that of Proposition \ref{prop:H-dr-in-h-combo}.
\end{proof}
\begin{example}
Continuing the example from Example \ref{ex:polywrapping-block}, we find $\rho = \pocon(D) = 1^{3,4,3}2^{8,3,3,1}3^{2,1,2}$ with $r = 5$. So $\frac{(r-1)\area(\rho)}{r} = \frac{4\cdot 30}{5} = 24$. So, $\psgn_+(D) = (-1)^{24}(-4) = -4$.
\end{example}
Let $\sigma = d_1^{r_1}d_2^{r_2}\ldots d_k^{r_k}$ be types of a positive integer $n$.
For a polywrapping block tabloid $D = (D^{(1)}, D^{(2)},\ldots, D^{(k)})$, define $\psgn_+(D) = \prod\limits_{i=1}^k \psgn_+(D^{(i)})$.
\begin{theorem}
For types $\si$ and $\tau$ of $n$, the coefficient of $e^\otimes_\tau$ in the $e^\ox$-expansion of $E^+_\si$ is $\sum_{T} \psgn_+(T)$, where the sum is over all polywrapping block tabloids $T$ of shape $\si$ and content $\tau$.
\end{theorem}
\begin{proof}
The proof is similar to that of Theorem \ref{thm:H-in-etensor}.
\end{proof}
\begin{example}
Continuing Example \ref{ex:pw-block-tabloid}, we compute \[\psgn_+(T) = (-1)^{\frac{3(8+4)}{4}}(1\cdot -1)(-1)^{\frac{3(4+8)}{4}}(-2\cdot 1)(-1)^{\frac{2(9)}{3}}(3) = 6.\]
\end{example}
\begin{prop}\label{prop:E+dr-in-h}
For integers $d,r\geq 1$,
\[
E^+_{d^r} = \sum_{\gamma\in \XCom_r(dr)} \left[ (-1)^{\ell(\ga)} \left( \sum_{\de\geq \ga} (-1)^{\frac{\area(\de)}{r}} \hat{s}_r(\de)\right)\right] h^\otimes_\ga.
\]
\end{prop}
\begin{proof}
The result follows by imitating the proof of Proposition \ref{prop:H-in-eten}, using Lemma \ref{lem:hten-to-eten}, and noticing that 
\[
(-1)^{\area(\ga) - \ell(\ga)} (-1)^{\frac{(r-1)\area(\ga)}{r}} = (-1)^{\ell(\ga)} (-1)^{2\area(\ga) - \frac{\area(\ga)}{r}} = (-1)^{\ell(\ga)}(-1)^{\frac{\area(\ga)}{r}}.
\]
Furthermore, $\area(\ga) = \area(\de)$ as $\delta\geq \ga$.
\end{proof}
For a labeled polywrapping $D\ast$ of a block $d^r$ of ordered content $\gamma$ such that ordered type of the labeling is $\de$, define $\psgn_+^*(D\ast) = (-1)^{\ell(\ga)} \left( \sum\limits_{\de\geq \ga} (-1)^{\frac{\area(\de)}{r}} \hat{s}_r(\de)\right)$. We write $\pocon(D\ast) = \ga$.
\begin{prop}
For integers $d\geq 0$ and $r\geq 1$,
\[
E^+_{d^r} = \sum_{D\ast} \psgn_+^*(D\ast) h^\otimes_{\ocon(D\ast)}
\]
where the sum is over all labeled polywrappings $D\ast$ of block $d^r$.
\end{prop}
\begin{proof}
The proof is similar to that of Proposition \ref{prop:Hdr-in-eten}.
\end{proof}
For a labeled polywrapping block tabloid $T\ast =(D\ast_{(1)}, D\ast_{(2)}, \ldots, D\ast_{(k)})$, define $\psgn\ast_+(T\ast) = \prod_{i=1}^k \psgn_+^*(D\ast_{(i)})$.
\begin{theorem}
Let $\si$ and $\tau$ be types of $n$. The coefficient of $h^\otimes_\tau$ in the $h^\ox$-expansion of $E^+_\si$ is $\sum_{T\ast} \psgn^*_+(T\ast)$, where the sum is over all labeled polywrapping block tabloids of shape $\si$ and content $\tau$.
\end{theorem}
\begin{proof}
The proof is similar to that of Theorem \ref{thm:H-in-etensor}.
\end{proof}
\section{Miscellanea and Open Questions}\label{sec:misc-open}

\subsection{$r$-decomposable partitions as dual special rim-hook tableaux of rectangular content}
A \textit{dual special rim-hook tableau} is a rim-hook tableau in which each label appears at least once in the top row. As in the case of SRHT, there is at most one dual SRHT of a given shape $\mu$ and content $\be$.
From \cite{turek, wildon} we know that each $r$-decomposable partition $\mu\in \Par(nr)$ can be uniquely associated with a dual special rim-hook tableau $T$ of shape $\mu$ and content $(r,r,\ldots, r)$. To see this correspondence, we start with a dual SRHT of shape $\mu$ and content $(r,r,\ldots, r)$. Notice that any $r$-ribbon with its cell in the first row cannot have a cell in a row lower than row $r$ as that would make the length of the ribbon greater than $r$. This shows that the length of $\mu$ must be at most $r$. As each ribbon has size $r$, it follows that $\mu$ is an $r$-decomposable partition. This shows that a dual SRHT of content $(r,r,\ldots, r)$ must have its shape an $r$-decomposable partition. To see the converse that an $r$-decomposable partition has a dual SRHT filling, note that when we remove an $r$-ribbon with a cell in the top row, the corresponding operation on the abacus is to slide the rightmost bead in the row with the highest index. The abacus $\ab_r(\mu)$ has exactly one bead per runner. If at each step, we choose to slide the rightmost bead in the row with the highest index up one row, then we can construct a unique dual SRHT with shape $\mu$ and content $(r,r,\ldots, r)$ by labeling each removed $r$-ribbons in descending order of labels. 
\begin{example}
The partition $\mu = (10,5,1)$ is 4-decomposable as we can successively remove 4-ribbons from the top row
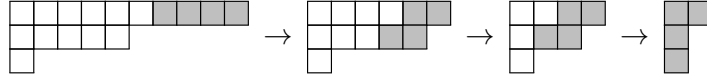
\[\y{6,5,1,}*[*(lightgray)]{6+4,5+0,1+0,}\: \to \: \y{4,3,1,}*[*(lightgray)]{4+2,3+2,1+0,} \: \to \: \y{2,1,1,}*[*(lightgray)]{2+2,1+2,1+0,} \: \to \: \y{0,0,0,}*[*(lightgray)]{0+2,0+1,0+1,}\]
\ytableausetup{nobaseline}
The corresponding dual SRHT is 
\[
\yt{1122334444,12233,1}
\]
\end{example}

\subsection{An alternate method to compute $\sgn_r(\la)$ from the abacus}
Lemma \ref{lem:rectangular-content} tells us that to compute $\sgn_r(\la)$, we may slide the beads in any order we want. For our purposes, we go through the runners from left to right and slide the bead in each runner all the way up while recording how many beads we jump over. For a list of non-negative integers $y$, we define $\inv(y)$ same as that for a permutation, and we emphasize that equal entries do not contribute to inversions. The list $y = (0,1,2,0,1)$ has 3 inversions: $(1,3)$, $(2,3)$ and $(2,4)$.
\begin{prop}\label{prop:inversionsign}
Let $\la$ be an $r$-decomposable partition and $y$ be a list of $r$ positive integers such that for $i$ satisfying $0\leq i \leq r-1$, the bead in the $i$th runner of $\ab_r(\la)$ is in the $y_i$th row from top (0-indexed). Then $\sgn_r(\la) = (-1)^{\inv(y)}$.
\end{prop}
\begin{proof}
The beads in $\ab_r(\la)$ are in positions $b_i = i + ry_i$ for $i = 0,1,\ldots, r-1$. We slide each bead in the $i$th runner in $\ab_r(\la)$ to position $i$ to obtain $\ab_r(\vn)$ in the order $i = 0,1,\ldots, r-1$. When sliding the bead in runner $i$, each bead in runner $j$ is already in position $j$ for $j = 0,1,\ldots, i-1$. So, a bead in runner $i$ when going to position $i$ can only jump over beads in runners $j > i$ which lie in a row above the bead, that is, $y_j < y_i$. Thus the total number of jumps computed over all bead slidings is equal to the number of inversions in $y$. Each jump over a bead contributes a factor of $-1$ to $\sgn_r(\la)$. The total sign $\sgn_r(\la)$ is computed as the product of signs obtained by all such inversions, and is thus equal to $(-1)^{\inv(y)}$.
\end{proof}

\begin{example}
In Example \ref{ex:r-decomp-abacus}, we have the abacus
\begin{center}\begin{tikzpicture}[shift up/.style = {transform canvas={yshift=.5mm}},shift down/.style = {transform canvas={yshift=-.5mm}},]\matrix (m) [matrix of nodes, column sep = .75cm, row sep = 0.2cm, nodes = {anchor = center}]
{0 & 1 & 2 & 3 & 4\\ 5 & 6 & 7 & 8 & 9\\ 10 & 11 & 12 & 13 & 14\\ 15 & 16 & 17 & 18 & 19\\ };
\draw (m-1-1) circle (0.3cm);
\draw (m-3-2) circle (0.3cm);
\draw (m-1-3) circle (0.3cm);
\draw (m-4-4) circle (0.3cm);
\draw (m-2-5) circle (0.3cm);
\end{tikzpicture}\end{center}
corresponding to following 5-ribbon removals
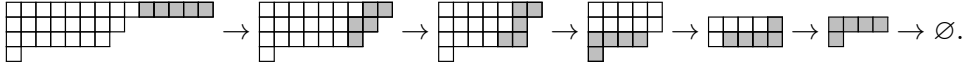
\boks{0.18}
\[
\yg{14,8,7,1}{9+5} \to \yg{9,8,7,1}{7+2,6+2,6+1} \to \yg{7,6,6,1}{5+2,5+1,4+2}\to
\yg{5,5,4,1}{0,0,4,1} \to \yg{5,5}{4+1,1+4} \to \yg{}{4,1}\to \vn.
\]
We compute $\sgn_5(\la) = (-1)^{0+2+2+1+1+1} = -1$ using the heights of the ribbons removed. The list of rows occupied in $\ab_5(\la)$ is ${y} = (0,2,0,3,1)$ which has 3 inversions: $(1,2)$, $(1,4)$ and $(3,4)$. Thus, $(-1)^{\inv(\textbf{y})} = -1$. Note that the coordinates in the inversion pairs are 0-indexed positions of entries of $y$.
\end{example}
\subsection{Sum of the entries of the input of $\hat{s}_r$ must be a multiple of $r$}
\begin{prop}\label{prop:must-be-multiple-of-r}
Let $r$ be a positive integer. If $w$ is not an $r$-weak composition of a multiple of $r$, then $\hat{s}_r(w) = 0$. Similarly, if $\mu$ is not a partition of a multiple of $r$, then  $\hat{s}_r(\mu) = 0$.
\end{prop}
\begin{proof}
Suppose the sum of entries of $w$ is $m$. The sum of entries of $\delta_r$ is $\binom{r}{2}$, and so the sum of entries of $w + \de_r$ is $m + \binom{r}{2}$. The sum of entries of $(w+\delta_r)\%r$ must be $m + \binom{r}{2}$ modulo $r$. If $(w+\delta_r)\%r$ is a permutation, or equivalently $\hat{s}_r(w) \neq 0$, then all entries of $(w+\delta_r)\%r$ are distinct and thus the sum of entries of $(w+\delta_r)\%r$ is $\binom{r}{2}$. Thus, we conclude $m$ is 0 modulo $r$, and the sum of entries of $w$ is a multiple of $r$. So by contrapositive, if the sum of entries of $w$ is not a multiple of $r$, then $(w +\de_r)\%r$ cannot be a permutation which means $\hat{s}_r(w) = 0$.

Recall that $\hat{s}_r(\mu)$ is the sum $\sum_w \hat{s}_r(w)$ over all $r$-weak compositions $w$ whose multiset of non-negative integer entries is the same as the multiset of parts of $\mu$. So when size of $\mu$ is not a multiple of $r$, then no $r$-weak composition $w$ in the index of the sum has sum of its entries a multiple of $r$. This shows that $\hat{s}_r(\mu)$ is a sum of zeros and is thus zero.
\end{proof}

\subsection{Cyclic invariance of $\hat{s}_r$}\label{ssec:cyclic-invariance}
Let $\mathbb{Z}_{\geq 0}$ be the monoid of all non-negative integers under addition. Let $\ZZ_r$ be the group $\{0,1,\ldots, r-1\}$ under addition modulo $r$ with the order $0 < 1 < \ldots < r-1$.  Define the map $F_r: \mathbb{Z}^r_{\geq 0} \to \mathbb{Z}^r_{r}$ by $F_r(w) = w + \delta_r \mod r$. For any list $w\in \mathbb{Z}_{r}^r$, define $\sgn(w) = (-1)^{\inv(w)}$ if all entries in $w$ are distinct, and zero otherwise. Thus, we can write $\hat{s}_r(w) = \sgn((w+\delta_r)\% r) = \sgn(F_r(w))$. For $w = (w_0, w_1, \ldots, w_{r-1})$, define $\rot(w) = (w_{r-1}, w_0, \ldots, w_{r-2})$ as the clockwise cyclic shift of $w$. As teased in Remark \ref{rem:multiple-of-r}, we want to show that $\hat{s}_r(\rot(w)) = \hat{s}_r(w)$. We first see how $F_r(\rot(w))$ relates to $F_r(w)$.

For two lists $v$ and $w$ with $r$ entries, we say $v\equiv_r w$, if $v_i = w_i \mod r$ for all $i = 0,1,\ldots, r-1$. Define $\vec{1}_r = (1,1, \ldots, 1)$ as a list containing $r$ 1s.
\begin{lemma}\label{lem:rot-mod-r}
    For $w\in\mathbb{Z}^r_{\geq 0} $, $F_r(\rot(w)) \equiv_r \rot(F_r(w)) + \vec{1}_r$.
\end{lemma}
\begin{proof}
    We find the $i$th terms of both $F_r(\rot(w))$ and $\rot(F_r(w)) + \vec{1}_r$, and show that they are equivalent modulo $r$. The zeroth term in $\rot(w)$ is $w_{r-1}$ and thus the zeroth term in $\rot(w) + \de_r$ is also $w_{r-1}$. So the zeroth term in $F_r(\rot(w))$ is $w_{r-1} \mod r$ while the $i$th term in $F_r(\rot(w))$ for $i = 1,2,\ldots, r-1$ is $w_{i-1} + i \mod r$.

    The $(r-1)$st term in $F_r(w)$ is $w_{r-1} + (r-1) \mod r$ which becomes the zeroth term of $\rot(F_r(w))$. If we add 1 to this term, we obtain $w_{r-1} + r - 1 + 1\mod r$ which is $w_{r-1}\mod r$. Thus the zeroth terms of $F_r(\rot(w))$ and $\rot(F_r(w)) + \vec{1}_r$ match modulo $r$.
    
  For $i\geq 1$, the $i$th term in $F_r(w)$ is $w_{i} + i \mod r$ and so the $i$th term in $\rot(F_r(w))$ is $w_{i-1} + (i-1) \mod r$. Adding one to this term gives $w_{i-1} + i\mod r$ which is the $i$th term of $F_r(\rot(w))$ as previously mentioned.
\end{proof}
In order to find $\hat{s}_r(\rot(w))$, it is sufficient for us to find how the sign changes when a list is cyclically shifted and when $\vec{1}_r$ is added to it. Recall that when $v\in \ZZ_r^r$, then $\sgn(v) = 0$ if and only if $v$ is not an $r$-permutation. We use the term permutation instead of $r$-permutation in the upcoming proof as $r$ is understood.
\begin{lemma}\label{lem:sgn-rot+1}
    Let $v\in\mathbb{Z}_r^r$.
    \begin{enumerate}
    \item If $v$ is not a permutation, or equivalently $\sgn(v) = 0$, then $\sgn(\rot(v)) = 0$ and $\sgn(v + \vec{1}) = 0$
        \item If $v$ is a permutation of $\{0,1,\ldots, r-1\}$, or equivalently $\sgn(v) \neq 0$, then
        \begin{enumerate}
            \item $\sgn(\rot(v)) = (-1)^{2k-r + 1}\sgn(v) $, where $k = v_{r-1}$, i.e., the last entry of $v$.
            \item $\sgn(v + \vec{1}_r) =(-1)^{2l - r + 1} \sgn(v)$, where $l$ is the (zero indexed) position of $r-1$ in $v$.
        \end{enumerate}
        So, $\sgn(\rot(v)) = \sgn(v + \vec{1}_r) = (-1)^{r-1}\sgn(v)$.
    \end{enumerate}
\end{lemma}
\begin{proof}
    \begin{enumerate}
        \item Suppose $v$ is not a permutation. There exist $i,j$ satisfying $0\leq i < j \leq r-1$ such that $v_i = v_j$ . For $\rot(v)$, we have $\rot(v)_{i+1} = \rot(v)_{j+1}$ where the indices are understood modulo $r$. Thus, $\rot(v)$ is also not a permutation and $\sgn(\rot(v)) = 0$. On the other hand, $(v+\vec{1}_r)_i = (v+\vec{1}_r)_j$ as $v_i + 1 = v_j + 1$. Thus, $v + \vec{1}_r$ is not a permutation and $\sgn(v + \vec{1}_r) = 0$.

        \item Now assume $v$ is a permutation and so $\sgn(v) \neq 0$. In this case $\sgn(v) = (-1)^{\inv(v)}$, so we need to see how the two operations affect the inversions.
        \begin{enumerate}
            \item If $v_{r-1} = k$, then $\rot(v)_0 = k$. For $i$ and $j$ satisfying $0\leq i < j \leq r-2$, $(i,j)$ is an inversion in $v$ if and only if $(i+1,j+1)$ is an inversion $\rot(v)$. So, the number of inversions in $\rot(v)$ differs from the number of inversions in $v$ only due to the last entry $k$ cycling to the zeroth position. For an entry $k$, there are $r-1-k$ entries greater than it and $k$ entries smaller than it. Thus, $v$ has $r-1-k$ inversions of the form $(i,r-1)$ where $i = 0,1,\ldots, r-2$, while $\rot(v)$ has $k$ inversions of the form $(0,j)$ for $j = 1,2,\ldots, r-1$. This yields the following relation 
            \[
            \frac{\sgn(\rot(v))}{\sgn(v)} = \frac{(-1)^{k}}{(-1)^{r-1-k}}
            \]
            \item Now we look at how the inversions change when we go from $v$ to $v + \vec{1}_r$. For $0\leq i < j \leq r-1$, $(i,j)$ is an inversion in $v$ if and only if $(i,j)$ is an inversion in $v + \vec{1}_r$ except when either $v_i$ or $v_j$ is equal to $r-1$. Suppose for $l$ between $0$ and $r-1$, we have $v_l = r-1$. Then $(v+\vec{1}_r)_l = 0$. There are $r- 1 - l$ entries to the right of position $l$, thus there are $r-1-l$ inversions of the form $(l,j)$ for $j = l+1, l+2, \ldots, r-1$ in $v$. On the other hand, there are $l$ inversions of the form $(i,l)$ in $v + \vec{1}_r$ for $i = 0,1,\ldots, l-1$. Thus,
\[
\frac{\sgn(v + \vec{1}_r)}{\sgn(v)} = \frac{(-1)^{r-l-1}}{(-1)^{l}}.
\] 
        \end{enumerate}
    \end{enumerate}
\end{proof}
Combining the above lemmas gives us our result.
\begin{prop}\label{prop:cyclic-invariance}
    Let $w\in \mathbb{Z}_{\geq 0}^r$, then $\hat{s}_r(\rot(w)) = \hat{s}_r(w)$.
\end{prop}
\begin{proof}
For all $w\in \ZZ_{\geq 0}^r$, $\hat{s}_r(w) = \sgn(F_r(w))$. It follows that $\hat{s}_r(\rot(w)) = \sgn(F_r(\rot(w)))$, which by Lemma \ref{lem:rot-mod-r} is equal to $\sgn(\rot(F_r(w)) + \vec{1}_r)$ where the entries of the input are considered modulo $r$.

 If $\hat{s}_r(w) = \sgn(F_r(w)) =0$, then $\sgn(\rot(F_r(w))) = 0$ by (1) of Lemma \ref{lem:sgn-rot+1} and by the same lemma we have $\sgn(\rot(F_r(w))+\vec{1}_r) = 0$, which shows $\hat{s}_r(\rot(w)) = \hat{s}_r(w) = 0$.

    Now suppose $\hat{s}_r(w)\neq 0$. By Lemma \ref{lem:sgn-rot+1}(2), \[
    \sgn(\rot(F_r(w)) + \vec{1}_r) = (-1)^{r-1}\sgn(\rot(F_r(w))
    \]
    and again by Lemma \ref{lem:sgn-rot+1}(2)
    \[
    \sgn(\rot(F_r(w)) = (-1)^{r-1}\sgn(F_r(w)).
    \]
    This shows that $\hat{s}_r(\rot(w)) = (-1)^{r-1}(-1)^{r-1}\hat{s}_r(w) = \hat{s}_r(w)$.
\end{proof}
\begin{corollary}
Let $n$ and $r$ be positive integers such that $r$ is prime. For any partition $\mu$ of $nr$ not equal to $(n,n,\ldots, n)$, the coefficient of $h_\mu$ in the $h$-expansion of $h_n[p_r]$ is divisible by $r$.
\end{corollary}
\begin{proof}
The coefficient of $h_\mu$ in the $h$-expansion of $h_n[p_r]$ is \[
\hat{s}_r(\mu)= \sum_w \hat{s}_r(w),\quad (*)
\] where the sum is over all $r$-weak compositions $w$ of $nr$ whose multiset of positive parts is the multiset of parts of $\mu$. Define $\text{Rot}(w) = \{\rot^n(w):n\geq 0\}$ as the orbit of $w$ under $\rot$. By Proposition \ref{prop:cyclic-invariance}, all $v\in \text{Rot}(w)$ satisfy $\hat{s}_r(v) = \hat{s}_r(w)$. 
As $r$ is prime, $|\text{Rot}(w)|$ is equal to $r$ for all $w\in \WCom_r(nr)$. If $\rot^\dagger$ is the transversal set of $\rot$, then $\hat{s}_r(\mu) = r\sum_{w\in \rot^\dagger} \hat{s}_r(w)$. This establishes our claim. 
\end{proof}
\begin{remark}
It may seem that the above claim might have exceptions for other rectangular partitions of the form $\mu = (a,a,\ldots, a)$ where $a$ divides $n$. But while computing $\hat{s}_r(\mu)$, we pad the partition with extra zeros which makes the orbit of $\mu$ under cyclic shift non-trivial and of size $r$.
\end{remark}

\subsection{Open Questions}
Viewing $\hat{s}_r$ as a statistic over the set of $r$-weak compositions of $n$, $\WCom_r(n)$, raises some interesting questions.
\begin{openproblem}
When the composition $\alpha$ refines $\be$, what is the relationship between $\hat{s}_r(\al)$ and $\hat{s}_r(\be)$? What is the relationship between $\hat{s}_r(\al)$ and $\sum_{\beta \geq \alpha} \hat{s}_r(\be)$?
\end{openproblem}
The answer to the latter question would simplify the expressions for $\psgn\ast(D\ast)$ for labeled polywrapping block tabloids.

\begin{openproblem}
What are the algebraic combinatorial dynamics of the statistic $\hat{s}_r$ on $\WCom_r(nr)$?
\end{openproblem}
For a partition $\mu$ of $nr$, define $R_r(\mu)$ to be the set of $w\in \WCom_r(nr)$ such that the multiset of positive parts of $w$ is the same as the multiset of parts of $\mu$. One may compute the quantity $\frac{\hat{s}_r(\mu)}{R_\mu}$ for partitions to study whether $\hat{s}_r(\mu)$ is homomesic. A statistic is \textit{homomesic} if its average value is the same over each orbit, and so in our case, we care whether $\frac{\hat{s}_r(\mu)}{R_\mu}$ is independent of the partition we choose. The statistic $\hat{s}_r$ is not homomesic over $\WCom_r(nr)$ as $\frac{\hat{s}_3((6,3,3))}{R_3((6,3,3))} = 1$ while $\frac{\hat{s}_3((6,5,1))}{R_3((6,3,3))} = \frac{-1}{2}$. We may now ask what the set of values taken by $\frac{\hat{s}_r(\mu)}{R_r(\mu)}$ is. The following table summarizes some of the data.

\begin{table}[H]
\centering
\begin{tabular}{|c|l|}
\hline
$r$ & Values attained by $\frac{\hat{s}_r(\mu)}{R_r(\mu)}$ for $\mu\in \Par$\\
\hline
2 & $\{-1 , 1 , 0\}$\\
3 & $\{-\frac{1}{2} , 1 , 0\}$\\
4 & $\{0 , -\frac{1}{3} , \frac{1}{3} , -1 , 1\}$\\[0.06cm]
5 & $\{0 , \frac{1}{6} , -\frac{1}{4} , -\frac{1}{24} , 1\}$\\[0.05cm]
6 & $\{-\frac{1}{10} , 0 , -\frac{1}{5} , \frac{1}{5} , \frac{1}{10} , -1 , 1\}$\\[0.06cm]
7 & $\{0 , -\frac{1}{6} , \frac{1}{15} , \frac{1}{36} , -\frac{1}{90} , \frac{1}{120} , 1 , -\frac{1}{20}\}$\\[0.06cm]
\hline
\end{tabular}
\end{table}

\begin{openproblem}
Is there a shorter proof of the claim in Section \ref{ssec:cyclic-invariance}?
\end{openproblem}
We prove the invariance of the statistic $\hat{s}_r: \WCom_r(nr)\to \ZZ$ under cyclic shift by analyzing the change of sign under two operations: addition of $\vec{1}_r$ and rotation of the list. The former may be viewed as a cycling of values of the list, while the latter cycles positions in the list. The effects of both these operations cancel out to give us the cyclic invariance of $\hat{s}_r$. We believe there is an alternate proof of this claim that requires less work. As the statistic arises from abaci, one could also try to understand $\hat{s}_r$ through that perspective.

\section{Example transition matrices for $n = 4$}\label{sec:matrices}
We index the rows and columns of the transition matrices with types of $n = 4$. We remove the commas in the partitions in multiplicities for brevity. We read the expansion coefficient of $G_\tau$ in the $G$-expansion of $F_\si$ as the entry in row $\tau$ and column $\si$ of $\mcM_4(F,G)$. For example,
\[
P_{3^1 1^1} = 3 h^\ox_{1^{31}} - 3h^\ox_{1^{211}} + h^\ox_{1^{1111}} + 3h^\ox_{3^1 1^1}.
\]
\begin{small}
\[
\mcM_4(P,h^\ox) = \bbmatrix{
~ & 1^4 & 1^{31} & 1^{22} & 1^{211} & 1^{1111} & 2^11^2 & 2^11^{11} & 3^1 1^1 & 2^2 & 2^{11} & 4^1\cr
1^4 & 4 & 0 & 0 & 0 & 0 & 0 & 0 & 0 & 4 & 0 & 4 \cr
1^{31} & -4 & 3 & 0 & 0 & 0 & 0 & 0 & 3 & -4 & 0 & -4 \cr
1^{22} & -2 & 0 & 4 & 0 & 0 & 4 & 0 & 0 & -2 & 4 & -2 \cr
1^{211} & 4 & -3 & -4 & 2 & 0 & -4 & 2 & -3 & 4 & -4 & 4 \cr
1^{1111} & -1 & 1 & 1 & -1 & 1 & 1 & -1 & 1 & -1 & 1 & -1 \cr
2^11^2 & 0 & 0 & 0 & 0 & 0 & 4 & 0 & 0 & 0 & 8 & 0 \cr
2^11^{11} & 0 & 0 & 0 & 0 & 0 & -2 & 2 & 0 & 0 & -4 & 0 \cr
3^11^1 & 0 & 0 & 0 & 0 & 0 & 0 & 0 & 3 & 0 & 0 & 0 \cr
2^2 & 0 & 0 & 0 & 0 & 0 & 0 & 0 & 0 & 4 & 0 & 4 \cr
2^{11} & 0 & 0 & 0 & 0 & 0 & 0 & 0 & 0 & -2 & 4 & -2 \cr
4^1 & 0 & 0 & 0 & 0 & 0 & 0 & 0 & 0 & 0 & 0 & 4 \cr
}
\]
\end{small}
\begin{small}
\[
\mcM_4(P,e^\ox) = \bbmatrix{
~ & 1^4 & 1^{31} & 1^{22} & 1^{211} & 1^{1111} & 2^11^2 & 2^11^{11} & 3^1 1^1 & 2^2 & 2^{11} & 4^1\cr
1^4 & -4 & 0 & 0 & 0 & 0 & 0 & 0 & 0 & -4 & 0 & -4 \cr
1^{31} & 4 & 3 & 0 & 0 & 0 & 0 & 0 & 3 & 4 & 0 & 4 \cr
1^{22} & 2 & 0 & 4 & 0 & 0 & 4 & 0 & 0 & 2 & 4 & 2 \cr
1^{211} & -4 & -3 & -4 & -2 & 0 & -4 & -2 & -3 & -4 & -4 & -4 \cr
1^{1111} & 1 & 1 & 1 & 1 & 1 & 1 & 1 & 1 & 1 & 1 & 1 \cr
2^11^2 & 0 & 0 & 0 & 0 & 0 & -4 & 0 & 0 & 0 & -8 & 0 \cr
2^11^{11} & 0 & 0 & 0 & 0 & 0 & 2 & 2 & 0 & 0 & 4 & 0 \cr
3^11^1 & 0 & 0 & 0 & 0 & 0 & 0 & 0 & 3 & 0 & 0 & 0 \cr
2^2 & 0 & 0 & 0 & 0 & 0 & 0 & 0 & 0 & -4 & 0 & -4 \cr
2^{11} & 0 & 0 & 0 & 0 & 0 & 0 & 0 & 0 & 2 & 4 & 2 \cr
4^1 & 0 & 0 & 0 & 0 & 0 & 0 & 0 & 0 & 0 & 0 & 4 \cr
}
\]
\end{small}
\begin{small}
\[
\mcM_4(H,h^\ox) = \bbmatrix{
~ & 1^4 & 1^{31} & 1^{22} & 1^{211} & 1^{1111} & 2^11^2 & 2^11^{11} & 3^1 1^1 & 2^2 & 2^{11} & 4^1\cr
1^4 & 4 & 0 & 0 & 0 & 0 & 0 & 0 & 0 & 2 & 0 & 1 \cr
1^{31} & -4 & 3 & 0 & 0 & 0 & 0 & 0 & 1 & -2 & 0 & 0 \cr
1^{22} & -2 & 0 & 4 & 0 & 0 & 2 & 0 & 0 & 1 & 1 & 0 \cr
1^{211} & 4 & -3 & -4 & 2 & 0 & -1 & 1 & 0 & 0 & 0 & 0 \cr
1^{1111} & -1 & 1 & 1 & -1 & 1 & 0 & 0 & 0 & 0 & 0 & 0 \cr
2^11^2 & 0 & 0 & 0 & 0 & 0 & 2 & 0 & 0 & 0 & 2 & 1 \cr
2^11^{11} & 0 & 0 & 0 & 0 & 0 & -1 & 1 & 1 & 0 & 0 & 0 \cr
3^11^1 & 0 & 0 & 0 & 0 & 0 & 0 & 0 & 1 & 0 & 0 & 1 \cr
2^2 & 0 & 0 & 0 & 0 & 0 & 0 & 0 & 0 & 2 & 0 & 1 \cr
2^{11} & 0 & 0 & 0 & 0 & 0 & 0 & 0 & 0 & -1 & 1 & 0 \cr
4^1 & 0 & 0 & 0 & 0 & 0 & 0 & 0 & 0 & 0 & 0 & 1 \cr
}
\]
\end{small}
\begin{small}
\[
\mcM_4(H,e^\ox) = \bbmatrix{
~ & 1^4 & 1^{31} & 1^{22} & 1^{211} & 1^{1111} & 2^11^2 & 2^11^{11} & 3^1 1^1 & 2^2 & 2^{11} & 4^1\cr
1^4 & -4 & 0 & 0 & 0 & 0 & 0 & 0 & 0 & -2 & 0 & -1 \cr
1^{31} & 4 & 3 & 0 & 0 & 0 & 0 & 0 & 1 & 2 & 0 & 2 \cr
1^{22} & 2 & 0 & 4 & 0 & 0 & 2 & 0 & 0 & 3 & 1 & 1 \cr
1^{211} & -4 & -3 & -4 & -2 & 0 & -3 & -1 & -2 & -4 & -2 & -3 \cr
1^{1111} & 1 & 1 & 1 & 1 & 1 & 1 & 1 & 1 & 1 & 1 & 1 \cr
2^11^2 & 0 & 0 & 0 & 0 & 0 & -2 & 0 & 0 & 0 & -2 & -1 \cr
2^11^{11} & 0 & 0 & 0 & 0 & 0 & 1 & 1 & 1 & 0 & 2 & 1 \cr
3^11^1 & 0 & 0 & 0 & 0 & 0 & 0 & 0 & 1 & 0 & 0 & 1 \cr
2^2 & 0 & 0 & 0 & 0 & 0 & 0 & 0 & 0 & -2 & 0 & -1 \cr
2^{11} & 0 & 0 & 0 & 0 & 0 & 0 & 0 & 0 & 1 & 1 & 1 \cr
4^1 & 0 & 0 & 0 & 0 & 0 & 0 & 0 & 0 & 0 & 0 & 1 \cr
}
\]
\end{small}
\begin{small}
\[
\mcM_4(E,h^\ox) = \bbmatrix{
~ & 1^4 & 1^{31} & 1^{22} & 1^{211} & 1^{1111} & 2^11^2 & 2^11^{11} & 3^1 1^1 & 2^2 & 2^{11} & 4^1\cr
1^4 & -4 & 0 & 0 & 0 & 0 & 0 & 0 & 0 & -2 & 0 & -1 \cr
1^{31} & 4 & 3 & 0 & 0 & 0 & 0 & 0 & 1 & 2 & 0 & 2 \cr
1^{22} & 2 & 0 & 4 & 0 & 0 & 2 & 0 & 0 & 3 & 1 & 1 \cr
1^{211} & -4 & -3 & -4 & -2 & 0 & -3 & -1 & -2 & -4 & -2 & -3 \cr
1^{1111} & 1 & 1 & 1 & 1 & 1 & 1 & 1 & 1 & 1 & 1 & 1 \cr
2^11^2 & 0 & 0 & 0 & 0 & 0 & 2 & 0 & 0 & 0 & 2 & 1 \cr
2^11^{11} & 0 & 0 & 0 & 0 & 0 & -1 & -1 & -1 & 0 & -2 & -1 \cr
3^11^1 & 0 & 0 & 0 & 0 & 0 & 0 & 0 & 1 & 0 & 0 & 1 \cr
2^2 & 0 & 0 & 0 & 0 & 0 & 0 & 0 & 0 & -2 & 0 & -1 \cr
2^{11} & 0 & 0 & 0 & 0 & 0 & 0 & 0 & 0 & 1 & 1 & 1 \cr
4^1 & 0 & 0 & 0 & 0 & 0 & 0 & 0 & 0 & 0 & 0 & -1 \cr
}
\]
\end{small}
\begin{small}
\[
\mcM_4(E,e^\ox) = \bbmatrix{
~ & 1^4 & 1^{31} & 1^{22} & 1^{211} & 1^{1111} & 2^11^2 & 2^11^{11} & 3^1 1^1 & 2^2 & 2^{11} & 4^1\cr
1^4 & 4 & 0 & 0 & 0 & 0 & 0 & 0 & 0 & 2 & 0 & 1 \cr
1^{31} & -4 & 3 & 0 & 0 & 0 & 0 & 0 & 1 & -2 & 0 & 0 \cr
1^{22} & -2 & 0 & 4 & 0 & 0 & 2 & 0 & 0 & 1 & 1 & 0 \cr
1^{211} & 4 & -3 & -4 & 2 & 0 & -1 & 1 & 0 & 0 & 0 & 0 \cr
1^{1111} & -1 & 1 & 1 & -1 & 1 & 0 & 0 & 0 & 0 & 0 & 0 \cr
2^11^2 & 0 & 0 & 0 & 0 & 0 & -2 & 0 & 0 & 0 & -2 & -1 \cr
2^11^{11} & 0 & 0 & 0 & 0 & 0 & 1 & -1 & -1 & 0 & 0 & 0 \cr
3^11^1 & 0 & 0 & 0 & 0 & 0 & 0 & 0 & 1 & 0 & 0 & 1 \cr
2^2 & 0 & 0 & 0 & 0 & 0 & 0 & 0 & 0 & 2 & 0 & 1 \cr
2^{11} & 0 & 0 & 0 & 0 & 0 & 0 & 0 & 0 & -1 & 1 & 0 \cr
4^1 & 0 & 0 & 0 & 0 & 0 & 0 & 0 & 0 & 0 & 0 & -1 \cr
}
\]
\end{small}
\begin{small}
\[
\mcM_4(E^+,h^\ox) = \bbmatrix{
~ & 1^4 & 1^{31} & 1^{22} & 1^{211} & 1^{1111} & 2^11^2 & 2^11^{11} & 3^1 1^1 & 2^2 & 2^{11} & 4^1\cr
1^4 & 4 & 0 & 0 & 0 & 0 & 0 & 0 & 0 & -2 & 0 & -1 \cr
1^{31} & -4 & 3 & 0 & 0 & 0 & 0 & 0 & 1 & 2 & 0 & 2 \cr
1^{22} & -2 & 0 & 4 & 0 & 0 & -2 & 0 & 0 & 3 & 1 & 1 \cr
1^{211} & 4 & -3 & -4 & 2 & 0 & 3 & -1 & -2 & -4 & -2 & -3 \cr
1^{1111} & -1 & 1 & 1 & -1 & 1 & -1 & 1 & 1 & 1 & 1 & 1 \cr
2^11^2 & 0 & 0 & 0 & 0 & 0 & 2 & 0 & 0 & 0 & -2 & -1 \cr
2^11^{11} & 0 & 0 & 0 & 0 & 0 & -1 & 1 & 1 & 0 & 2 & 1 \cr
3^11^1 & 0 & 0 & 0 & 0 & 0 & 0 & 0 & 1 & 0 & 0 & 1 \cr
2^2 & 0 & 0 & 0 & 0 & 0 & 0 & 0 & 0 & 2 & 0 & -1 \cr
2^{11} & 0 & 0 & 0 & 0 & 0 & 0 & 0 & 0 & -1 & 1 & 1 \cr
4^1 & 0 & 0 & 0 & 0 & 0 & 0 & 0 & 0 & 0 & 0 & 1 \cr
}
\]
\end{small}
\vspace{1cm}
\begin{small}
\[
\mcM_4(E^+,e^\ox) = \bbmatrix{
~ & 1^4 & 1^{31} & 1^{22} & 1^{211} & 1^{1111} & 2^11^2 & 2^11^{11} & 3^1 1^1 & 2^2 & 2^{11} & 4^1\cr
1^4 & -4 & 0 & 0 & 0 & 0 & 0 & 0 & 0 & 2 & 0 & 1 \cr
1^{31} & 4 & 3 & 0 & 0 & 0 & 0 & 0 & 1 & -2 & 0 & 0 \cr
1^{22} & 2 & 0 & 4 & 0 & 0 & -2 & 0 & 0 & 1 & 1 & 0 \cr
1^{211} & -4 & -3 & -4 & -2 & 0 & 1 & 1 & 0 & 0 & 0 & 0 \cr
1^{1111} & 1 & 1 & 1 & 1 & 1 & 0 & 0 & 0 & 0 & 0 & 0 \cr
2^11^2 & 0 & 0 & 0 & 0 & 0 & -2 & 0 & 0 & 0 & 2 & 1 \cr
2^11^{11} & 0 & 0 & 0 & 0 & 0 & 1 & 1 & 1 & 0 & 0 & 0 \cr
3^11^1 & 0 & 0 & 0 & 0 & 0 & 0 & 0 & 1 & 0 & 0 & 1 \cr
2^2 & 0 & 0 & 0 & 0 & 0 & 0 & 0 & 0 & -2 & 0 & 1 \cr
2^{11} & 0 & 0 & 0 & 0 & 0 & 0 & 0 & 0 & 1 & 1 & 0 \cr
4^1 & 0 & 0 & 0 & 0 & 0 & 0 & 0 & 0 & 0 & 0 & 1 \cr
}
\]
\end{small}

\end{document}